\documentclass[11pt,a4paper,usenames,dvipsnames]{article}

\usepackage{url,amsmath,amssymb,latexsym,mathrsfs,comment,amsthm,enumerate,geometry,calc,ifthen}
\usepackage[noadjust]{cite}
\usepackage[colorlinks]{hyperref}

\usepackage[margin=10pt,font=small,labelfont=bf, labelsep=period]{caption}

\usepackage[x11names, svgnames, rgb,table]{xcolor}
\usepackage{hhline}
\usepackage{multirow}

\usepackage{enumitem}

\usepackage[utf8]{inputenc}
\usepackage[T1]{fontenc}

\usepackage{tikz}
\usetikzlibrary{decorations.markings}
\usetikzlibrary{matrix}
\usepgflibrary{snakes,arrows,shapes}
\renewcommand{\arraystretch}{1.2}

\newcommand\nc\newcommand
\nc\rnc\renewcommand

\nc\al\alpha
\nc\be\beta
\nc\ga\gamma
\nc\de\delta
\rnc\th\theta
\nc\lam\lambda
\nc\om\omega
\nc\ze\zeta
\nc\si\sigma
\nc\Si\Sigma
\nc\Ga\Gamma
\nc\De\Delta
\nc\Th\Theta
\nc\Om\Omega
\nc\nab\nabla

\nc\AND{\qquad\text{and}\qquad}
\nc\ANd{\quad\text{and}\quad}
\nc\COMMA{,\qquad}
\nc\COMMa{,\quad}

\rnc\iff{\ \Leftrightarrow\ }
\nc\IFf{\quad \Leftrightarrow\quad }
\nc\Iff{\ \ \Leftrightarrow\ \ }
\nc\IFF{\qquad \Leftrightarrow\qquad }
\rnc\implies{\ \Rightarrow\ }
\nc\IMPLIES{\qquad \Rightarrow\qquad }

\nc\set[2]{\{#1:#2\}}
\nc\bigset[2]{\big\{#1:#2\big\}}

\nc\bit{\begin{itemize}}
\nc\eit{\end{itemize}}
\nc\ben{\begin{thmenumerate}}
\nc\bena{\begin{enumerate}[label=\textup{(\alph*)},leftmargin=10mm]}
\nc\een{\end{thmenumerate}}
\nc\eena{\end{enumerate}}

\nc\pf{\begin{proof}}
\nc\epf{\end{proof}}
\nc\pfclaim{\begin{quote}\begin{proof}}
\nc\epfclaim{\end{proof}\end{quote}}
\nc\epfres{\hfill\qed}
\nc\epfreseq{\tag*{\qed}}
\let\oldproofname=\proofname
\renewcommand{\proofname}{\rm\bf{\oldproofname}}

\renewcommand{\H}{\mathscr H}
\renewcommand{\L}{\mathscr L}
\newcommand{\R}{\mathscr R}
\newcommand{\D}{\mathscr D}
\newcommand{\J}{\mathscr J}
\newcommand{\K}{\mathscr K}
\nc\rH{\mathrel{\H}}
\nc\rL{\mathrel{\L}}
\nc\rR{\mathrel{\R}}
\nc\rD{\mathrel{\D}}
\nc\rJ{\mathrel{\J}}
\nc\rK{\mathrel{\K}}
\nc\rsi{\mathrel{\si}}

\renewcommand{\P}{\mathcal P} 
\newcommand{\Ptw}[1]{\mathcal{P}_{#1}^\Phi}
\rnc\S{\mathcal S}
\nc\A{\mathcal A}
\nc\F{\mathcal F}
\nc\I{\mathcal I}
\nc\C{\mathcal C}
\nc\LL{\mathcal L}

\newcommand{\id}{\operatorname{id}}

\newcommand{\ba}{\boldsymbol{a}}
\newcommand{\bb}{\boldsymbol{b}}
\newcommand{\bc}{\boldsymbol{c}}

\newcommand{\Float}{\operatorname{\Phi}}
\newcommand{\multw}{}
\newcommand{\multwk}[1]{\cdot}

\nc\permdiff\partial
\nc\pd\permdiff
\nc\Proj[1]{\overline{#1}}
\nc\norm[1]{\langle\!\langle#1\rangle\!\rangle}

\newcommand{\N}{\mathbb{N}}
\nc\bn{{\bf n}}
\nc\bm{{\bf m}}
\nc\bnz{{\bf n}_0}
\nc\bmz{{\bf m}_0}
\nc\bdz{{\bf d}_0}
\nc\bq{{\bf q}}
\nc\zero{{\bf 0}}

\newcommand{\coker}{\operatorname{coker}}
\newcommand{\dom}{\operatorname{dom}} 
\newcommand{\codom}{\operatorname{codom}}
\newcommand{\rank}{\operatorname{rank}}

\newcommand{\per}{\operatorname{per}}
\newcommand{\mumin}{\operatorname{\mu in}}

\newcommand{\Cong}{\text{\sf Cong}}

\newcommand{\muup}{\mu^\uparrow}
\newcommand{\mudown}{\mu^\downarrow}

\newcommand{\cg}{\text{\sf cg}}
\newcommand{\cgx}{\text{\sf cgx}}
\newcommand{\fcg}{\text{\sf cg}}
\nc\thx{\theta^{\textsf{x}}}

\nc\Pair\Pi

\newcommand{\Cmatsetup}{\arraycolsep=1.4pt\def\arraystretch{1}\footnotesize}
\makeatletter
\DeclareRobustCommand
  \vvdots{\vbox{\baselineskip4\p@ \lineskiplimit\z@\kern4\p@
    \hbox{.}\hbox{.}\hbox{.}}}
\makeatother

\newcommand{\restr}{\mathord{\restriction}}

\newcommand{\leqc}{\leq_C}

\newcommand{\geqc}{\geq_C}

\newcommand{\hgt}{{\text{\sf x}}}
\rnc\emptyset{\varnothing}
\nc\sub\subseteq
\nc\mt\mapsto
\nc\wh\widehat
\nc\normal\unlhd
\nc\sm\setminus
\nc\mr\mathrel
\nc\pc[2]{(#1,#2)^\sharp}

\nc{\uv}[1]{\fill (#1,2)circle(.17);}
\nc{\lv}[1]{\fill (#1,0)circle(.17);}
\nc{\uvs}[1]{{\foreach \x in {#1} { \uv{\x}}}}
\nc{\lvs}[1]{{\foreach \x in {#1} { \lv{\x}}}}
\nc{\darcx}[3]{\draw(#1,0)arc(180:90:#3) (#1+#3,#3)--(#2-#3,#3) (#2-#3,#3) arc(90:0:#3);}
\nc{\darc}[2]{\darcx{#1}{#2}{.4}}
\nc{\uarcx}[3]{\draw(#1,2)arc(180:270:#3) (#1+#3,2-#3)--(#2-#3,2-#3) (#2-#3,2-#3) arc(270:360:#3);}
\nc{\uarc}[2]{\uarcx{#1}{#2}{.4}}
\nc{\stline}[2]{\draw(#1,2)--(#2,0);}

\numberwithin{equation}{section}

\newtheorem{thm}[equation]{Theorem}
\newtheorem{lemma}[equation]{Lemma}
\newtheorem{cor}[equation]{Corollary}
\newtheorem{prop}[equation]{Proposition}

\theoremstyle{definition}

\newtheorem{defn}[equation]{Definition}

\newtheorem{rem}[equation]{Remark}
\newtheorem{exa}[equation]{Example}

\newenvironment{thmenumerate}{
   \begin{enumerate}[label=\textup{(\roman*)}, widest=(5), leftmargin=10mm]}{
    \end{enumerate}}
\newenvironment{thmsubenumerate}{
   \begin{enumerate}[label=\textup{(\alph*)}, widest=(b), leftmargin=10mm]}{
    \end{enumerate}}

\newcounter{caseco}
\newcommand{\case}{\refstepcounter{caseco}\medskip\noindent\textbf{Case \thecaseco:}\ } 

\newcounter{subcaseco}
\newcommand{\subcase}{\refstepcounter{subcaseco}\medskip\noindent\textbf{Subcase \thecaseco.\thesubcaseco:}\ }    

\newcounter{stepco}

\newcounter{stageco}

\newcounter{RT}\renewcommand{\theRT}{RT\arabic{RT}}
\newenvironment{rowtype}{\medskip\par\noindent\refstepcounter{RT}\textbf{Row Type \theRT}.}{\par}
\newcounter{fR}\renewcommand{\thefR}{fR\arabic{fR}}

\definecolor{delcol}{HTML}{F4F3F4}
\definecolor{excepcol}{HTML}{D50B53}
\definecolor{mucol}{HTML}{B9C406}
\definecolor{Ncol}{HTML}{FAA565}
\definecolor{Rcol}{HTML}{A882C1}

\newcounter{ncols}
\newcounter{incols}
\newenvironment{partn}[1]{
  \setcounter{ncols}{#1} \setcounter{incols}{\thencols - 1}\setlength{\arraycolsep}{1pt}
  \Bigl( \hspace{-1.5truemm}\scriptsize \renewcommand*{\arraystretch}{1}
    \begin{array}{@{\hskip 3pt} c *{\theincols}{|c} @{\hskip 3pt}  }
}{
     \end{array}
     \normalsize \hspace{-1.5truemm}\Bigr)\setlength{\arraycolsep}{6pt}
}

\begin{document}

\title{\vspace{-1cm} Classification of congruences of twisted partition monoids}

\author{James East\footnote{Centre for Research in Mathematics and Data Science, Western Sydney University, Locked Bag 1797, Penrith NSW 2751, Australia. {\it Email:} {\tt j.east\,@\,westernsydney.edu.au}} 
\ and
Nik Ru\v{s}kuc\footnote{Mathematical Institute, School of Mathematics and Statistics, University of St Andrews, St Andrews, Fife KY16 9SS, UK. {\it Email:} {\tt nik.ruskuc\,@\,st-andrews.ac.uk}}}
\date{}

\maketitle

\vspace{-1cm}
\begin{abstract}
\noindent
The twisted partition monoid $\Ptw{n}$ is an infinite monoid obtained from the classical finite partition monoid $\P_n$ by taking into account the number of floating components when multiplying partitions.
The main result of this paper is a complete description of the congruences on $\Ptw{n}$.
The succinct encoding of a congruence, which we call a C-pair, consists of a sequence of $n+1$ congruences on the additive monoid $\N$ of natural numbers and a certain $(n+1)\times\N$ matrix.
We also give a description of the inclusion ordering of congruences in terms of a lexicographic-like ordering on C-pairs.
This is then used to classify congruences on the finite $d$-twisted partition monoids~$\Ptw{n,d}$, which are obtained by factoring out from $\Ptw{n}$ the ideal of all partitions with more than~$d$ floating components.
Further applications of our results, elucidating the structure and properties of the congruence lattices of the ($d$-)twisted partition monoids, will be the subject of a future article.
\medskip

\medskip


\end{abstract}

\setcounter{tocdepth}{1}
\tableofcontents


\section{Introduction}\label{sect:intro}

The partition algebras were independently discovered in the 1990s by Vaughan Jones \cite{Jones1994_2} and Paul Martin \cite{Martin1994}.  These algebras have bases consisting of certain set partitions, which are represented and composed diagrammatically, and they naturally contain classical structures such as Brauer and Temperley-Lieb algebras, as well as symmetric group algebras \cite{TL1971,Brauer1937,Schur1927}.  These `diagram algebras' have diverse origins and applications, including in theoretical physics, classical groups, topology, invariant theory and logic \cite{MR1998,Xi1999,LZ2017,Martin2015,Wenzl1988,LZ2015,Kauffman1987,Jones1994_a,Jones1987,HR2005,GL1996,Martin1996, Abramsky2008,Kauffman1990,Jones1983_2,BH2014, BH2019,Jones1994_2,LZ2012,Martin1994,MM2013}.  The representation theory of the algebras plays a crucial role in many of the above studies, and the need to understand the kernels of representations was highlighted by Lehrer and Zhang in their article~\cite{LZ2012}, which does precisely that for Brauer's original representation of the (now-named) Brauer algebra by invariants of the orthogonal group \cite{Brauer1937}.  This has recently been extended to partition algebras by Benkart and Halverson in \cite{BH2019}.  Kernels of representations can be equivalently viewed as ideals or as congruences.  Understanding congruences is the key motivation for the current article, and indeed for the broader program of which it is a part \cite{EMRT2018,ERtwisted2,ERPX,ER2020}.

A partition algebra can be constructed as a twisted semigroup algebra of an associated (finite) partition monoid, since the product in the algebra of two partitions $\al,\be$ is always a scalar multiple of another partition, denoted $\al\be$.  The scalar is always a power of a fixed element of the underlying field, and the power to which this element is raised is the number $\Float(\al,\be)$ of `floating components' when the partitions $\al,\be$  are connected.  (Formal definitions are given below.)  It is also possible to construct partition algebras via (ordinary) semigroup algebras of twisted partition monoids.  These are countably infinite monoids whose elements are pairs $(i,\al)$, consisting of a partition $\al$ and some natural number $i$ of floating components.  The product of pairs is given by $(i,\al)(j,\be)=(i+j+\Float(\al,\be),\al\be)$.  
By incorporating the $\Float$ parameters, the twisted partition monoids reflect more of the structure of the algebras than do the ordinary partition monoids.
The above connection with semigroup algebras was formalised by Wilcox \cite{Wilcox2007}, but the idea has its origins in the work of Jones~\cite{Jones1983_2} and Kauffman~\cite{Kauffman1990}; see also~\cite{HR2005}.  Partition monoids, and other diagram monoids, have been studied by many authors, as for example in \cite{MM2007,DEG2017,DEEFHHLM2019,DEEFHHL2015,KM2006,JEgrpm,Auinger2014, ADV2012_2, Maz2002,FL2011,EMRT2018,MM2013,AMM2015}; see \cite{EG2017} for many more references.  Studies of twisted diagram monoids include \cite{DE2017,DE2018,LF2006,BDP2002,ACHLV2015,CHKLV2019,KV2019,AV2020}.

The congruences of the partition monoid $\P_n$ were determined in \cite{EMRT2018}, which also treated several other diagram monoids such as the Brauer, Jones (a.k.a.~Temperley-Lieb) and Motzkin monoids.  The article \cite{EMRT2018} also developed general machinery for constructing congruences on arbitrary monoids, which has subsequently been applied to infinite partition monoids in \cite{ERPX}, and extended to categories and their ideals in~\cite{ER2020}.  The classification of congruences on $\P_n$ is stated below in Theorem~\ref{thm:CongPn}, and the lattice $\Cong(\P_n)$ of all congruences is shown in Figure~\ref{fig:CongPn}.
It can be seen from the figure that the lattice has a rather neat structure; apart from a small prism-shaped part at the bottom, the lattice is mostly a chain.  As explained in \cite{EMRT2018}, this is a consequence of several convenient structural properties of the monoid~$\P_n$, including the following:
\bit
\item The ideals of $\P_n$ form a chain, ${I_0\subset I_1\subset\cdots\subset I_n}$.
\item The maximal subgroups of $\P_n$ are symmetric groups $\S_q$ ($q=0,1,\ldots,n$), the normal subgroups of which also form chains.
\item The minimal ideal $I_0$ is a rectangular band.
\item The second-smallest ideal $I_1$ is retractable, in the sense that there is a surmorphism $I_1\to I_0$ fixing $I_0$, and no larger ideal is retractable.
\eit
In addition to these factors, a crucial role is also played by certain technical `separation properties', which were explored in more depth in \cite{ER2020}.  Roughly speaking, these properties ensure that pairs of partitions suitably `separated' by Green's relations \cite{Green1951} generate `large' principal congruences.  

The current article concerns the twisted partition monoid $\Ptw n$, which, as explained above, is obtained from $\P_n$ by taking into account the number of floating components formed when multiplying partitions.  We also study the finite $d$-twisted quotients $\Ptw{n,d}$, which are obtained by limiting the number of floating components to at most $d$, and collapsing all other elements to zero.  The main results are the classification of the congruences of $\Ptw n$ and~$\Ptw{n,d}$, and the characterisation of the inclusion order in the lattices $\Cong(\Ptw n)$ and $\Cong(\Ptw{n,d})$.

The congruences of $\Ptw n$ are far more complicated than those of $\P_n$.  This is of course to be expected, given the additional complexity in the structure of the twisted monoid.  For example,~$\Ptw n$ has (countably) infinitely many ideals, and these do not form a chain.  Moreover, there are infinite descending chains of ideals, and there is no minimal (non-empty) ideal.  Nevertheless, the ideals still have a reasonably simple description; the principal ones are denoted~$I_{qi}$ (and defined below), indexed by integers $0\leq q\leq n$ and $i\geq0$, and we have $I_{qi}\sub I_{rj}$ if and only if $q\leq r$ and~$i\geq j$.  This allows us to view $\Ptw n$ as an $(n+1)\times\omega$ `grid', and leads to a convenient encoding of congruences by certain matrices of the same dimensions, combined with a chain $\th_0\supseteq\th_1\supseteq\cdots\supseteq\th_n$ of congruences on the additive monoid of natural numbers.  We will see that each allowable matrix-chain pair leads to either one or two distinct congruences, depending on its nature.  
The inclusion ordering on congruences involves a lexicographic-like ordering on pairs, and some additional factors.  For the finite $d$-twisted monoids $\Ptw{n,d}$, congruences are determined by the matrices alone, which are now $(n+1)\times(d+1)$.  In the very special case when $d=0$, the $0$-twisted monoid $\Ptw{n,0}$ is in fact a chain of ideals, and its congruence lattice shares some similarities with that of $\P_n$ itself, as can be observed by comparing Figures~\ref{fig:CongPn} and~\ref{fig:CongPn0}.  The case of $d\geq1$ is much more complicated, even for small $n$ and $d$; for example, the lattice $\Cong(\Ptw{3,2})$ has size $329$, and is shown in Figure~\ref{fig:CongP32}.

The article is organised as follows.  We begin in Section \ref{sec:prelim} with preliminaries on (twisted) partition monoids.  Section \ref{sec:main} contains the main result, Theorem \ref{thm:main}, which completely classifies the congruences of $\Ptw n$; a number of examples are also considered, and some simple consequences are recorded in Corollaries \ref{co:countable} and \ref{co:fi}.  The proof of Theorem \ref{thm:main} occupies the next two sections.  Section \ref{sec:I} shows that the relations stated in the theorem are indeed congruences, and Section \ref{sec:II} shows, conversely, that every congruence has one of the stated forms.  In Section~\ref{sec:inclusion} we characterise the inclusion ordering on the lattice $\Cong(\Ptw n)$; see Theorem \ref{thm:comparisons}.  We then apply the above results to the finite $d$-twisted monoids in Section \ref{sec:dtwist}.  Theorems \ref{thm:finmain} and \ref{thm:fincomp} respectively classify the congruences of $\Ptw{n,d}$ and characterise the inclusion ordering in $\Cong(\Ptw{n,d})$.  Theorem~\ref{thm:CongPn0} shows how the classification simplifies in the special case of $0$-twisted monoids $\Ptw{n,0}$.  We also discuss visualisation techniques for the (finite) lattices; see Figures \ref{fig:CongPn0}--\ref{fig:CongP14}.  Finally, Section~\ref{sec:n1} discusses the somewhat degenerate cases where $n\leq1$.

In the forthcoming article \cite{ERtwisted2}, we give a detailed analysis of the algebraic and combinatorial/order-theoretic properties of the lattices $\Cong(\Ptw n)$ and $\Cong(\Ptw{n,d})$, proving results on (bounded) generation of congruences, (co)atoms, covers, (anti-)chains, distributivity, modularity and enumeration.

\subsection*{Acknowledgements}

The first author is supported by ARC Future Fellowship FT190100632.  The second author is supported by EPSRC grant EP/S020616/1.  We thank Volodymyr Mazorchuk for his suggestion to look at congruences of the $0$-twisted partition monoids $\Ptw{n,0}$, which (eventually) led to the current paper and \cite{ERtwisted2}.


\section{Preliminaries}\label{sec:prelim}

This section contains the necessary background material.  After reviewing some basic concepts on monoids and congruences in Subsection \ref{subsec:M}, we recall the definition of the partition monoids in Subsection \ref{subsec:P} and state the classification of their congruences from \cite{EMRT2018}.  In Subsection~\ref{subsec:Ptw} we define the twisted partition monoids, and prove some basic results concerning floating components and Green's relations.  We define the finite $d$-twisted monoids in Subsection \ref{subsec:finite}, and then prove further auxilliary results in Subsection \ref{subsec:auxPnFl}.

\subsection{Monoids and congruences}\label{subsec:M}

We briefly recall some basic facts on monoids; for more background, see for example \cite{Howie1995,RS2009}.

A \emph{congruence} on a monoid $M$ is an equivalence relation $\si$ on $M$ that is 
\emph{compatible} with the product, meaning that for all $(x,y)\in\si$ and $a\in M$ we have $(ax,ay),(xa,ya)\in\si$.  We will often write $a\cdot(x,y)$ for $(ax,ay)$, with similar meanings for $(x,y)\cdot a$ and $a\cdot(x,y)\cdot b$.

The set of all congruences on the monoid $M$, denoted $\Cong(M)$, is a lattice under inclusion.  The meet of two congruences $\si,\tau\in\Cong(M)$ is their intersection, $\si\cap\tau$, while the join $\si\vee\tau$ is the transitive closure of their union.  The top and bottom elements of $\Cong(M)$ are the universal and trivial congruences:
\[
\nab_M := M\times M \AND \De_M := \bigset{(x,x)}{x\in M}.
\]
We write $\Om^\sharp$ for the congruence generated by a set of pairs $\Om\sub M\times M$.  When $\Om=\big\{(x,y)\big\}$ contains a single pair, we write $(x,y)^\sharp=\Om^\sharp$ for the \emph{principal congruence} generated by the pair.

An important family of congruences come from ideals.  A subset $I$ of $M$ is an \emph{ideal} if ${MIM\sub I}$.  It will be convenient for us to consider the empty set to be an ideal.  For $x\in M$, the \emph{principal ideal} of $M$ generated by $x$ is $MxM$.
An ideal $I$ of $M$ gives rise to the \emph{Rees congruence} 
\[
R_I := \De_M \cup \nab_I = \bigset{(x,y)\in M\times M}{x=y\text{ or }x,y\in I}.
\]
In particular, we have $R_M=\nab_M$ and $R_\emptyset=\De_M$.  

\begin{defn}\label{defn:Isi}
Let $\si$ be a congruence on a monoid $M$.  Let $\I$ be the set of all ideals $I$ of $M$ such that $R_I\sub\si$, and define $I(\si) := \bigcup_{I\in\I}I$.
\end{defn}

It is easy to see that $I(\si)$ is the largest ideal $I$ of $M$ such that $R_I\sub\si$, but note that we might have $I(\si)=\emptyset$, even if $\si$ is non-trivial.

\emph{Green's equivalences} $\R$, $\L$, $\J$, $\H$ and $\D$ on the monoid $M$ are defined as follows.  For $x,y\in M$, we have
\[
x\rR y \iff xM=yM \COMMA x\rL y \iff Mx=My \COMMA x\rJ y \iff MxM = MyM.
\]
The remaining relations are defined by $\H = \R\cap\L$ and $\D = \R\vee\L$.  In any monoid we have $\D=\R\circ\L=\L\circ\R$.  When $M$ is finite, we have $\D=\J$.  
The set $M/\J=\set{J_x}{x\in M}$ of all $\J$-classes of $M$ has a partial order $\leq$ defined, for $x,y\in M$, by
\[
J_x\leq J_y \iff x\in MyM.
\]

In all that follows, an important role will be played by the additive monoid of natural numbers, $\N=\{0,1,2,\ldots\}$.  Let us recall the simple structure of congruences on $\N$.  For every such non-trivial congruence~$\th$ 
there exist unique $m\geq0$ and $d\geq 1$, such that
\[
\th = (m,m+d)^\sharp = \De_\N \cup \bigset{ (i,j)\in\N\times\N}{i,j\geq m,\ i\equiv j\!\!\!\!\pmod{d}}.
\]
The number $m$ will be called the \emph{minimum} of $\th$ and denoted $\min\th$; the number $d$ will be called the \emph{period} of $\th$ and denoted $\per\th$.  For the universal congruence we have $\nab_\N=(0,1)^\sharp$,  $\min\nab_\N=0$ and $\per\nab_\N=1$.  For the trivial congruence it is convenient to define $\min\De_\N=\per\De_\N=\infty$.  If~$\th_1$ and $\th_2$ are congruences on $\N$, then 
\begin{equation}
\label{eq:th1th2}
\th_1\sub\th_2 \IFF \min\th_1\geq\min\th_2 \ANd \per\th_2\mid\per\th_1.
\end{equation}
Here $\mid$ is the division relation on $\N\cup\{\infty\}$, with the understanding that every element of this set divides $\infty$.

\subsection{Partition monoids}\label{subsec:P}

For $n\geq1$, we write $\bn=\{1,\dots,n\}$ and $\bnz=\bn\cup\{0\}$, and let $\bn'=\{1',\ldots,n'\}$ and $\bn''=\{1'',\ldots,n''\}$ be two disjoint copies of $\bn$.  The elements of the \emph{partition monoid} $\P_n$ are the set partitions of $\bn\cup\bn'$.  Such a partition $\al\in\P_n$ is identified with any graph on vertex set $\bn\cup\bn'$ whose connected components are the blocks of~$\al$.  When drawing such a partition, vertices from $\bn$ are drawn on an upper line, with those from~$\bn'$ directly below.  See Figure \ref{fig:P6} for some examples.

Given two partitions $\al,\be\in\P_n$, the product $\al\be$ is defined as follows.  First, let $\al_\downarrow$ be the graph on vertex set $\bn\cup\bn''$ obtained by changing every lower vertex $x'$ of $\al$ to $x''$, and let $\be^\uparrow$ be the graph on vertex set $\bn''\cup\bn'$ obtained by changing every upper vertex $x$ of $\be$ to $x''$.  The \emph{product graph} of the pair $(\al,\be)$ is the graph $\Ga(\al,\be)$ on vertex set $\bn\cup\bn''\cup\bn'$ whose edge set is the union of the edge sets of $\al_\downarrow$ and $\be^\uparrow$.  We then define $\al\be$ to be the partition of $\bn\cup\bn'$ such that vertices $x,y\in\bn\cup\bn'$ belong to the same block of $\al\be$ if and only if $x,y$ belong to the same connected component of $\Ga(\al,\be)$.  An example product is given in Figure \ref{fig:P6}.

\begin{figure}[ht]
\begin{center}
\begin{tikzpicture}[scale=.5]

\begin{scope}[shift={(0,0)}]	
\uvs{1,...,6}
\lvs{1,...,6}
\uarcx14{.6}
\uarcx23{.3}
\uarcx56{.3}
\darc12
\darcx26{.6}
\darcx45{.3}
\stline34
\draw(0.6,1)node[left]{$\alpha=$};
\draw[->](7.5,-1)--(9.5,-1);
\end{scope}

\begin{scope}[shift={(0,-4)}]	
\uvs{1,...,6}
\lvs{1,...,6}
\uarc12
\uarc34
\darc45
\darc56
\darc23
\stline31
\stline55
\draw(0.6,1)node[left]{$\beta=$};
\end{scope}

\begin{scope}[shift={(10,-1)}]	
\uvs{1,...,6}
\lvs{1,...,6}
\uarcx14{.6}
\uarcx23{.3}
\uarcx56{.3}
\darc12
\darcx26{.6}
\darcx45{.3}
\stline34
\draw[->](7.5,0)--(9.5,0);
\end{scope}

\begin{scope}[shift={(10,-3)}]	
\uvs{1,...,6}
\lvs{1,...,6}
\uarc12
\uarc34
\darc45
\darc56
\stline31
\stline55
\darc23
\end{scope}

\begin{scope}[shift={(20,-2)}]	
\uvs{1,...,6}
\lvs{1,...,6}
\uarcx14{.6}
\uarcx23{.3}
\uarcx56{.3}
\darc14
\darc45
\darc56
\stline21
\darcx23{.2}
\draw(6.4,1)node[right]{$=\alpha\beta$};
\end{scope}

\end{tikzpicture}
\caption{Multiplication of two partitions in $\P_6$.}
\label{fig:P6}
\end{center}
\end{figure}
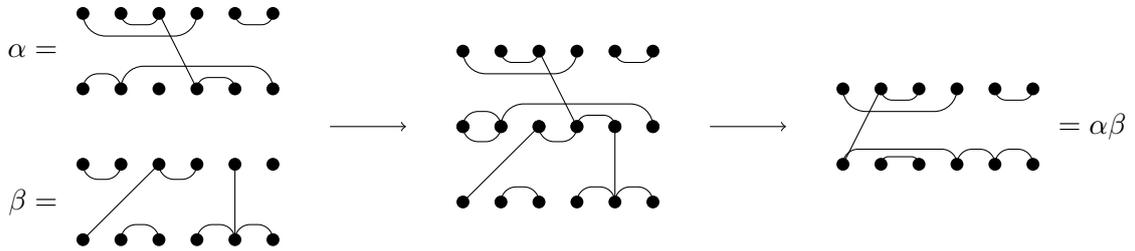

A block of a partition $\al\in\P_n$ is called a \emph{transversal} if it contains both dashed and un-dashed elements; any other block is either an \emph{upper non-transversal} (only un-dashed elements) or a \emph{lower non-transversal} (only dashed elements).  
%
%
The \emph{(co)domain} and \emph{(co)kernel} of $\al$ are defined by:
\begin{align*}
\dom\al &:= \set{x\in\bn}{x \text{ belongs to a transversal of }\al}, \\
\codom\al &:= \set{x\in\bn}{x' \text{ belongs to a transversal of }\al}, \\
\ker\al &:= \set{(x,y)\in\bn\times\bn}{x\text{ and }y \text{ belong to the same block of }\al},\\
\coker\al &:= \set{(x,y)\in\bn\times\bn}{x'\text{ and }y' \text{ belong to the same block of }\al}.
\end{align*}
The \emph{rank} of $\al$, denoted $\rank\al$, is the number of transversals of $\al$.  
We will typically use the following result without explicit reference; for proofs see \cite{Wilcox2007,FL2011}.

\begin{lemma}\label{la:Green_Pn}
For $\al,\be\in\P_n$, we have
\ben
\item $\al \rR \be \iff \dom\al=\dom\be$ and $\ker\al=\ker\be$,
\item $\al \rL \be \iff \codom\al=\codom\be$ and $\coker\al=\coker\be$,
\item $\al \rD \be \iff \al \rJ \be \iff \rank\al=\rank\be$.  
\een
The $\D=\J$-classes and non-empty ideals of $\P_n$ are the sets
\[
D_q := \set{\al\in\P_n}{\rank\al=q} \AND I_q := \set{\al\in\P_n}{\rank\al\leq q}  \qquad\text{for $q\in\bnz$,}
\]
and these are ordered by $D_q \leq D_r \iff I_q \sub I_r \iff q\leq r$.  \epfres
\end{lemma}

The above notation for the $\D$-classes and ideals of $\P_n$ will be fixed throughout the paper.

Given a partition $\al\in\P_n$, we write
\[
\al = \begin{partn}{6} A_1&\dots&A_q&C_1&\dots&C_s\\ \hhline{~|~|~|-|-|-} B_1&\dots&B_q&E_1&\dots&E_t\end{partn}
\]
to indicate that $\al$ has transversals $A_i\cup B_i'$ ($1\leq i\leq q$), upper non-transversals $C_i$ ($1\leq i\leq s$), and lower non-transversals $E_i'$ ($1\leq i\leq t$).  Here for any $A\sub\bn$ we write $A'=\set{a'}{a\in A}$, and we will also later refer to sets of the form $A''=\set{a''}{a\in A}$.  
Thus, with $\al\in\P_6$ as in Figure~\ref{fig:P6} we have $\al =  \begin{partn}{3} 2,3&1,4&5,6\\ \hhline{~|-|-} 4,5&1,2,6&3\end{partn}$.  The identity element of $\P_n$ is the partition
\[
\id := \begin{partn}{3}1&\cdots&n\\ \hhline{~|~|~} 1&\cdots&n\end{partn}.
\]

The congruences on the partition monoid $\P_n$ were determined in \cite{EMRT2018}, and the classification will play an important role in the current paper.  To state it, we first introduce some notation.  First, we have a map
\[
\P_n \to D_0 : \al=\begin{partn}{6} A_1&\dots&A_q&C_1&\dots&C_s\\ \hhline{~|~|~|-|-|-} B_1&\dots&B_q&E_1&\dots&E_t\end{partn} \mt \wh\al = \begin{partn}{6} A_1&\dots&A_q&C_1&\dots&C_s\\ \hhline{-|-|-|-|-|-} B_1&\dots&B_q&E_1&\dots&E_t\end{partn},
\]
whose effect is to break apart all transversals of $\al$ into their upper and lower parts.  Equivalently,~$\wh\al$ is the unique element of $D_0$ with the same kernel and cokernel as $\al$.  We will need the following basic result, which follows from \cite[Lemmas 3.3 and 5.2]{EMRT2018}:

\begin{lemma}\label{la:hat}
For any $\al\in I_1$ and $\eta_1,\eta_2\in\P_n$ we have $\wh{\eta_1\al\eta_2} = \eta_1\wh\al\eta_2$.  \epfres
\end{lemma}

Next we have a family of relations on $D_q$ ($2\leq q\leq n$), denoted $\nu_N$, indexed by normal subgroups $N$ of the symmetric group $\S_q$.  To define these relations consider a pair $(\al,\be)$ of $\H$-related elements from $D_q$: 
\[
\al = \begin{partn}{6} A_1&\dots&A_q&C_1&\dots&C_s\\ \hhline{~|~|~|-|-|-} B_1&\dots&B_q&E_1&\dots&E_t\end{partn}
\AND
\be = \begin{partn}{6} A_1&\dots&A_q&C_1&\dots&C_s\\ \hhline{~|~|~|-|-|-} B_{1\pi}&\dots&B_{q\pi}&E_1&\dots&E_t\end{partn} \qquad\text{for some $\pi\in\S_q$.}
\]
We then define $\pd(\al,\be)=\pi$, which we think of as the \emph{permutational difference} of $\al$ and~$\be$.  Note that~$\pd(\al,\be)$ is only well-defined up to conjugacy in $\S_q$, as $\pi$ depends on the above ordering on the transversals of $\al$ and $\be$.  Nevertheless, for any normal subgroup $N\normal\S_q$, we have a well-defined equivalence relation (see~\cite[Lemmas 3.17 and 5.6]{EMRT2018}):
\[
\nu_N = \bigset{(\al,\be)\in\H\restr_{D_q}}{\pd(\al,\be)\in N}.
\]
As extreme cases, note that $\nu_{\S_q}=\H\restr_{D_q}$ and $\nu_{\{\id_q\}}=\De_{D_q}$.  

\begin{thm}[{\cite[Theorem 5.4]{EMRT2018}}]
\label{thm:CongPn}
For $n\geq1$, the congruences on the partition monoid~$\P_n$ are precisely:
\begin{itemize}
\item
the Rees congruences $R_q:= R_{I_q} = \bigset{ (\alpha,\beta)\in\P_n\times\P_n}{ \alpha=\beta\text{ or } \rank\alpha,\rank\beta\leq q}$ for $q\in\{0,\dots,n\}$, including $\nab_{\P_n} = R_n$;
\item
the relations $R_N:=R_{q-1}\cup\nu_N$ for $q\in\{2,\dots,n\}$ and $\{\id_q\}\neq N\unlhd \S_q$;
\item
the relations 
\begin{align*}
\lam_q &:= \bigset{ (\alpha,\beta)\in I_q\times I_q}{ \widehat{\alpha}\rL\widehat{\beta}}\cup\Delta_{\P_n}, \\
\rho_q &:= \bigset{ (\alpha,\beta)\in I_q\times I_q}{ \widehat{\alpha}\rR\widehat{\beta}}\cup\Delta_{\P_n}, \\
\mu_q &:= \bigset{ (\alpha,\beta)\in I_q\times I_q}{ \widehat{\alpha}=\widehat{\beta}}\cup\Delta_{\P_n}, 
\end{align*}
for $q\in\{0,1\}$, including $\De_{\P_n} = \mu_0$, and the relations
\[
\lam_{\S_2}:=\lam_1\cup\nu_{\S_2} \COMMA \rho_{\S_2}:=\rho_1\cup\nu_{\S_2} \COMMA \mu_{\S_2}:=\mu_1\cup\nu_{\S_2}.
\]
\end{itemize}
The congruence lattice $\Cong(\P_n)$ is shown in Figure \ref{fig:CongPn}.  
 \epfres
\end{thm}

The above notation for the congruences of $\P_n$ will be fixed and used throughout the paper.

\begin{rem}
As explained in \cite{EMRT2018,ER2020}, the $\al\mt\wh\al$ mapping is largely responsible for the additional complexity in the bottom part of $\Cong(\P_n)$, as compared to the top part.  
In the twisted partition monoid, to be defined shortly, the complexity increases hugely, and this mapping remains one among the key factors.
\end{rem}

\begin{figure}[ht]

\begin{center}

\scalebox{0.8}{
\begin{tikzpicture}[scale=0.7]
\begin{scope}[minimum size=8mm,inner sep=0.5pt, outer sep=1pt]

\node (m0) at (6,0) [draw=blue,fill=white,circle,line width=2pt] { $\mu_0$};
\node (R0) at (9,4) [draw=blue,fill=white,circle,line width=2pt] { $R_0$};
\node (R1) at (6,6) [draw=blue,fill=white,circle,line width=2pt] { $R_1$};
\node (R2) at (3,10) [draw=blue,fill=white,circle,line width=2pt] { $R_2$};
\node (R3) at (3,16) [draw=blue,fill=white,circle,line width=2pt] { $R_3$};
\node (Rn) at (3,20) [draw=blue,fill=white,circle,line width=2pt] { $R_n$};

\node (r0) at (9,2) [draw,circle] { $\rho_0$};
\node (r1) at (6,4) [draw,circle] { $\rho_1$};
\node (rS2) at (3,6) [draw,circle] { $\rho_{\S_2}$};

\node (l0) at (6,2) [draw,circle] { $\lam_0$};
\node (l1) at (3,4) [draw,circle] { $\lam_1$};
\node (lS2) at (0,6) [draw,circle] { $\lam_{\S_2}$};

\node (m1) at (3,2) [draw,circle] { $\mu_1$};
\node (mS2) at (0,4) [draw,circle] { $\mu_{\S_2}$};

\node (RS2) at (3,8) [draw,circle] { $R_{\S_2}$};
\node (RA3) at (3,12) [draw,circle] { $R_{\A_3}$};
\node (RS3) at (3,14) [draw,circle] { $R_{\S_3}$};

\node (dots) at (3,18) [draw,circle, white] { $\vdots$}; \node (dotss) at (3,18.15)  { $\vdots$};

\node () at (7.6,0) { $=\De_{\P_n}$};
\node () at (4.6,20) { $=\nab_{\P_n}$};

\end{scope}

\draw 
(m0) -- (r0)
(m0) -- (l0)
(m0) -- (m1)
(m1) -- (mS2)
(m1) -- (l1)
(m1) -- (r1)
(l0) -- (l1)
(l0) -- (R0)
(r0) -- (r1)
(r0) -- (R0)
(mS2) -- (lS2)
(mS2) -- (rS2)
(l1) -- (lS2)
(l1) -- (R1)
(r1) -- (rS2)
(r1) -- (R1)
(R0) -- (R1)
(lS2) -- (RS2)
(rS2) -- (RS2)
(R1) -- (RS2)
(RS2) -- (R2)
(R2) -- (RA3)
(RA3) -- (RS3)
(RS3) -- (R3)
(R3) -- (dots)
(dots) -- (Rn)
;

\end{tikzpicture}
}
\caption{The Hasse diagram of $\Cong(\P_n)$; see Theorem \ref{thm:CongPn}.  Rees congruences are indicated in blue outline.}
\label{fig:CongPn}
\end{center}

\end{figure}
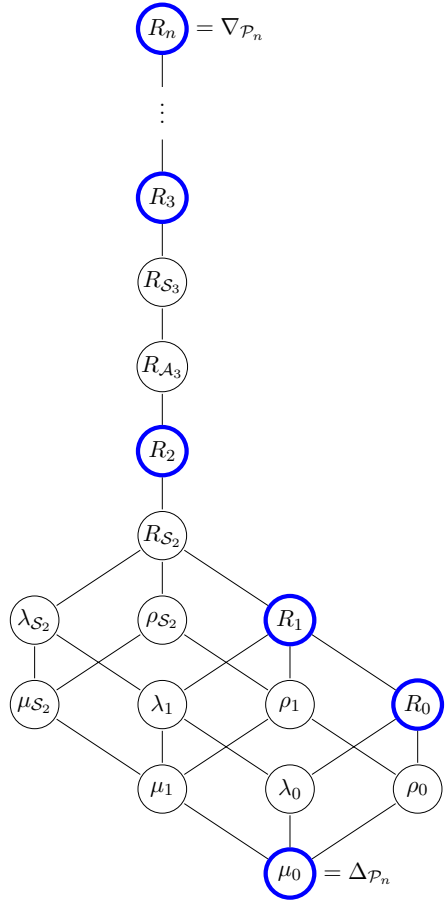

The partition monoid has an involution defined by
\[
\P_n \to \P_n : \al=\begin{partn}{6} A_1&\dots&A_q&C_1&\dots&C_s\\ \hhline{~|~|~|-|-|-} B_1&\dots&B_q&E_1&\dots&E_t\end{partn} \mt \al^* = \begin{partn}{6}B_1&\dots&B_q&E_1&\dots&E_t \\ \hhline{~|~|~|-|-|-} A_1&\dots&A_q&C_1&\dots&C_s\end{partn},
\]
satisfying $(\al\be)^*=\be^*\al^*$ and $\al=\al\al^*\al$, for all $\al,\be\in\P_n$;
so $\P_n$ is a regular $*$-semigroup in the sense of \cite{NS1978}.  Although we will not use this involution explicitly, it is responsible for a natural left-right symmetry/duality that will allow us to shorten many proofs.

\subsection{Twisted partition monoids}\label{subsec:Ptw}

Consider two partitions $\alpha,\beta\in\P_n$.  A connected component of the product graph $\Ga(\alpha,\beta)$ is said to be \emph{floating} if all its vertices come from the middle row, $\bn''$. Denote the number of floating components in $\Ga(\alpha,\beta)$ by $\Float(\alpha,\beta)$.  For example, with $\al,\be\in\P_6$ as in Figure \ref{fig:P6}, we have $\Float(\al,\be)=1$, as $\{1'',2'',6''\}$ is the unique floating component of $\Ga(\al,\be)$.  The next result is pivotal in all that follows, and will be used without explicit reference; for a proof, see \cite[Lemma~4.1]{FL2011}:

\begin{lemma}\label{la:Float}
For any $\al,\be,\ga\in\P_n$ we have $\Float(\al,\be) + \Float(\al\be,\ga) = \Float(\al,\be\ga) + \Float(\be,\ga)$. \epfres
\end{lemma}

We will write $\Float(\al,\be,\ga)$ for the common value in Lemma \ref{la:Float}, and note that this is the number of floating components created when forming the product $(\al\be)\ga=\al(\be\ga)$.

The \emph{twisted partition monoid} $\Ptw n$ is defined by
\[
\Ptw{n}:=\N\times\P_n
\qquad\text{with product}\qquad
(i,\alpha)\multw(j,\beta):= \big(i+j+\Float(\alpha,\beta),\alpha\beta\big).
\]%
The operation featuring in the first component is the addition of natural numbers, and in the second composition of partitions.  Associativity  follows from 
Lemma \ref{la:Float}.
Geometrically, one can think of $(i,\al)\in\Ptw n$ as a diagram consisting of a graph representing $\al$ along with $i$ additional floating components, as explained in \cite{ACHLV2015,BDP2002}. 
In the formation of the product $(i,\alpha)(j,\beta)$, each factor contributes its existing floating components, and a further
$\Float(\alpha,\beta)$ new ones are created.

In order to describe Green's relations on $\Ptw n$, we first need some basic lemmas.
The first describes two situations when two multiplications are guaranteed to create the same number of floating components.

\begin{lemma}\label{la:A4}
Let $\al,\be\in\P_n$.
\ben
\item If $\al\rL\be$, then $\Float(\al,\eta)=\Float(\be,\eta)$ for all $\eta\in\P_n$.
\item If $\al\rR\be$, then $\Float(\eta,\al)=\Float(\eta,\be)$ for all $\eta\in\P_n$.
\een
\end{lemma}

\pf
It suffices to prove the first statement, the second being dual.
A floating component in $\Ga(\al,\eta)$ has the form $F''=B_1''\cup\dots\cup B_k''$ for some collection $B_1',\ldots,B_k'$ of lower blocks of~$\al$, which are `brought together' by means of upper non-transversals of $\eta$.
 Since $\al\rL\be$, the $B_i'$ are also lower blocks of $\be$, and $F''$ is a floating component in $\Ga(\be,\eta)$ as well.  Thus, by symmetry, $\Ga(\al,\eta)$ and $\Ga(\be,\eta)$ have exactly the same floating components.
\epf

The next lemma will be of considerable importance throughout the paper, as it identifies situations when we can avoid creating any floating components in multiplication:

\begin{lemma}\label{la:0float}
\ben
\item \label{it:0f1} For any $\al,\be\in\P_n$, there exist $\al',\be'\in\P_n$ such that
\[
\al\be = \al'\be = \al\be' \AND \Float(\al',\be) = \Float(\al,\be') = 0.
\]
\item \label{it:0f2} For any $\al,\be,\ga\in\P_n$, there exist $\al',\ga'\in\P_n$ such that
\[
\al\be\ga = \al'\be\ga' \AND \Float(\al',\be,\ga') = 0.
\]
\een
\end{lemma}

\begin{proof}
We just prove the existence of $\al'$ in \ref{it:0f1}; the existence of $\be'$ is dual, and \ref{it:0f2} follows from~\ref{it:0f1}.  
Let the floating components in $\Ga(\al,\be)$ be $F_1'',\ldots,F_k''$, where $k=\Float(\al,\be)\geq0$.  For each $1\leq i\leq k$, we have $F_i=B_{i1}\cup\cdots\cup B_{im_k}$, where the $B_{ij}'$ are lower non-transversals of $\al$.  Fix any block $A\cup B'$ of $\al$ with $A\neq \emptyset$ (it does not matter if $B=\emptyset$).  We then take $\al'$ to be the partition obtained from $\al$ by replacing the blocks $A\cup B'$ and the $B_{ij}'$ ($1\leq i\leq k$, $1\leq j\leq m_k$) by the single block $A\cup B'\cup F_1'\cup\cdots\cup F_k'$.
\end{proof}

\begin{lemma}\label{la:Green_Ptwn}
If $\K$ is any of Green's relations, and if $\al,\be\in\P_n$ and $i,j\in\N$, then
\[
(i,\al) \rK (j,\be) \text{ in } \Ptw n \IFF i=j \ANd \al \rK \be \text{ in } \P_n.
\]
The $\D=\J$-classes and principal ideals of $\Ptw{n}$ are the sets
\[
D_{qi} := \{i\}\times D_q \AND I_{qi} := \{i,i+1,i+2,\dots\} \times I_q  \qquad\text{for $q\in\bnz$ and $i\in\N$,}
\]
and these are ordered by $D_{qi} \leq D_{rj} \iff I_{qi} \sub I_{rj} \iff q\leq r$ and $i\geq j$.
\end{lemma}

\pf
We just prove the first statement for $\K=\R$, as everything else is analogous.  Suppose first that $(i,\al) \rR (j,\be)$, so that 
\[
(i,\al) = (j,\be)\multw(k,\ga) = (j+k+\Float(\be,\ga),\be\ga ) \ANd (j,\be) = (i,\al)\multw(l,\de) = (i+l+\Float(\al,\de),\al\de) 
\]
for some $\ga,\de\in\P_n$ and $k,l\in\N$.  The second coordinates immediately give $\al\rR\be$, and the first quickly lead to $i=j$.

Conversely, suppose $i=j$ and $\al\rR\be$.  Then $\al=\be\ga$ and $\be=\al\de$ for some $\ga,\de\in\P_n$.  By Lemma \ref{la:0float} there exist $\ga',\de'\in\P_n$ such that $\al=\be\ga'$ and $\be=\al\de'$, with $\Float(\be,\ga')=\Float(\al,\de')=0$.  It then follows that $(i,\al) = (i,\be)\multw(0,\ga')$ and $(i,\be) = (i,\al)\multw(0,\de')$, so ${(i,\al) \rR (j,\be)}$.
\epf

By the previous lemma the poset $(\Ptw n/\D,\leq)$ of $\J=\D$-classes is isomorphic to the direct product $(\bnz,\leq)\times(\N,\geq)$.  
Motivated by this, we will frequently view $\Ptw{n}$ as a rectangular grid of $\D$-classes indexed by~$\bnz\times\N$, as in Figure \ref{fig:gridid}. Thus, we will refer to \emph{columns} $\{i\}\times \P_n$ ($i\in\N$) and \emph{rows} $\N\times D_q$ ($q\in\bnz$) of $\Ptw{n}$. This grid structure will feed into our description of congruences on $\Ptw{n}$, in which certain $\bnz\times\N$ matrices will play a key part.

\begin{figure}[ht]
\begin{center}
$
\begin{array}{c|*{5}{>{$}m{6mm}<{$}|}c|}\hhline{~|-|-|-|-|-|-|}
\renewcommand{\arraystretch}{2}
\text{\scriptsize 4} &\cellcolor{delcol}D_{40} &\cellcolor{delcol} D_{41}& \cellcolor{delcol} D_{42}& \cellcolor{delcol}D_{43}&\cellcolor{delcol}D_{44}&\cellcolor{delcol}\cdots\\ 
\hhline{~|-|-|-|-|-|-|}
\renewcommand{\arraystretch}{2}
\text{\scriptsize 3} &\cellcolor{delcol} D_{30}&\cellcolor{delcol}D_{31} & \cellcolor{delcol}D_{32} & \cellcolor{Rcol}D_{33}&\cellcolor{Rcol}D_{34}&\cellcolor{Rcol}\cdots\\
\hhline{~|-|-|-|-|-|-|}
\renewcommand{\arraystretch}{2}
\text{\scriptsize 2} &\cellcolor{delcol}D_{20} &\cellcolor{delcol}D_{21} & \cellcolor{delcol}D_{22} & \cellcolor{Rcol}D_{23}&\cellcolor{Rcol}D_{24}&\cellcolor{Rcol}\cdots\\
\hhline{~|-|-|-|-|-|-|}
\renewcommand{\arraystretch}{2}
\text{\scriptsize 1} &\cellcolor{delcol}D_{10} &\cellcolor{delcol}D_{11} & \cellcolor{Rcol}D_{12} & \cellcolor{Rcol}D_{13}&\cellcolor{Rcol}D_{14}&\cellcolor{Rcol}\cdots\\
\hhline{~|-|-|-|-|-|-|}
\renewcommand{\arraystretch}{2}
\text{\scriptsize 0} &\cellcolor{Rcol}D_{00} &\cellcolor{Rcol}D_{01} & \cellcolor{Rcol} D_{02}& \cellcolor{Rcol}D_{03}&\cellcolor{Rcol}D_{04}&\cellcolor{Rcol}\cdots\\
\hhline{~------}
\multicolumn{1}{c}{\text{\scriptsize $q/i$}}&\multicolumn{1}{c}{\text{\scriptsize 0}}&\multicolumn{1}{c}{\text{\scriptsize 1}}&\multicolumn{1}{c}{\text{\scriptsize 2}}&\multicolumn{1}{c}{\text{\scriptsize 3}}&\multicolumn{1}{c}{\text{\scriptsize 4}}&\multicolumn{1}{c}{\cdots}
\end{array}
$
\\
\caption{$\Ptw{4}$ as a grid, and the ideal determined by $(0,0)$, $(1,2)$ and $(3,3)$.}
\label{fig:gridid}
\end{center}
\end{figure}

\subsection{\boldmath Finite $d$-twisted partition monoids}\label{subsec:finite}

In addition to the monoid $\Ptw{n}$, we will also be interested in certain finite quotients, where we limit the number of floating components that are allowed to appear.  
Specifically, for $d\in\N$, the \emph{$d$-twisted partition monoid} is defined to be the quotient
\[
\Ptw{n,d} := \Ptw n/R_{I_{n,d+1}}
\]
by the Rees congruence associated to the (principal) ideal $I_{n,d+1} = \{d+1,d+2,\dots\} \times \P_n$.  We can also think of $\Ptw{n,d}$ as $\Ptw{n}$ with all elements with more than $d$ floating components equated to a zero element $\zero$. Thus we may take $\Ptw{n,d}$ to be the set
\[
\Ptw{n,d}:=(\bdz\times\P_n)\cup\{\zero\},
\]
with multiplication
\begin{equation}\label{eq:multwk}
\ba\multwk{d}\bb:=
\begin{cases} 
\ba\multw\bb & \text{if } \ba=(i,\alpha),\ \bb=(j,\beta) \text{ and } i+j+\Float(\alpha,\beta)\leq d,\\
\zero & \text{otherwise}. 
\end{cases}
\end{equation}
In this interpretation, $\Ptw{n,d}$ consists of columns $0,1,\ldots,d$ of $\Ptw n$, plus the zero element $\zero$.

Clearly the product in $\Ptw{n}$ of two pairs $(i,\al)$ and $(j,\be)$ will be equal to their product in
all~$\Ptw{n,d}$ for sufficiently large $d$.
So $\Ptw n$ can be regarded as a limit of $\Ptw{n,d}$ as $d\rightarrow\infty$.
 One may wonder to what extent this is reflected on the level of congruences, 
and this will be discussed in more detail in Section \ref{sec:dtwist}, and further in \cite{ERtwisted2}.

For $d=0$, the  $0$-twisted partition monoid $\Ptw{n,0}$ is (isomorphic to) $\P_n\cup\{\zero\}$ with multiplication
\begin{equation}\label{eq:multw0}
\al\multwk{0}\be:=
\begin{cases} 
\al\be & \text{if $\al,\be\in\P_n$ and $\Float(\al,\be)=0$,}\\
\zero & \text{otherwise}. 
\end{cases}
\end{equation}
These monoids are closely related to the 0-partition algebras, which are important in representation theory; see for example \cite{DW2000}.

\subsection{Auxiliary results}
\label{subsec:auxPnFl}

We now gather some preliminary results concerning the multiplication of partitions and the floating components that can arise when forming such products; these results will be used extensively throughout the paper. 

In \cite{ER2020} it was shown that underpinning the classification of congruences on $\P_n$ (Theorem~\ref{thm:CongPn}) are  certain `separation properties' of multiplication.  In the current work, we need to extend these to also include information about floating components, and the following is a suitable strengthening of \cite[Lemma~6.2]{ER2020}.

\begin{lemma}
\label{la:sep}
Suppose $\alpha\in D_q$ and $\beta\in D_r$ with $q\geq r$.
\begin{thmenumerate}
\item \label{it:sep1} If $q>r$ and $q\geq 2$, then there exists $\ga\in \P_n$ such that $\ga\al\in D_{q-1}$, $\ga\be\in I_{q-1}\sm H_{\ga\al}$ and $\Float(\ga,\al)=0$.
\item \label{it:sep2} If $q=r\geq1$ and $(\al,\be)\not\in\H$, then there exists $\ga\in \P_n$ such that, swapping $\al,\be$ if necessary,
\[
\text{$[\al\ga\in D_q$, $\be\ga\in I_{q-1}$ and $\Float(\al,\ga)=0]$ \ \ \ or \ \ \ $[\ga\al\in D_q$, $\ga\be\in I_{q-1}$ and $\Float(\ga,\al)=0]$.}
\]
\item \label{it:sep3} If $q\geq2$ and $\be\in H_\al\sm\{\al\}$, then there exists $\ga\in \P_n$ such that $\ga\al\in D_{q-1}$, ${\ga\be\in I_{q-1}\sm H_{\ga\al}}$ and $\Float(\ga,\al)=0$.  
\end{thmenumerate}
\end{lemma}

\pf
Throughout the proof, we write $\al = \begin{partn}{6} A_1&\dots&A_q&C_1&\dots&C_s\\ \hhline{~|~|~|-|-|-} B_1&\dots&B_q&E_1&\dots&E_t\end{partn}$, and we put $C=C_1\cup\dots\cup C_s$.  For each $1\leq i\leq q$, we fix some $a_i\in A_i$.  
To reduce notational clutter, we will sometimes omit the singleton blocks from our notation for partitions.

\ref{it:sep1}  If $\dom\al\not\sub\dom\be$, then we may assume without loss that $a_1\not\in\dom\be$.  If $\dom\al\sub\dom\be$, then by the pigeon-hole principle we may assume without loss that $(a_1,a_2)\in\ker\be$.  In either case, we take $\ga=\begin{partn}{4}a_1&\cdots&a_{q-2}&a_{q-1}\\ \hhline{~|~|~|~} a_1&\cdots&a_{q-2}&\{a_{q-1}\}\cup C\end{partn}$.  
Then $\dom(\ga\al)=\{a_1,\ldots,a_{q-1}\}$ and $\ker(\ga\al)$ is trivial, so that $\ga\al\in D_{q-1}$, and we have $\Float(\ga,\al)=0$.  
Note also that $\dom(\ga\be)\sub\dom\ga=\{a_1,\dots,a_{q-1}\}$.  In the $\dom\al\sub\dom\be$ case, we clearly have $(a_1,a_2)\in\ker(\ga\be)$.  In the $\dom\al\not\sub\dom\be$ case, we either have $a_1\not\in\dom(\ga\be)$ or else $(a_1,a_{q-1})\in\ker(\ga\be)$; to see this, consider the component of the product graph $\Ga(\ga,\be)$ containing $a_1$.  Thus, in both cases we have $\ga\be\in I_{q-2}\sub I_{q-1}\sm H_{\ga\al}$.

\ref{it:sep2}  We assume that $(\al,\be)\not\in\R$, the case of $(\al,\be)\not\in\L$ being dual.  So either $\dom\al\neq \dom\be$ or $\ker\al\neq \ker\be$.

\setcounter{caseco}{0}

\case   $\dom\al\neq \dom\be$.  Swapping $\al,\be$ if necessary, we may assume that $a_1\not\in\dom\be$.  We then take $\ga=\begin{partn}{4} a_1&\cdots&a_{q-1}&a_q\\ \hhline{~|~|~|~} a_1&\cdots&a_{q-1}&\{a_q\}\cup C\end{partn}$.  With similar reasoning to part \ref{it:sep1}, we have $\ga\al\in D_q$, $\ga\be\in I_{q-1}$ and $\Float(\ga,\al)=0$.

\case   $\dom\al=\dom\be$ but $\ker\al\neq \ker\be$.  Swapping $\al,\be$ if necessary, we may assume there exists $(x_1,x_2)\in\ker\be\sm\ker\al$.  Note then that $x_1$ and $x_2$ either both belong to $\dom\be=\dom\al$ or else both belong to $\bn\sm\dom\al$.

\setcounter{subcaseco}{0}

\subcase  $x_1,x_2\in\dom\al$.  Here we may assume that $x_1=a_1$ and $x_2=a_2$.  Again we take $\ga=\begin{partn}{4}a_1&\cdots&a_{q-1}&a_q\\ \hhline{~|~|~|~} a_1&\cdots&a_{q-1}&\{a_q\}\cup C\end{partn}$, and we have $\ga\al\in D_q$, $\ga\be\in I_{q-1}$ and $\Float(\ga,\al)=0$.

\subcase   $x_1,x_2\not\in\dom\al$.  We may also assume that $A_1,\ldots,A_r$ are the upper parts of the transversals of $\be$ (or otherwise we would be in the previous subcase).  Without loss we may assume that $x_1\in C_1$, and we write $E = C_2\cup\dots\cup C_s$, noting that $x_2\in E$.  This time we define $\ga=\begin{partn}{5} A_1&\dots&A_{q-1}&A_q\cup C_1&E \\ \hhline{~|~|~|~|-} A_1&\dots&A_{q-1}&C_1&A_q\cup E\end{partn}$, and we have $\ga\be\in D_q$, $\ga\al\in I_{q-1}$ and $\Float(\ga,\be)=0$.

\ref{it:sep3}  Here we have $\be = \begin{partn}{6} A_1&\dots&A_q&C_1&\dots&C_s\\ \hhline{~|~|~|-|-|-} B_{1\pi}&\dots&B_{q\pi}&E_1&\dots&E_t\end{partn}$ for some permutation $\pi\in\S_q$, and without loss we may assume that $1\pi=q$.  We then take $\ga=\begin{partn}{4}a_1&\cdots&a_{q-2}&a_{q-1}\\ \hhline{~|~|~|~} a_1&\cdots&a_{q-2}&\{a_{q-1}\}\cup C\end{partn}$, and the desired conditions are easily checked, noting that $B_q\sub\codom(\ga\be)\sm\codom(\ga\al)$, which gives $(\ga\al,\ga\be)\not\in\L$.
\epf

Note that in Lemma \ref{la:sep}\ref{it:sep3} we actually have $\Float(\ga,\al)=\Float(\ga,\be)=0$; indeed, this follows from the proof or by Lemma \ref{la:A4}.  We cannot similarly strengthen the other parts of Lemma~\ref{la:sep} in general, but the next result shows that part \ref{it:sep2} can be in certain special cases:

\begin{lemma}\label{la:A3}
\ben
\item If $\al,\be\in D_1$ and $\ker\al\neq \ker\be$, then there exists $\gamma\in\P_n$ such that $\rank(\gamma\alpha)\neq\rank(\gamma\beta)$ and ${\Float(\gamma,\alpha)=\Float(\gamma,\beta)=0}$.
\item If $\al,\be\in D_1$ and $\coker\al\neq \coker\be$, then there exists $\gamma\in\P_n$ such that $\rank(\alpha\gamma)\neq\rank(\beta\gamma)$ and ${\Float(\alpha,\gamma)=\Float(\beta,\gamma)=0}$.
\een
\end{lemma}

\begin{proof}
Only the first assertion needs to be proved, the second being dual.  We may assume without loss
that there exists $(a,b)\in\ker\alpha\setminus\ker\beta$.
Since $\rank\beta=1$, at most one of $a,b$ belongs to $\dom\beta$.
Without loss suppose $b\not\in\dom\beta$ and let $B$ be the upper block of $\beta$ containing $b$.
Then it is straightforward to check the stated conditions for
$\gamma:=\begin{partn}{2} B&\bn\setminus B\\ \hhline{~|-} B&\bn\setminus B\end{partn}$.
\end{proof}

It will turn out later on that the behaviour of congruences on $\Ptw{n}$ on rows 0 and 1 is quite different from that on other rows. One of the main technical reasons behind this is contained in the following:

\begin{lemma}
\label{la:A2}
For all $\alpha\in I_1$ and $\eta\in \P_n$ we have
\[
\rank\alpha-\rank(\alpha\eta)=\Float(\widehat{\alpha},\eta)-\Float(\alpha,\eta)
\AND
\rank\alpha-\rank(\eta\alpha)=\Float(\eta,\widehat{\alpha})-\Float(\eta,\alpha).
\]
\end{lemma}

\begin{proof}
It is sufficient to prove the first statement; the second is dual.
When $\rank\alpha=0$ then $\rank(\alpha\eta)=0$ and $\alpha=\widehat{\alpha}$, so the equality is trivial.
So suppose $\rank\alpha=1$, and let
 $A\cup B'$ be its unique transversal.
Let the connected components in $\Ga(\al,\eta)$ and $\Ga(\wh\al,\eta)$ containing $B''$ be~$U$ and $V$, respectively.  So $V\sub\bn''\cup\bn'$ and $U=A\cup V$.  Then
\[
\text{$V$ is floating in $\Ga(\wh\al,\ga)$} \iff V\sub\bn'' \iff U\sub\bn\cup\bn'' \iff \rank(\al\eta)=0.
\]
With the possible exception of $V$, the graphs $\Ga(\al,\eta)$ and $\Ga(\wh\al,\eta)$ have exactly the same floating components, and the result follows.
\end{proof}

Our final preliminary lemma concerns the relation $\nu_N$:

\begin{lemma}\label{la:AA2}
Let $N\normal\S_q$ where $2\leq q\leq n$, and let $(\al,\be)\in\H\restr_{D_q}$ and $\ga\in\P_n$.  Then
\ben
\item $\al\ga\in D_q \iff \be\ga\in D_q$, in which case $(\al,\be)\in\nu_N \iff (\al\ga,\be\ga)\in\nu_N$,
\item $\ga\al\in D_q \iff \ga\be\in D_q$, in which case $(\al,\be)\in\nu_N \iff (\ga\al,\ga\be)\in\nu_N$.  
\een

\end{lemma}

\pf
We just prove the first part, as the second is dual.  Since $\L$ is a right congruence, we have $(\al,\be)\in\H\sub\L \implies (\al\ga,\be\ga)\in\L\sub\D$, so certainly $\al\ga\in D_q \iff \be\ga\in D_q$.  For the second equivalence, the forwards implication follows immediately from the fact that $R_N=R_{I_{q-1}}\cup\nu_N$ is a congruence.  The converse follows similarly, since, by Green's Lemma \cite[Lemma 2.2.1]{Howie1995}, $\al=(\al\ga)\de$ and $\be=(\be\ga)\de$ for some~${\de\in\P_n}$.
\epf

\section{C-pairs and the statement of the main result}
\label{sec:main}

In this section we give the statement of the main result, Theorem \ref{thm:main} below, which classifies the congruences of the twisted partition monoid $\Ptw n$.  The classification involves what we will call C-pairs,
which consist of a descending chain $\th_0\supseteq\cdots\supseteq\th_n$ of congruences on the additive monoid $\N$, and a certain $\bnz\times\N$ matrix.  The precise definitions are given in Subsection \ref{subsec:Cpairs}, and the main result in Subsection~\ref{subsec:main}.  Since the definitions are somewhat technical, we will 
begin by looking at some motivating examples in 
 Subsection \ref{subsec:exa}.
En route we also discuss the projections of a congruence on $\Ptw{n}$ onto its `components' $\P_n$ and $\N$.

\subsection{Examples and projections}
\label{subsec:exa}

We begin with the simplest kind of congruences, the Rees congruences:

\begin{exa}\label{ex:Rees}
From the description of principal ideals in Lemma \ref{la:Green_Ptwn}, and the fact that every ideal is a union of principal ideals, we see that the ideals of $\Ptw{n}$ correspond to the downward closed subsets of the poset $(\bnz,\leq)\times(\N,\geq)$.
It is easy to see that in this poset there are no infinite strictly increasing sequences, or infinite antichains, and hence 
for every ideal $I$ of~$\Ptw{n}$ there exists a uniquely-determined finite collection of mutually incomparable elements
${(q_1,i_1),\dots,(q_k,i_k)\in \bnz\times\N}$  such that
\[
I = I_{q_1i_1}\cup\cdots\cup I_{q_ki_k} = \bigset{ (i,\alpha)\in \Ptw{n} }{i\geq i_t\text{ and } \rank\alpha\leq q_t\ 
(\exists t,\ 1\leq t\leq k)}.
\]
If $\Ptw{n}$ is visualised as a grid, as discussed in Subsection \ref{subsec:Ptw}, then an ideal looks like a SW--NE staircase;
see Figure \ref{fig:gridid} for an illustration.
To every ideal $I$ there corresponds the \emph{Rees congruence}
\[
R_I = \bigset{ (\ba,\bb)\in \Ptw{n}\times\Ptw{n}}{ \ba=\bb\text{ or } \ba,\bb\in I}.
\]
\end{exa}

To motivate the next family of congruences on $\Ptw{n}$, and for subsequent use, we make the following definition.

\begin{defn}[\bf The projection of a congruence]\label{defn:Proj}
Given a congruence $\si$ on $\Ptw n$, its \emph{projection} to $\P_n$ is the relation
\[
\Proj\si:=\bigset{ (\alpha,\beta)\in \P_n\times\P_n}{((i,\alpha),(j,\beta))\in\si \ (\exists i,j\in\N)}.
\]
\end{defn}

\begin{prop}
\label{prop:projcong}
The projection $\Proj\si$  of any congruence $\si\in \Cong(\Ptw{n})$ is a congruence on $\P_n$.
\end{prop}

\begin{proof}
Reflexivity and symmetry are obvious, and compatibility follows from the fact that the second components multiply as in $\P_n$. For transitivity, suppose $(\alpha,\beta),(\beta,\gamma)\in\Proj\si$, with
$((i,\alpha),(j,\beta)),((k,\beta),(l,\gamma))\in\si$.
Without loss assume that $j\leq k$. Multiplying the first pair by $(k-j,\id)$ we deduce
$((i+k-j,\alpha),(k,\beta))\in\si$.
By transitivity of $\si$ we have $((i+k-j,\alpha),(l,\gamma))\in\si$, and hence $(\alpha,\gamma)\in\Proj\si$,
as required.
\end{proof}

It turns out that every congruence on $\P_n$ arises as the projection of a congruence on $\Ptw{n}$, via the following construction.

\begin{exa}\label{ex:tau}
For any $\tau\in\Cong(\P_n)$ the relation $\bigset{ ((i,\alpha),(j,\beta))}{ (\alpha,\beta)\in \tau,\ i,j\in\N}$ is a congruence on $\Ptw{n}$, and its projection is $\tau$.
\end{exa}

One may wonder whether, analogously, the projection of a congruence of $\Ptw{n}$ onto the \emph{first} component $\N$ is also a congruence.  This turns out not to be the case in general, as the following example demonstrates.  The example also highlights some of the unusual behaviour that occurs on rows $0$ and $1$.

\begin{exa}
\label{ex:weirdproj}
Consider the relation
\[
\si := \De_{\Ptw n} \cup \bigset{ ((i,\alpha),(j,\beta))}{ i,j\in\N,\ \alpha,\beta\in I_1,\ \widehat{\alpha}=\widehat{\beta},\ \rank\alpha-\rank\beta=i-j} .
\]
It relates all pairs $\ba,\bb\in I_{10}$
whose underlying partitions satisfy $\wh\al=\wh\be$, and which belong to a single $D_{1i}$, or one of them belongs to $D_{0i}$ and the other to $D_{1,i+1}$.
We show that it is a congruence on $\Ptw{n}$.  Indeed, symmetry and reflexivity are obvious, while transitivity follows quickly upon rewriting $\rank\alpha-\rank\beta=i-j$ as $\rank\alpha-i=\rank\beta-j$.  For compatibility, let $(\ba,\bb)\in\si$ and let $\bc\in\Ptw n$ be arbitrary.  We just show that $(\ba\bc,\bb\bc)\in\si$; the proof that $(\bc\ba,\bc\bb)\in\si$ is dual.  There is nothing to show if $\ba=\bb$, so suppose $\ba=(i,\al)$ and $\bb=(j,\be)$ where $\al,\be\in I_1$, $\wh\al=\wh\be$ and $\rank\alpha-\rank\beta=i-j$.  Also write $\bc=(k,\ga)$.  Then
\[
\ba\bc = (i+k+\Float(\alpha,\gamma),\alpha\gamma) \AND \bb\bc = (j+k+\Float(\beta,\gamma),\beta\gamma).
\]
Since $I_1$ is an ideal we have $\al\ga,\be\ga\in I_1$, and Lemma \ref{la:hat} gives $\widehat{\alpha\gamma}=\widehat{\alpha}\gamma=\widehat{\beta}\gamma=\widehat{\beta\gamma}$.
Also, using Lemma \ref{la:A2}, we have:
\begin{align*}
(i+k+\Float(\alpha,\gamma))&-(j+k+\Float(\beta,\gamma))
\\
=&(i-j)+
(\rank(\alpha\gamma)-\rank\alpha+\Float(\widehat{\alpha},\gamma))-
(\rank(\beta\gamma)-\rank\beta+\Float(\widehat{\beta},\gamma))
\\
=& 
(\rank(\alpha\gamma)-\rank(\beta\gamma))+(i-j)-(\rank\alpha-\rank\beta)+
(\Float(\widehat{\alpha},\gamma)-\Float(\widehat{\beta},\gamma))
\\
=&\rank(\alpha\gamma)-\rank(\beta\gamma).
\end{align*}
So $\si$ is indeed a congruence.
However, the projection of $\si$
to $\N$ is the relation
\[
\bigset{ (i,j)\in\N\times\N}{ |i-j|\leq 1},
\]
which is not transitive.
\end{exa}

On the other hand, given a congruence on $\N$ we can always construct a congruence on $\Ptw{n}$ with that projection.

\begin{exa}
\label{exa:del}
If $\th$ is a congruence on $\N$ then the relation
\[
\si:=\bigset{ ((i,\alpha),(j,\alpha))}{\al\in\P_n,\ (i,j)\in\th}
\]
is a congruence of $\Ptw{n}$.  Indeed, $\si$ is clearly an equivalence. For right compatibility (left is dual) suppose we have
$(\ba,\bb)=((i,\alpha),(j,\alpha))\in\si$ and $\bc=(k,\beta)\in\Ptw{n}$.
Then
\[
\ba\bc=(i+k+\Float(\alpha,\beta),\alpha\beta) \AND
\bb\bc=(j+k+\Float(\alpha,\beta),\alpha\beta).
\]
Since $(i,j)\in\th$, and since $\th$ is a congruence on $\N$, it follows that $(i+k+\Float(\alpha,\beta),j+k+\Float(\alpha,\beta))\in\th$, and so $(\ba\bc,\bb\bc)\in\si$.
In the special case that $\th=\nab_\N$, the congruence constructed here is $\si=\bigset{((i,\al),(j,\al))}{\al\in\P_n,\ i,j\in\N}$, the kernel of the natural epimorphism $\Ptw{n}\rightarrow \P_n$, $(i,\alpha)\mapsto \alpha$.
\end{exa}

In fact we can obtain more congruences by further developing the idea behind Example \ref{exa:del}.

\begin{exa}
\label{exa:del2}
Suppose $\th_0\supseteq\th_1\supseteq\dots\supseteq\th_n$ is a chain of congruences on $\N$, and define
\[
\si:= \bigcup_{q\in\bnz} \bigset{ ((i,\alpha),(j,\alpha))}{ \alpha\in D_q,\ (i,j)\in\th_q}.
\]
This is a congruence, with essentially the same proof as in the previous example, and recalling additionally that
$\rank(\alpha\beta),\rank(\beta\alpha)\leq\rank\alpha$.
Note that $\Proj\si = \De_{\P_n}$ for this congruence $\si$.  In what follows, it will transpire that every congruence on $\Ptw n$ with trivial projection onto $\P_n$ is of this form.
\end{exa}

\subsection{C-pairs and congruences}\label{subsec:Cpairs}

We will encode congruences on $\Ptw{n}$ by means of certain pairs $(\Th, M)$, which we will call \emph{C-pairs}.
Here $\Th$ will be a descending chain $\th_0\supseteq \dots\supseteq \th_n$ of congruences on $\N$;
and $M=(M_{qi})_{\bnz\times\N}$ will be an infinite matrix, whose entries are drawn from the following set of symbols:
\[
\{\Delta,\muup,\mudown,\mu,\lambda,\rho, R\}\cup\set{ N}{ \{\id_q\}\neq N\unlhd \S_q,\ 2\leq q\leq n}.
\] 
We will refer to the entries in the second set collectively as the \emph{$N$-symbols}.
The entry $M_{qi}$ of~$M$ can be thought of as corresponding to the $\D$-class $D_{qi}$ of $\Ptw{n}$.
Therefore, we will think of the matrix $M$ having its first entry $M_{00}$ in the bottom left corner to correspond to our visualisation of $\Ptw{n}$ as in Figure \ref{fig:gridid}.
In the first approximation, and not entirely accurately, one can think of the symbol $M_{qi}$ as a specification for the restriction of the intended congruence to the corresponding $\D$-class.

We now describe the allowable matrices $M$, given a fixed chain ${\Th=(\th_0\supseteq\cdots\supseteq\th_n)}$.
The description will be by row, with a total of ten allowable \emph{row types}, denoted \ref{RT1}--\ref{RT10}, and with two \emph{verticality conditions} \ref{V1} and \ref{V2} governing allowable combinations of rows.
The first seven types deal simultaneously with the two bottom rows.

\begin{rowtype}
\label{RT1}
Rows $0$ and $1$ may consist of $\Delta$s only:
\[
\begin{array}{r|c|c|c|c|}\hhline{~|-|-|-|-|}
\text{\scriptsize 1} &\cellcolor{delcol}\Delta & \cellcolor{delcol}\Delta & \cellcolor{delcol}\Delta& \cellcolor{delcol}\dots\\ \hhline{~|-|-|-|-|}
\text{\scriptsize 0}&\cellcolor{delcol}\Delta & \cellcolor{delcol}\Delta & \cellcolor{delcol}\Delta& \cellcolor{delcol}\dots\\ \hhline{~|-|-|-|-|}
\end{array}
\]
\end{rowtype}

\begin{rowtype}
\label{RT2}
If $\th_0=\th_1=\Delta_\N$, rows $0$ and $1$ may be:
\[
\begin{array}{r|c|c|c|c|c|c|c|}\hhline{~|-|-|-|-|-|-|-|}
\renewcommand{\arraystretch}{2}
\text{\scriptsize 1} &\cellcolor{delcol}\Delta &\cellcolor{delcol}\dots & \cellcolor{delcol}\Delta & \cellcolor{excepcol}\zeta&\cellcolor{mucol}\mu&\cellcolor{mucol}\mu& \cellcolor{mucol}\dots\\ \hhline{~|-|-|-|-|-|-|-|}
\text{\scriptsize 0}&\cellcolor{delcol}\Delta &\cellcolor{delcol}\dots &\cellcolor{delcol}\Delta & \cellcolor{mucol}\mu&\cellcolor{mucol}\mu&\cellcolor{mucol}\mu& \cellcolor{mucol}\dots\\ \hhline{~|-|-|-|-|-|-|-|}
\multicolumn{1}{c}{}&\multicolumn{1}{c}{}&\multicolumn{1}{c}{}&\multicolumn{1}{c}{}&\multicolumn{1}{c}{\text{\scriptsize $i$}}&\multicolumn{1}{c}{}&\multicolumn{1}{c}{}&\multicolumn{1}{c}{}
\end{array}
\]
Here $i\geq 0$.  The symbol $\zeta$ can be any of $\mu$, $\muup$, $\mudown$ or $\Delta$.
\end{rowtype}

\begin{rowtype}
\label{RT3}
If $\th_0=(m,m+1)^\sharp$, rows $0$ and $1$ may be:
\[
\begin{array}{r|c|c|c|c|c|c|c|}\hhline{~|-|-|-|-|-|-|-|}
\renewcommand{\arraystretch}{2}
\text{\scriptsize 1} &\cellcolor{delcol}\Delta &\cellcolor{delcol}\dots & \cellcolor{delcol}\Delta & \cellcolor{delcol}\Delta&\cellcolor{delcol}\Delta&\cellcolor{delcol}\Delta& \cellcolor{delcol}\dots\\ \hhline{~|-|-|-|-|-|-|-|}
\text{\scriptsize 0}&\cellcolor{delcol}\Delta &\cellcolor{delcol}\dots &\cellcolor{delcol}\Delta & \cellcolor{Rcol}\xi&\cellcolor{Rcol}\xi&\cellcolor{Rcol}\xi& \cellcolor{Rcol}\dots\\ \hhline{~|-|-|-|-|-|-|-|}
\multicolumn{1}{c}{}&\multicolumn{1}{c}{}&\multicolumn{1}{c}{}&\multicolumn{1}{c}{}&\multicolumn{1}{c}{\text{\scriptsize $m$}}&\multicolumn{1}{c}{}&\multicolumn{1}{c}{}&\multicolumn{1}{c}{}
\end{array}
\]
The symbol $\xi$ can be any of $\rho$, $\lambda$ or $R$.
\end{rowtype}

\begin{rowtype}
\label{RT4}
If $\th_0=\th_1=(m,m+d)^\sharp$, rows $0$ and $1$ may be:
\[
\begin{array}{r|c|c|c|c|c|c|c|}\hhline{~|-|-|-|-|-|-|}
\renewcommand{\arraystretch}{2}
\text{\scriptsize 1} &\cellcolor{delcol}\Delta &\cellcolor{delcol}\dots & \cellcolor{delcol}\Delta & \cellcolor{Rcol}\xi&\cellcolor{Rcol}\xi& \cellcolor{Rcol}\dots\\ \hhline{~|-|-|-|-|-|-|}
\text{\scriptsize 0}&\cellcolor{delcol}\Delta &\cellcolor{delcol}\dots &\cellcolor{delcol}\Delta & \cellcolor{Rcol}\xi&\cellcolor{Rcol}\xi& \cellcolor{Rcol}\dots\\ \hhline{~|-|-|-|-|-|-|}
\multicolumn{1}{c}{}&\multicolumn{1}{c}{}&\multicolumn{1}{c}{}&\multicolumn{1}{c}{}&\multicolumn{1}{c}{\text{\scriptsize $m$}}&\multicolumn{1}{c}{}&\multicolumn{1}{c}{}
\end{array}
\]
If $d=1$ the symbol $\xi$ can be any of
$\mu$, $\rho$, $\lambda$ or $R$; if $d>1$ then $\xi=\mu$.
\end{rowtype}

\begin{rowtype}
\label{RT5}
If $\th_0=(m,m+d)^\sharp$ and $\th_1=(m+1,m+1+d)^\sharp$, rows $0$ and $1$ may be:
\[
\begin{array}{r|c|c|c|c|c|c|c|c|c|c|c|c|}\hhline{~|-|-|-|-|-|-|-|-|-|-|-|}
\renewcommand{\arraystretch}{2}
\text{\scriptsize 1} &\cellcolor{delcol}\Delta &\cellcolor{delcol}\dots & \cellcolor{delcol}\Delta & 
\cellcolor{excepcol}\zeta & \cellcolor{mucol} \mu & \cellcolor{mucol} \dots &\cellcolor{mucol} \mu &\cellcolor{mucol} \mu &
\cellcolor{Rcol}\xi&\cellcolor{Rcol}\xi& \cellcolor{Rcol}\dots\\ \hhline{~|-|-|-|-|-|-|-|-|-|-|-|}
\text{\scriptsize 0}&\cellcolor{delcol}\Delta &\cellcolor{delcol}\dots &\cellcolor{delcol}\Delta & 
\cellcolor{mucol}\mu & \cellcolor{mucol} \mu & \cellcolor{mucol} \dots &\cellcolor{mucol} \mu &
\cellcolor{Rcol} \xi &\cellcolor{Rcol}\xi&\cellcolor{Rcol}\xi& \cellcolor{Rcol}\dots\\ \hhline{~|-|-|-|-|-|-|-|-|-|-|-|}
\multicolumn{1}{c}{}&\multicolumn{1}{c}{}&\multicolumn{1}{c}{}&\multicolumn{1}{c}{}&\multicolumn{1}{c}{\text{\scriptsize $i$}}&\multicolumn{1}{c}{}&\multicolumn{1}{c}{}&\multicolumn{1}{c}{}&\multicolumn{1}{c}{\text{\scriptsize $m$}}&\multicolumn{1}{c}{}&\multicolumn{1}{c}{}
\end{array}
\]
Here $0\leq i< m$.  If $d=1$ the symbol $\xi$ can be any of $\mu$, $\rho$, $\lambda$ or $R$; if $d>1$ then $\xi=\mu$.  The symbol $\zeta$ can be any of $\mu$, $\muup$, $\mudown$ or $\Delta$.
\end{rowtype}

\begin{rowtype}
\label{RT6}
If $\th_0=(m,m+d)^\sharp$ and $\th_1=(l,l+d)^\sharp$
with $l>m$, rows $0$ and $1$ may be:
\[
\begin{array}{r|c|c|c|c|c|c|c|c|c|}\hhline{~|-|-|-|-|-|-|-|-|-|}
\renewcommand{\arraystretch}{2}
\text{\scriptsize 1} &\cellcolor{delcol}\Delta &\cellcolor{delcol}\dots & \cellcolor{delcol}\Delta & 
\cellcolor{delcol}\Delta & \cellcolor{delcol}\dots & \cellcolor{delcol}\Delta &
\cellcolor{excepcol} \zeta & \cellcolor{Rcol}\xi& \cellcolor{Rcol}\dots
\\ \hhline{~|-|-|-|-|-|-|-|-|-|}
\text{\scriptsize 0}&\cellcolor{delcol}\Delta &\cellcolor{delcol}\dots &\cellcolor{delcol}\Delta & 
\cellcolor{Rcol} \xi &\cellcolor{Rcol}\dots&\cellcolor{Rcol}\xi&\cellcolor{Rcol}\xi&\cellcolor{Rcol}\xi& \cellcolor{Rcol}\dots
\\ \hhline{~|-|-|-|-|-|-|-|-|-|}
\multicolumn{1}{c}{}&\multicolumn{1}{c}{}&\multicolumn{1}{c}{}&\multicolumn{1}{c}{}&\multicolumn{1}{c}{\text{\scriptsize $m$}}&\multicolumn{1}{c}{}&\multicolumn{1}{c}{}&\multicolumn{1}{c}{}&\multicolumn{1}{c}{\text{\scriptsize $l$}}&\multicolumn{1}{c}{}
\end{array}
\]
If $d=1$ the symbol $\xi$ can be any of
$\mu$, $\rho$, $\lambda$ or $R$; if $d>1$ then $\xi=\mu$.
The symbol $\zeta$ can be any of
$\mu$, $\muup$, $\mudown$ or $\Delta$.
\end{rowtype}

\begin{rowtype}
\label{RT7}
If $\th_0=(m,m+d)^\sharp$ and $\th_1=(l,l+d)^\sharp$
with $l-1>m>0$ and ${l-1\equiv m\pmod{d}}$, rows $0$ and $1$ may be:
\[
\begin{array}{r|c|c|c|c|c|c|c|c|c|c|}\hhline{~|-|-|-|-|-|-|-|-|-|-|}
\renewcommand{\arraystretch}{2}
\text{\scriptsize 1} &\cellcolor{delcol}\Delta &\cellcolor{delcol}\dots & \cellcolor{delcol}\Delta  & \cellcolor{delcol}\Delta &
\cellcolor{delcol}\Delta & \cellcolor{delcol}\dots  & \cellcolor{delcol}\Delta &
\cellcolor{mucol} \mu & \cellcolor{Rcol}\xi& \cellcolor{Rcol}\dots
\\ \hhline{~|-|-|-|-|-|-|-|-|-|-|}
\text{\scriptsize 0}&\cellcolor{delcol}\Delta &\cellcolor{delcol}\dots &\cellcolor{delcol}\Delta  
&\cellcolor{mucol}\mu&
\cellcolor{Rcol} \xi &\cellcolor{Rcol}\dots&\cellcolor{Rcol}\xi&\cellcolor{Rcol}\xi&\cellcolor{Rcol}\xi& \cellcolor{Rcol}\dots
\\ \hhline{~|-|-|-|-|-|-|-|-|-|-|}
\multicolumn{1}{c}{}&\multicolumn{1}{c}{}&\multicolumn{1}{c}{}&\multicolumn{1}{c}{}&\multicolumn{1}{c}{}&\multicolumn{1}{c}{\text{\scriptsize $m$}}&\multicolumn{1}{c}{}&\multicolumn{1}{c}{}&\multicolumn{1}{c}{}&\multicolumn{1}{c}{\text{\scriptsize $l$}}
\end{array}
\]
If $d=1$ the symbol $\xi$ can be any of
$\mu$, $\rho$, $\lambda$ or $R$; if $d>1$ then $\xi=\mu$.
\end{rowtype}

\bigskip

In the above, note that $\per\th_0=\per\th_1$ in almost all cases, the possible exceptions being only in types \ref{RT1} and \ref{RT3}.  Also note that the only symbols that can appear before $\min\th_0$ or $\min\th_1$ are $\De$, $\mu$, $\muup$ and $\mudown$; the only entries that can appear after (or at) $\min\th_0$ or $\min\th_1$ are $\De$,~$\mu$,~$\lam$,~$\rho$ or $R$.

The remaining three types \ref{RT8}--\ref{RT10} specify an arbitrary row $q$ with $q\geq 2$.

\begin{rowtype}
\label{RT8}
Row $q\geq2$ may consist of $\Delta$s only:
\[
\begin{array}{r|c|c|c|c|}\hhline{~|-|-|-|-|}
\text{\scriptsize $q$} &\cellcolor{delcol}\Delta & \cellcolor{delcol}\Delta & \cellcolor{delcol}\Delta& \cellcolor{delcol}\dots\\ \hhline{~|-|-|-|-|}
\end{array}
\]
\end{rowtype}

\begin{rowtype}
\label{RT9}
Row $q\geq 2$ may be:
\[
\begin{array}{r|c|c|c|c|c|c|c|c|c|c|}\hhline{~|-|-|-|-|-|-|-|-|-|-|}
\renewcommand{\arraystretch}{2}
\text{\scriptsize $q$} &\cellcolor{delcol}\Delta &\cellcolor{delcol}\dots & \cellcolor{delcol}\Delta  & 
\cellcolor{Ncol} N_i &\cellcolor{Ncol} N_{i+1} & \cellcolor{Ncol}\dots & \cellcolor{Ncol} N_{k-1} &
\cellcolor{Ncol} N & \cellcolor{Ncol} N& \cellcolor{Ncol}\dots
\\ \hhline{~|-|-|-|-|-|-|-|-|-|-|}
\multicolumn{1}{c}{}&\multicolumn{1}{c}{}&\multicolumn{1}{c}{}&\multicolumn{1}{c}{}&\multicolumn{1}{c}{\text{\scriptsize $i$}}&\multicolumn{1}{c}{}&\multicolumn{1}{c}{}&\multicolumn{1}{c}{}&\multicolumn{1}{c}{\text{\scriptsize $k$}}&\multicolumn{1}{c}{}&\multicolumn{1}{c}{}
\end{array}
\]
Here $0\leq i \leq k \leq \min\theta_q$,   and $\{\id_q\}\not=N_i\leq \dots \leq N_{k-1}\leq N$ are non-trivial normal subgroups of $\S_q$.
\end{rowtype}

\begin{rowtype}
\label{RT10}
If $q\geq 2$ and $\th_q=(m,m+1)^\sharp$, row $q$ may be:
\[
\begin{array}{r|c|c|c|c|c|c|c|c|c|c|}\hhline{~|-|-|-|-|-|-|-|-|-|-|}
\renewcommand{\arraystretch}{2}
\text{\scriptsize $q$} &\cellcolor{delcol}\Delta &\cellcolor{delcol}\dots & \cellcolor{delcol}\Delta  & 
\cellcolor{Ncol} N_i &\cellcolor{Ncol} N_{i+1} & \cellcolor{Ncol}\dots & \cellcolor{Ncol} N_{m-1} &
\cellcolor{Rcol} R & \cellcolor{Rcol} R& \cellcolor{Rcol}\dots
\\ \hhline{~|-|-|-|-|-|-|-|-|-|-|}
\multicolumn{1}{c}{}&\multicolumn{1}{c}{}&\multicolumn{1}{c}{}&\multicolumn{1}{c}{}&\multicolumn{1}{c}{\text{\scriptsize $i$}}&\multicolumn{1}{c}{}&\multicolumn{1}{c}{}&\multicolumn{1}{c}{}&\multicolumn{1}{c}{\text{\scriptsize $m$}}&\multicolumn{1}{c}{}&\multicolumn{1}{c}{}
\end{array}
\]
Here $0\leq i \leq m $,   and $\{\id_q\}\neq N_i\leq \dots \leq N_{m-1}$ are non-trivial normal subgroups of $\S_q$.
\end{rowtype}

\bigskip

Having specified the possible rows in $M$, the way they can be put together is governed by the following \emph{verticality conditions}:

\begin{enumerate}[label=\textsf{(V\arabic*)}, widest=(V2), leftmargin=10mm]
\item
\label{V1}
An $N$-symbol cannot be immediately above $\Delta$, $\muup$, $\mudown$, or another $N$-symbol.
\item
\label{V2}
Every entry equal to $R$ in row $q\geq 2$ must be directly above an $R$ entry from row~$q-1$.  (The same automatically holds for $R$s in row $q=1$ by examining types \ref{RT1}--\ref{RT7}.)
\end{enumerate}

\begin{defn}[\bf C-pair]
\label{de:Cpair}
A \emph{C-pair} $(\Th,M)$ consists of
a descending chain $\Th=(\th_0,\ldots,\th_n)$ of congruences on $\N$, and 
a matrix $M=(M_{qi})_{\bnz\times\N}$,  in which rows $0$ and $1$ are of one of the types \ref{RT1}--\ref{RT7}, 
each of the remaining rows is of one of the types \ref{RT8}--\ref{RT10}, and 
 the verticality conditions \ref{V1} and \ref{V2} are satisfied.
We refer to $\Theta$ as a \emph{C-chain}, and to $M$ as a \emph{C-matrix}.
With a slight abuse of terminology, we will say that $M$ is of type \ref{RT1}--\ref{RT7}, as appropriate, according to the type of rows $0$ and $1$.
\end{defn}

\begin{rem}
\label{re:se}
The specifications of row types and the verticality conditions impose severe restrictions about the content of a C-matrix:
\ben
\item \label{se1}  For any $q\in\bnz$, and for any $(i,j)\in\th_q$, we have $M_{qi}=M_{qj}$.  Thus, if $m:=\min\th_q\neq \infty$, then $M_{qi}=M_{qm}$ for all $i\geq m$.
\item \label{se2}  If $M_{1i}\neq \De$ for some $i\in\N$, then $M_{0j}=M_{1,j+1}$ for all $j\geq i$.
\item \label{se3}  If $M_{1i}=\xi\neq \De$ for some $i\geq\min\th_1$, then $M_{1j}=M_{0k}=\xi$ for all $j\geq\min\th_1$ and all $k\geq\min\th_0$.
\item \label{se4}  Symbols $\Delta$ and $R$ can appear in any row;  $N$-symbols can appear in rows $q\geq 2$; $\mu$,~$\rho$ and~$\lambda$ can appear in rows $0$ and $1$; $\muup$ and $\mudown$ can appear only in row $1$, and~$M$ has at most one entry from $\{\muup,\mudown\}$.
\item \label{se5}  Given an entry $M_{qi}$, only certain entries can occur directly to the right or below it; they are given in Table 
\ref{tab:se}.
\item \label{se6}  At most one row can be of type \ref{RT9}, and any rows above such a row consist entirely of~$\De$s.
\een
\end{rem}

\begin{table}[ht]
\begin{center}
$
\begin{array}{c| *{6}{|c}}
\text{\boldmath{$M_{qi}$}} & \Delta & \muup , \mudown , \mu & \rho &\lambda & N & R\\ \hline\hline
\text{\boldmath{$M_{q-1,i}$}}&\text{any}  &\mu,\rho,\lambda,R &\rho&\lambda&
\mu,\rho,\lambda,R & R\\ \hline
\text{\boldmath{$M_{q,i+1}$}}&\text{any} &\mu,\rho,\lambda,R &\rho&\lambda& N'(\geq N),R&R
\end{array}
$

\caption{Allowed entries below and to the right of an entry in a C-matrix.}
\label{tab:se}

\end{center}

\end{table}

The next definition gives a detailed specification for the congruence corresponding to a C-pair. That this indeed is a congruence will be proved in Section \ref{sec:I}.  The definition involves the $\pd$ operator, defined just before Theorem \ref{thm:CongPn}.

\begin{defn}[\bf Congruence corresponding to a C-pair]
\label{de:cg}
The congruence associated with a C-pair $(\Th,M)$ is the relation $\cg(\Th,M)$ on $\Ptw{n}$ consisting of all pairs $((i,\alpha),(j,\beta))\in \Ptw{n}\times \Ptw{n}$ such that one of the following holds, writing $q=\rank\alpha$ and $r=\rank\beta$:
\begin{enumerate}[label=\textsf{(C\arabic*)}, widest=(C8), leftmargin=10mm]
\item
\label{C1}
$M_{qi}=M_{rj}=\De$, $(i,j)\in\th_q$ and $\alpha=\beta$;
\item
\label{C2}
$M_{qi}=M_{rj}=R$;
\item
\label{C3}
$M_{qi}=M_{rj}=N$, $(i,j)\in \th_q$, $\alpha\rH\beta$ and $\pd(\alpha,\beta)\in N$;
\item
\label{C4}
$M_{qi}=M_{rj}=\lambda$ and $\widehat{\alpha}\rL\widehat{\beta}$;
\item
\label{C5}
$M_{qi}=M_{rj}=\rho$ and $\widehat{\alpha}\rR\widehat{\beta}$;
\item
\label{C6}
$M_{qi}=M_{rj}=\mudown$,  $\widehat{\alpha}=\widehat{\beta}$ and $\alpha \rL\beta$;
\item
\label{C7}
$M_{qi}=M_{rj}=\muup$,  $\widehat{\alpha}=\widehat{\beta}$ and $\alpha \rR\beta$;
\item
\label{C8}
$M_{qi}=M_{rj}=\mu$,  $\widehat{\alpha}=\widehat{\beta}$ and one of the following holds:
\bit
\item $q=r$ and $(i,j)\in\th_q$, or
\item $q\neq r$, $(i+r,j+q)\in\th_0$, $i<\min\th_q$ and $j<\min\th_r$, or
\item $q\neq r$, $(i+r,j+q)\in\th_0$, $i\geq\min\th_q$ and $j\geq\min\th_r$.
\eit
\end{enumerate}
\end{defn}

Note that in \ref{C1} and \ref{C3} we necessarily have $q=r$.  Similarly, in \ref{C4}, \ref{C5} and~\ref{C8} we have $q,r\in\{0,1\}$; in \ref{C6} and \ref{C7} we have $q=r=1$ and $i=j$.
The comparatively complex rule in \ref{C8} is to do with the 
interactions between the $\al\mt\wh\al$ map and the $\Float$ parameters, as already gleaned  
 in Lemma \ref{la:A2} and Example \ref{ex:weirdproj}.

\begin{rem}\label{rem:C1'}
It will often be convenient to replace \ref{C1} in the above definition by:
\begin{enumerate}[label=\textsf{(C\arabic*$'$)}, widest=(C8), leftmargin=10mm]
\item \label{C1'} $(i,j)\in\th_q$ and $\alpha=\beta$.
\end{enumerate}
While \ref{C1} of course implies \ref{C1'}, the converse is not true.  Nevertheless, \ref{C1'} implies that one of \ref{C1}--\ref{C8} holds, as is easily checked, keeping in mind that $\al=\be$ implies $q=r$ and then $(i,j)\in\th_q$ implies $M_{qi}=M_{rj}$ by Remark \ref{re:se}\ref{se1}.
\end{rem}

\begin{rem}
\label{rem:C8help}
If $(i,\alpha)$ and $(j,\beta)$ are related via \ref{C8} then $i<\min\theta_q$ if and only if $ j<\min\theta_r$.
\end{rem}

\begin{rem}\label{rem:nuN}
It will sometimes be convenient to treat $\Delta$ in row $q\geq2$ as an $N$-symbol, by allowing the trivial subgroup
$\{\id_q\}$ among the latter, and then identifying $\Delta$ with it.
In this way, \ref{C1} is contained in \ref{C3} for $q\geq2$, as $\al\rH\al$ and $\pd(\al,\al)=\id_q$ for all $\al\in D_q$.
This convention will be particularly useful in the treatment of exceptional C-pairs (see Definition \ref{def:exc} below) where the exceptional row is $q=2$ and we have $\{\id_2\}=\A_2$.
\end{rem}

It turns out that `most' congruences on $\Ptw n$ are of the form $\cg(\Th,M)$.  Only one other family of congruences arises, and this only for a very specific kind of C-pair:

\begin{defn}[\bf Exceptional C-pair]
\label{def:exc}
A C-pair $(\Th,M)$ is \emph{exceptional} if there exists $q\geq 2$ such that:
\begin{itemize}[label=\textbullet, leftmargin=5mm]
\item
$\th_q=(m,m+2d)^\sharp$ for some $m\geq0$ and $d\geq1$;
\item
$M_{qm}=\A_q$ if $q>2$;
\item
If $q=2$ then $M_{2m}=\Delta$, $M_{1m}\in\{\mu,\rho,\lam,R\}$ and $(m,m+d)^\sharp\subseteq\theta_1$.
\end{itemize}
\end{defn}

This $q$ is necessarily unique (Remark \ref{re:se}\ref{se6}), and we call row $q$ the \emph{exceptional row}, and write~$q=:\hgt(M)$.
If $\theta_q=(m,m+2d)^\sharp$, we let $\thx:=(m,m+d)^\sharp$.
Thus the final condition in the last bullet point above states that if $\hgt(M)=2$ then $\thx_2\subseteq \theta_1$.
In fact, $\thx_q\subseteq \theta_{q-1}$ for any value of $q=\hgt(M)$.
 Indeed, for $q\geq3$, condition \ref{V1} ensures that the entry below $M_{qm}=\A_q$ is $R$, and then we have $\per\th_{q-1}=1$; we also have $\min\th_{q-1}\leq\min\th_q=m$, as $\th_q\sub\th_{q-1}$.

\begin{defn}[\bf Exceptional congruence]
To the exceptional C-pair $(\Th,M)$, in addition to $\cg(\Th,M)$, we also associate the \emph{exceptional congruence} $\cgx(\Th,M)$ consisting of all pairs $((i,\al),(j,\be))$ such that one of \ref{C1}--\ref{C8} holds, or else:
\begin{enumerate}[label=\textsf{(C\arabic*)}, widest=(C8), leftmargin=10mm]
\addtocounter{enumi}{8}
\item \label{C9} $(i,j)\in\thx_q\sm\th_q$, $\al\rH\be$ and $\pd(\alpha,\beta)\in \S_q\setminus \A_q$. 
\end{enumerate}
\end{defn}

Intuitively we can think about the extra pairs in \ref{C9} as follows.  Keeping the above notation, the partition $\S_q=\A_q\,\dot{\cup}\, (\S_q\setminus A_q)$ induces a partition of an arbitrary $\H$-class $H$ contained in~$D_q$, say $H=A\,\dot{\cup}\, B$, using the $\pd$ operator.
Rule \ref{C3} implies in particular that for $i,j\geq m$ with ${i\equiv j\pmod{2d}}$,
the elements of $\{i\}\times A$ and $\{j\}\times A$ are all related to each other, and similarly with $\{i\} \times B$ and $\{j\}\times B$.
 What rule \ref{C9} does is introduce additional `in-between' pairs, which `twist around' $A$ and $B$, in the sense that for $i,j\geq m$ with $i\equiv j+d \pmod{2d}$, the elements of $\{i\}\times A$ and $\{j\}\times B$ are all related to each other, and similarly with $\{i\} \times B$ and $\{j\}\times A$.

\subsection{The main result}\label{subsec:main}

We are now ready to state the main result of the paper.

\begin{thm}
\label{thm:main}
For $n\geq1$, the congruences on the twisted partition monoid $\Ptw{n}$ are precisely:
\begin{itemize}[label=\textbullet, leftmargin=5mm]
\item
$\cg(\Th,M)$ where $(\Th,M)$ is any C-pair;
\item
$\cgx(\Th,M)$ where $(\Th,M)$ is any exceptional C-pair.
\end{itemize}
\end{thm}

\begin{proof}[\bf Outline of proof.]
The proof naturally splits into two parts:  we show in Section \ref{sec:I} that each relation listed in the theorem is indeed a congruence, and in Section~\ref{sec:II} that any congruence on~$\Ptw n$ has one of the listed forms.  
\end{proof}

Before we proceed with the proof it is worth returning to the example congruences from Subsection \ref{subsec:exa}, and finding their associated C-pairs.  
We will adopt the notation for $C$-pairs where we write the matrix as usual and write each congruence $\theta_q$ to the right of row $q$.

\begin{exa}
\label{ex:ReesC}
Regarding Rees congruences, consider an ideal $I$ of $\Ptw n$.  Then $R_I=\cg(\Th,M)$, where the C-pair $(\Th,M)$ is defined as follows.  First, for any $q\in\bnz$ and $i\in\N$, we have 
\[
M_{qi} = \begin{cases}
R &\text{if $D_{qi}\sub I$}\\
\De &\text{otherwise.}
\end{cases}
\]
For any $q\in\bnz$, we have $\th_q=\De_\N$ if row $q$ of $M$ consists entirely of $\De$s; otherwise, $\th_q=(m_q,m_q+1)^\sharp$ where $m_q=\min\set{i\in\N}{M_{qi}=R}$.  For example, if $I$ is the ideal of $\Ptw4$ pictured in Figure \ref{fig:gridid}, then $R_I=\cg(\Pi)$, where $\Pi=(\Theta,M)$ is
\[
 \Cmatsetup\begin{array}{|c|c|c|c|c|c|cc}
\hhline{|-|-|-|-|-|-|~~}
\cellcolor{delcol}\Delta & \cellcolor{delcol}\Delta & \cellcolor{delcol}\Delta & \cellcolor{delcol}\Delta & \cellcolor{delcol}\Delta & \cellcolor{delcol}\cdots &\hspace{2mm}& \Delta_\N \\ \hhline{|-|-|-|-|-|-|~~}
\cellcolor{delcol}\Delta & \cellcolor{delcol}\Delta & \cellcolor{delcol}\Delta & \cellcolor{Rcol}R & \cellcolor{Rcol}R & \cellcolor{Rcol}\cdots && (3,4)^\sharp \\ \hhline{|-|-|-|-|-|-|~~}
\cellcolor{delcol}\Delta & \cellcolor{delcol}\Delta & \cellcolor{delcol}\Delta & \cellcolor{Rcol}R & \cellcolor{Rcol}R & \cellcolor{Rcol}\cdots &&  (3,4)^\sharp \\ \hhline{|-|-|-|-|-|-|~~}
\cellcolor{delcol}\Delta & \cellcolor{delcol}\Delta & \cellcolor{Rcol}R & \cellcolor{Rcol}R & \cellcolor{Rcol}R & \cellcolor{Rcol}\cdots &&  (2,3)^\sharp\\ \hhline{|-|-|-|-|-|-|~~}
\cellcolor{Rcol}R & \cellcolor{Rcol}R & \cellcolor{Rcol}R & \cellcolor{Rcol}R & \cellcolor{Rcol}R & \cellcolor{Rcol}\cdots
&& \nabla_\N \\ \hhline{|-|-|-|-|-|-|~~}
\end{array}
\ .
\]
\end{exa}

\begin{exa}
Next, let $\tau\in\Cong(\P_n)$, and let $\si$ be the congruence on $\Ptw n$ defined in Example~\ref{ex:tau}.  This time $\si=\cg(\Th,M)$, where $\Th=(\nab_\N,\ldots,\nab_\N)$, and where the form of $M$ depends on the congruence $\tau$ (as per Theorem \ref{thm:CongPn}).  For example with $n=4$, and taking $\tau$ to be $\lam_0$, $\rho_1$, $\mu_{\S_2}$ or $R_{\A_3}$, respectively, $\sigma=\cg(\Pi)$, where $\Pi$ is:
\[
\Cmatsetup
\begin{array}{|c|c|c|c|cc}
\hhline{|-|-|-|-|~~}
\cellcolor{delcol}\Delta & \cellcolor{delcol}\Delta & \cellcolor{delcol}\Delta & \cellcolor{delcol}\cdots&\hspace{2mm}&\nabla_\N \\ \hhline{|-|-|-|-|~~}
\cellcolor{delcol}\Delta & \cellcolor{delcol}\Delta & \cellcolor{delcol}\Delta & \cellcolor{delcol}\cdots&&\nabla_\N \\ \hhline{|-|-|-|-|~~}
\cellcolor{delcol}\Delta & \cellcolor{delcol}\Delta & \cellcolor{delcol}\Delta & \cellcolor{delcol}\cdots&&\nabla_\N \\ \hhline{|-|-|-|-|~~}
\cellcolor{delcol}\Delta & \cellcolor{delcol}\Delta & \cellcolor{delcol}\Delta & \cellcolor{delcol}\cdots&&\nabla_\N \\ \hhline{|-|-|-|-|~~}
\cellcolor{Rcol}\lam & \cellcolor{Rcol}\lam & \cellcolor{Rcol}\lam & \cellcolor{Rcol}\cdots&&\nabla_\N \\ \hhline{|-|-|-|-|~~}
\end{array}
\ \COMMA
\begin{array}{|c|c|c|c|cc}
\hhline{|-|-|-|-|~~}
\renewcommand{\arraystretch}{2}
\cellcolor{delcol}\Delta & \cellcolor{delcol}\Delta & \cellcolor{delcol}\Delta & \cellcolor{delcol}\cdots&\hspace{2mm}&\nabla_\N \\ \hhline{|-|-|-|-|~~}
\cellcolor{delcol}\Delta & \cellcolor{delcol}\Delta & \cellcolor{delcol}\Delta & \cellcolor{delcol}\cdots&&\nabla_\N \\ \hhline{|-|-|-|-|~~}
\cellcolor{delcol}\Delta & \cellcolor{delcol}\Delta & \cellcolor{delcol}\Delta & \cellcolor{delcol}\cdots&&\nabla_\N \\ \hhline{|-|-|-|-|~~}
\cellcolor{Rcol}\rho & \cellcolor{Rcol}\rho & \cellcolor{Rcol}\rho & \cellcolor{Rcol}\cdots &&\nabla_\N\\ \hhline{|-|-|-|-|~~}
\cellcolor{Rcol}\rho & \cellcolor{Rcol}\rho & \cellcolor{Rcol}\rho & \cellcolor{Rcol}\cdots &&\nabla_\N\\ \hhline{|-|-|-|-|~~}
\end{array}
\ \COMMA
\begin{array}{|c|c|c|c|cc}
\hhline{|-|-|-|-|~~}
\renewcommand{\arraystretch}{2}
\cellcolor{delcol}\Delta & \cellcolor{delcol}\Delta & \cellcolor{delcol}\Delta & \cellcolor{delcol}\cdots&\hspace{2mm}&\nabla_\N \\ \hhline{|-|-|-|-|~~}
\cellcolor{delcol}\Delta & \cellcolor{delcol}\Delta & \cellcolor{delcol}\Delta & \cellcolor{delcol}\cdots&&\nabla_\N \\ \hhline{|-|-|-|-|~~}
\cellcolor{Ncol}\S_2 & \cellcolor{Ncol}\S_2 & \cellcolor{Ncol}\S_2 & \cellcolor{Ncol}\cdots&&\nabla_\N \\ \hhline{|-|-|-|-|~~}
\cellcolor{Rcol}\mu & \cellcolor{Rcol}\mu & \cellcolor{Rcol}\mu & \cellcolor{Rcol}\cdots&&\nabla_\N \\ \hhline{|-|-|-|-|~~}
\cellcolor{Rcol}\mu & \cellcolor{Rcol}\mu & \cellcolor{Rcol}\mu & \cellcolor{Rcol}\cdots&&\nabla_\N \\ \hhline{|-|-|-|-|~~}
\end{array}
\qquad\text{or}\qquad
\begin{array}{|c|c|c|c|cc}
\hhline{|-|-|-|-|~~}
\renewcommand{\arraystretch}{2}
\cellcolor{delcol}\Delta & \cellcolor{delcol}\Delta & \cellcolor{delcol}\Delta & \cellcolor{delcol}\cdots&\hspace{2mm}&\nabla_\N \\ \hhline{|-|-|-|-|~~}
\cellcolor{Ncol}\A_3 & \cellcolor{Ncol}\A_3 & \cellcolor{Ncol}\A_3 & \cellcolor{Ncol}\cdots&&\nabla_\N \\ \hhline{|-|-|-|-|~~}
\cellcolor{Rcol}R & \cellcolor{Rcol}R & \cellcolor{Rcol}R & \cellcolor{Rcol}\cdots&&\nabla_\N \\ \hhline{|-|-|-|-|~~}
\cellcolor{Rcol}R & \cellcolor{Rcol}R & \cellcolor{Rcol}R & \cellcolor{Rcol}\cdots&&\nabla_\N \\ \hhline{|-|-|-|-|~~}
\cellcolor{Rcol}R & \cellcolor{Rcol}R & \cellcolor{Rcol}R & \cellcolor{Rcol}\cdots&&\nabla_\N \\ \hhline{|-|-|-|-|~~}
\end{array}
\ .
\]
\end{exa}

\begin{exa}
The relatively unusual congruence from Example \ref{ex:weirdproj} has the following C-pair: 
\[
\Cmatsetup
 \begin{array}{|c|c|c|c|cc}
\hhline{|-|-|-|-|~~}
\renewcommand{\arraystretch}{2}
\cellcolor{delcol}\Delta & \cellcolor{delcol}\Delta & \cellcolor{delcol}\Delta & \cellcolor{delcol}\cdots&\hspace{2mm}&\Delta_\N \\ \hhline{|-|-|-|-|~~}
\cellcolor{delcol}\vvdots & \cellcolor{delcol}\vvdots & \cellcolor{delcol}\vvdots & \cellcolor{delcol}\vvdots&&\vvdots \\ \hhline{|-|-|-|-|~~}
\cellcolor{delcol}\Delta & \cellcolor{delcol}\Delta & \cellcolor{delcol}\Delta & \cellcolor{delcol}\cdots&&\Delta_\N \\ \hhline{|-|-|-|-|~~}
\cellcolor{excepcol}\mu&\cellcolor{mucol}\mu&\cellcolor{mucol}\mu& \cellcolor{mucol}\cdots&&\Delta_\N\\ \hhline{|-|-|-|-|~~}
\cellcolor{mucol}\mu&\cellcolor{mucol}\mu&\cellcolor{mucol}\mu& \cellcolor{mucol}\cdots&&\Delta_\N\\ \hhline{|-|-|-|-|~~}
\end{array}
\ .
\]
\end{exa}

\begin{exa}
Finally, the congruences in Examples \ref{exa:del} and \ref{exa:del2} both have $M=(\De)_{\bnz\times\N}$.
\end{exa}

Note that none of the above congruences are exceptional.

\begin{exa}
\label{ex:excep}
The following are three examples of exceptional C-pairs with $n=4$,  and with the exceptional row at $q=4$, $3$ and $2$ respectively (in the first, $\mathcal K_4$ denotes the Klein 4-group):
\begin{gather*}
\Cmatsetup
 \begin{array}{|c|c|c|c|c|c|c|c|c|c|c|c|cc}
\hhline{|-|-|-|-|-|-|-|-|-|-|-|-|~~}
\cellcolor{delcol}\Delta & \cellcolor{delcol}\Delta &\cellcolor{delcol}\Delta &\cellcolor{delcol}\Delta &\cellcolor{delcol}\Delta &\cellcolor{delcol}\Delta &\cellcolor{Ncol} \mathcal K_4 &\cellcolor{Ncol}\mathcal K_4 &\cellcolor{Ncol}\A_4 &\cellcolor{Ncol}\A_4 &\cellcolor{Ncol}\A_4&\cellcolor{Ncol}\cdots &\hspace{2mm}&(9,11)^\sharp \\ \hhline{|-|-|-|-|-|-|-|-|-|-|-|-|~~}
\cellcolor{delcol}\Delta & \cellcolor{delcol}\Delta &\cellcolor{delcol}\Delta &\cellcolor{delcol}\Delta &\cellcolor{Ncol}\A_3 &\cellcolor{Ncol}\S_3 &\cellcolor{Rcol} R &\cellcolor{Rcol}R &\cellcolor{Rcol}R &\cellcolor{Rcol}R &\cellcolor{Rcol}R&\cellcolor{Rcol}\cdots &&(6,7)^\sharp \\ \hhline{|-|-|-|-|-|-|-|-|-|-|-|-|~~}
\cellcolor{delcol}\Delta & \cellcolor{Ncol}\S_2 &\cellcolor{Ncol}\S_2 &\cellcolor{Rcol}R &\cellcolor{Rcol}R &\cellcolor{Rcol}R &\cellcolor{Rcol} R &\cellcolor{Rcol}R &\cellcolor{Rcol}R &\cellcolor{Rcol}R &\cellcolor{Rcol}R&\cellcolor{Rcol}\cdots &&(3,4)^\sharp \\ \hhline{|-|-|-|-|-|-|-|-|-|-|-|-|~~}
\cellcolor{delcol}\Delta & \cellcolor{excepcol}\mu &\cellcolor{Rcol}R&\cellcolor{Rcol}R &\cellcolor{Rcol}R &\cellcolor{Rcol}R &\cellcolor{Rcol} R &\cellcolor{Rcol}R &\cellcolor{Rcol}R &\cellcolor{Rcol}R &\cellcolor{Rcol}R&\cellcolor{Rcol}\cdots &&(2,3)^\sharp \\ \hhline{|-|-|-|-|-|-|-|-|-|-|-|-|~~}
\cellcolor{Rcol}R & \cellcolor{Rcol}R &\cellcolor{Rcol}R&\cellcolor{Rcol}R &\cellcolor{Rcol}R &\cellcolor{Rcol}R &\cellcolor{Rcol} R &\cellcolor{Rcol}R &\cellcolor{Rcol}R &\cellcolor{Rcol}R &\cellcolor{Rcol}R&\cellcolor{Rcol}\cdots &&\nabla_\N \\ \hhline{|-|-|-|-|-|-|-|-|-|-|-|-|~~}
\end{array}
\ \COMMA
\Cmatsetup
 \begin{array}{|c|c|c|c|c|c|c|c|c|c|c|cc}
\hhline{|-|-|-|-|-|-|-|-|-|-|-|~~}
\cellcolor{delcol}\Delta & \cellcolor{delcol}\Delta &\cellcolor{delcol}\Delta &\cellcolor{delcol}\Delta &\cellcolor{delcol}\Delta &\cellcolor{delcol}\Delta &\cellcolor{delcol}\Delta &\cellcolor{delcol}\Delta &\cellcolor{delcol}\Delta &\cellcolor{delcol}\Delta &\cellcolor{delcol}\cdots &\hspace{2mm}&(7,23)^\sharp \\ \hhline{|-|-|-|-|-|-|-|-|-|-|-|~~}
\cellcolor{delcol}\Delta & \cellcolor{delcol}\Delta &\cellcolor{delcol}\Delta &\cellcolor{delcol}\Delta &\cellcolor{delcol}\Delta &\cellcolor{Ncol}\A_3&\cellcolor{Ncol}\A_3 &\cellcolor{Ncol}\A_3 &\cellcolor{Ncol}\A_3 &\cellcolor{Ncol}\A_3 &\cellcolor{Ncol}\cdots && (7,15)^\sharp \\ \hhline{|-|-|-|-|-|-|-|-|-|-|-|~~}
\cellcolor{delcol}\Delta & \cellcolor{delcol}\Delta &\cellcolor{delcol}\Delta &\cellcolor{delcol}\Delta &\cellcolor{Rcol}R &\cellcolor{Rcol}R&\cellcolor{Rcol} R &\cellcolor{Rcol} R &\cellcolor{Rcol} R &\cellcolor{Rcol} R &\cellcolor{Rcol}\cdots && (4,5)^\sharp \\ \hhline{|-|-|-|-|-|-|-|-|-|-|-|~~}
\cellcolor{delcol}\Delta & \cellcolor{delcol}\Delta &\cellcolor{delcol}\Delta &\cellcolor{mucol}\mu &\cellcolor{Rcol}R &\cellcolor{Rcol}R&\cellcolor{Rcol} R &\cellcolor{Rcol} R &\cellcolor{Rcol} R &\cellcolor{Rcol} R &\cellcolor{Rcol}\cdots && (4,5)^\sharp \\ \hhline{|-|-|-|-|-|-|-|-|-|-|-|~~}
\cellcolor{mucol}\mu & \cellcolor{Rcol}R &\cellcolor{Rcol}R &\cellcolor{Rcol}R &\cellcolor{Rcol}R &\cellcolor{Rcol}R&\cellcolor{Rcol} R &\cellcolor{Rcol} R &\cellcolor{Rcol} R &\cellcolor{Rcol} R &\cellcolor{Rcol}\cdots && (1,2)^\sharp \\ \hhline{|-|-|-|-|-|-|-|-|-|-|-|~~}
\end{array}
\ ,
\displaybreak[2]
\\[5mm]
\Cmatsetup
 \begin{array}{|c|c|c|c|c|c|c|c|c|c|c|cc}
\hhline{|-|-|-|-|-|-|-|-|-|-|-|~~}
\cellcolor{delcol}\Delta & \cellcolor{delcol}\Delta &\cellcolor{delcol}\Delta &\cellcolor{delcol}\Delta &\cellcolor{delcol}\Delta &\cellcolor{delcol}\Delta &\cellcolor{delcol}\Delta &\cellcolor{delcol}\Delta &\cellcolor{delcol}\Delta &\cellcolor{delcol}\Delta &\cellcolor{delcol}\cdots &\hspace{2mm}&\Delta_\N \\ \hhline{|-|-|-|-|-|-|-|-|-|-|-|~~}
\cellcolor{delcol}\Delta & \cellcolor{delcol}\Delta &\cellcolor{delcol}\Delta &\cellcolor{delcol}\Delta &\cellcolor{delcol}\Delta &\cellcolor{delcol}\Delta &\cellcolor{delcol}\Delta &\cellcolor{delcol}\Delta &\cellcolor{delcol}\Delta &\cellcolor{delcol}\Delta &\cellcolor{delcol}\cdots && (9,13)^\sharp \\ \hhline{|-|-|-|-|-|-|-|-|-|-|-|~~}
\cellcolor{delcol}\Delta & \cellcolor{delcol}\Delta &\cellcolor{delcol}\Delta &\cellcolor{delcol}\Delta &\cellcolor{delcol}\Delta &\cellcolor{delcol}\Delta &\cellcolor{delcol}\Delta &\cellcolor{delcol}\Delta &\cellcolor{delcol}\Delta &\cellcolor{delcol}\Delta &\cellcolor{delcol}\cdots && (8,12)^\sharp \\ \hhline{|-|-|-|-|-|-|-|-|-|-|-|~~}
\cellcolor{delcol}\Delta & \cellcolor{excepcol}\mudown &\cellcolor{mucol}\mu &\cellcolor{mucol}\mu &\cellcolor{mucol}\mu &\cellcolor{mucol}\mu &\cellcolor{Rcol}\mu &\cellcolor{Rcol}\mu &\cellcolor{Rcol}\mu &\cellcolor{Rcol}\mu &\cellcolor{Rcol}\cdots && (6,8)^\sharp \\ \hhline{|-|-|-|-|-|-|-|-|-|-|-|~~}
\cellcolor{delcol}\Delta & \cellcolor{mucol}\mu &\cellcolor{mucol}\mu &\cellcolor{mucol}\mu &\cellcolor{mucol}\mu &\cellcolor{Rcol}\mu &\cellcolor{Rcol}\mu &\cellcolor{Rcol}\mu &\cellcolor{Rcol}\mu &\cellcolor{Rcol}\mu &\cellcolor{Rcol}\cdots && (5,7)^\sharp \\\hhline{|-|-|-|-|-|-|-|-|-|-|-|~~}
\end{array}
\ .
\end{gather*}
\end{exa}

We conclude this section by recording some simple consequences of Theorem \ref{thm:main}.  The first concerns the number of congruences.  
Note that a semigroup $S$ can have as many as $|{\operatorname{Eq}(S)}|$ congruences, 
where $\operatorname{Eq}(S)$ is the set of all equivalence relations on $S$, and that $|{\operatorname{Eq}(S)}|=2^{|S|}$ when $S$ is infinite.

\begin{cor}\label{co:countable}
The twisted partition monoid $\Ptw{n}$ has only countably many congruences.
\end{cor}

\begin{proof}
There are only countably many congruences on $\N$, and hence only countably many C-chains.
The number of C-matrices is also countable, because each C-matrix has $n+1$ rows, and each row is eventually constant. Hence there are only countably many C-pairs, and each yields at most two congruences.
\end{proof}

We can also characterise congruences of finite index:

\begin{cor}
\label{co:fi}
Let $\sigma$ be a congruence of $\Ptw{n}$, and let $(\Theta,M)$ be the associated C-pair.
Then the quotient $\Ptw{n}/\sigma$ is finite if and only if $\theta_n\neq \Delta_\N$.
\end{cor}

\begin{proof}
If $\theta_n=\Delta_\N$ then row $n$ is of type \ref{RT8} or \ref{RT9}, and is not exceptional (though a lower row might be).  It follows from \ref{C1} or \ref{C3} that for any $i\in\N$, elements of $D_{ni}$ can only be $\si$-related to elements of $D_{ni}$, and so $\sigma$ has infinitely many classes. 

Conversely, if $\theta_n=(m,m+d)^\sharp$, then since every other~$\th_q$ contains $\th_n$, it follows from \ref{C1'} that each element of $\Ptw n$ is $\si$-related to an element from columns $0,1,\dots,m+d-1$. Since the columns themselves are finite, $\sigma$ has only finitely many classes.
\end{proof}

\begin{rem}
Although infinite, a C-pair $(\Th,M)$ can be finitely encoded.  Indeed, the C-chain $\Th=(\th_0,\ldots,\th_n)$ is determined by the numbers ${\min\th_q,\per\th_q\in\N\cup\{\infty\}}$, for each $q\in\bnz$; and each row of the C-matrix $M$, being eventually constant, is determined by the symbols that appear in that row, and the first position of each such symbol.  
\end{rem}

\begin{rem}
A reader can spot some similarities between the twisted partition monoid $\Ptw{n}$ and the direct product $\N\times\P_n$ of the additive monoid of natural numbers and the partition monoid $\P_n$.
Perhaps the similarities are most striking in the rectangular description of $\D$-classes, as illustrated in Figure \ref{fig:gridid}. The problem of finding congruences of a direct product in general is a difficult one,
but recent work of Ara\'{u}jo, Bentz and Gomes \cite{Araujo18} treats it in the special case of transformation and matrix semigroups.
There are certain formal similarities between their description and ours, and a careful examination of these may be a useful pointer for future investigations.
\end{rem}

\section{C-pair relations are congruences}
\label{sec:I}

This section is entirely devoted to proving the following:

\begin{prop}
\label{pro:arecongs}
For any C-pair $\Pair=(\Th,M)$ the relation
$\cg(\Pair)$ is a congruence, and, if $\Pair$ is exceptional, the relation $\cgx(\Pair)$ is also a congruence.
\end{prop}

\begin{proof}
For the first statement, 
write $\si:=\cg(\Pair)$.  First we check that $\si$ is an equivalence.  Indeed, symmetry and reflexivity follow immediately by checking each of \ref{C1}--\ref{C8}; for \ref{C3}, note that this says $(i,j)\in\th_q$ and $(\al,\be)$ belong to the equivalence $\nu_N$ defined just before Theorem~\ref{thm:CongPn}.  For transitivity, suppose $(\ba,\bb),(\bb,\bc)\in\si$, where $\ba=(i,\al)$, $\bb=(j,\be)$ and $\bc=(k,\ga)$.  We then identify which of the conditions \ref{C1}--\ref{C8} is `responsible' for the pair $(\ba,\bb)$ belonging to $\si$.  But then that the same condition applies to $(\bb,\bc)$ because of the associated matrix entries.  
It is now easy to verify directly in each case \ref{C1}--\ref{C8} that $(\ba,\bb)\in \sigma$; when dealing with \ref{C8} 
an appeal to Remark \ref{rem:C8help} deals with the conditions concerning $\min\theta_q$ and $\min\theta_r$.

For compatibility, fix $(\ba,\bb)\in \si$ and $\bc\in \Ptw{n}$.  We must show that $(\ba\bc,\bb\bc),(\bc\ba,\bc\bb)\in\si$.
Write $\ba=(i,\alpha)\in D_{qi}$, $\bb=(j,\beta)\in D_{rj}$, $\bc=(k,\gamma)$, and
\begin{align*}
\ba\bc&=(i+k+\Float(\alpha,\gamma),\alpha\gamma) \in D_{q_1i_1}, & \bc\ba&=(k+i+\Float(\gamma,\alpha),\gamma\alpha)\in D_{q_2i_2},\\
\bb\bc&=(j+k+\Float(\beta,\gamma),\beta\gamma)\in D_{r_1j_1},& \bc\bb&=(k+j+\Float(\gamma,\beta),\gamma\beta)\in D_{r_2j_2}.
\end{align*}

For $t\in\{1,2\}$, note that $q_t\leq q$, $r_t\leq r$, $i_t\geq i$ and $j_t\geq j$.
We now split the proof into cases, depending on which of \ref{C1}--\ref{C8} is responsible for 
the pair $(\ba,\bb)$ belonging to $\si$.
In each case we will check that $(\ba\bc,\bb\bc)\in\si$.
With the exception of~\ref{C6}, the proof that $(\bc\ba,\bc\bb)\in\si$
is dual, and is omitted without further comment.
\smallskip

\ref{C1}
From $\alpha=\beta$ and $(i,j)\in\th_q$ it follows that $\alpha\gamma=\beta\gamma$ and 
$(i_1,j_1)\in\th_q\sub\th_{q_1}$ since $\th_q$ is a congruence on $\N$.
Thus, $(\ba\bc,\bb\bc)\in\si$ via \ref{C1'} (from Remark \ref{rem:C1'}).
\smallskip

\ref{C2}
The entry $M_{q_1i_1}$ is to the right and below of $M_{qi}$ (possibly not strictly), and hence $M_{q_1i_1}=R$ by Table \ref{tab:se}. Analogously
$M_{r_1j_1}=R$, and hence
$(\ba\bc,\bb\bc)\in\si$ via \ref{C2}.
\smallskip

\ref{C3}
Since $(\al,\be)\in\H\sub\L\implies{(\al\ga,\be\ga)\in\L\sub\J}$, we have $q_1=r_1$.  
By Lemma \ref{la:A4} we have $\Float(\alpha,\gamma)=\Float(\beta,\gamma)$, and so $(i_1,j_1)\in\th_q\subseteq \th_{q_1}$.
Remark~\ref{re:se}\ref{se1} then gives $M_{q_1i_1}=M_{r_1j_1}$.  

\setcounter{caseco}{0}

\case $q_1=q$.  Here Lemma \ref{la:AA2} gives $(\alpha\gamma,\beta\gamma)\in\nu_N$.  By Table \ref{tab:se}, $M_{qi_1}=M_{qj_1}$ must be either~$R$ or else some $N'\geq N$, so it follows that $(\ba\bc,\bb\bc)\in\si$ via \ref{C2} or \ref{C3}, depending on the actual value of $M_{qi_1}$.

\case $q_1<q$.
If $q>2$ then $M_{q_1i_1}=M_{r_1j_1}=R$ by Table \ref{tab:se}, and hence $(\ba\bc,\bb\bc)\in\si$ via~\ref{C2}.  So now suppose $q=2$.  By Table \ref{tab:se}, we have $M_{q_1i_1}=M_{r_1j_1}\in \{R,\lambda,\rho,\mu\}$.  Since $(\al,\be)$ belongs to the congruence $\mu_{\S_2}$ (on $\P_n$), so too does $(\al\ga,\be\ga)$, so it follows that $\widehat{\alpha\gamma}=\widehat{\beta\gamma}$.  Hence $(\ba\bc,\bb\bc)\in\si$ via one of \ref{C2}, \ref{C4}, \ref{C5} or \ref{C8}, as appropriate.
\smallskip

\ref{C4}
Here we have $M_{q_1i_1}=M_{r_1j_1}=\lambda$ by Table \ref{tab:se}.
Since $\widehat{\alpha}\rL\widehat{\beta}$ means that $(\alpha,\beta)\in\lambda_1$,
a congruence on $\P_n$, it follows that $(\alpha\gamma,\beta\gamma)\in\lambda_1$, whence $\wh{\al\ga}\rL\wh{\be\ga}$, so
$(\ba\bc,\bb\bc) \in\si$ via~\ref{C4}.
\smallskip

\ref{C5} This is dual to \ref{C4}.
\smallskip

\ref{C6}
Here we must have $q=r=1$, $i=j$, $\wh\al=\wh\be$ and $\alpha\rL\beta$.  By Lemma \ref{la:A4} we have ${\Float(\al,\ga)=\Float(\be,\ga)}$, and so $i_1=j_1$.  As with \ref{C3}, we also have $\al\ga\rL\be\ga$ and so $q_1=r_1$.  In particular, we have $M_{q_1i_1}=M_{r_1j_1}\in\{ R,\lambda,\rho,\mu,\mudown\}$ by Table \ref{tab:se}.  From $\widehat{\alpha}=\widehat{\beta}$, Lemma \ref{la:hat} gives $\widehat{\alpha\gamma}=\widehat{\beta\gamma}$.  It follows that $(\ba\bc,\bb\bc)\in\si$ via \ref{C2}, \ref{C4}, \ref{C5}, \ref{C6} or \ref{C8}.

In this case we do need to also verify that $(\bc\ba,\bc\bb)\in\si$.  We still have $\wh{\ga\al}=\wh{\ga\be}$, but we might not have $q_2=r_2$ or $i_2=j_2$.  Writing $f:=\Float(\ga,\wh\al)=\Float(\ga,\wh\be)$, Lemma \ref{la:A2} gives
\begin{equation}\label{eq:Fga}
\Float(\ga,\al) = f-1+q_2 \AND \Float(\ga,\be) = f-1+r_2.
\end{equation}
Swapping $\al$ and $\be$ if necessary, we may assume that $q_2\geq r_2$.

\setcounter{caseco}{0}

\case $q_2=r_2=1$.  Here it follows quickly from \eqref{eq:Fga} that $i_2=j_2$, so again we have $M_{q_2i_2}=M_{r_2j_2}\in\{ R,\lambda,\rho,\mu,\mudown\}$.  
Using $\rank(\ga\al)=q_2=1=\rank(\alpha)$ it is easy to see that
$\codom(\gamma\alpha)=\codom(\alpha)$ and $\coker(\gamma\alpha)=\coker(\alpha)$,
i.e.~$\ga\al\rL\al$.
Similarly, $\ga\be\rL\be$, and so $\ga\al\rL\ga\be$.  It again follows that $(\bc\ba,\bc\bb)\in\si$ via \ref{C2}, \ref{C4}, \ref{C5}, \ref{C6} or \ref{C8}.

\case $q_2=r_2=0$.  Again $i_2=j_2$, but now $\ga\al=\wh{\ga\al}=\wh{\ga\be}=\ga\be$, so~$(\bc\ba,\bc\bb)\in\si$ via \ref{C1'}.

\case $q_2=1$ and $r_2=0$.  This time, \eqref{eq:Fga} gives $i_2=j_2+1$, and so $i_2>j_2\geq j=i$.  It follows from Remark \ref{re:se}\ref{se2} and \ref{se4} and Table \ref{tab:se} that ${M_{q_2i_2}=M_{1,j_2+1}=M_{0j_2}=M_{r_2j_2}\in\{R,\lam,\rho,\mu\}}$, and so $(\bc\ba,\bc\bb)\in\si$ via \ref{C2}, \ref{C4}, \ref{C5} or \ref{C8}.  In the $\mu$ case, we use the second or third option in \ref{C8}; since the presence of the $\mudown$ symbol implies type \ref{RT2}, \ref{RT5} or \ref{RT6}, the conditions on $\min\th_0$ and $\min\th_1$ are fulfilled because $i_2=j_2+1$.
\smallskip

\ref{C7} This is dual to \ref{C6}.
\smallskip

\ref{C8}
This time Table \ref{tab:se} gives $M_{q_1i_1},M_{r_1j_1}\in\{ R,\lambda,\rho,\mu\}$.  Also $\widehat{\alpha\gamma}=\widehat{\beta\gamma}$, as above.  This time we write $f:=\Float(\wh\al,\ga)=\Float(\wh\be,\ga)$, and Lemma \ref{la:A2} gives
\begin{equation}\label{eq:Fag}
\Float(\al,\ga) = f-q+q_1 \AND \Float(\be,\ga) = f-r+r_1.
\end{equation}
For the rest of the proof we write $m_0:=\min\th_0$ and $m_1:=\min\th_1$.

\setcounter{caseco}{0}

\case $q=r$ and $(i,j)\in\th_q$, i.e.~the first option from \ref{C8} holds.  Here it is convenient to consider subcases,
depending on whether or not $q_1=r_1$.

\setcounter{subcaseco}{0}

\subcase $q_1=r_1$.  It follows from \eqref{eq:Fag} that $(i_1,j_1)\in\th_q\sub\th_{q_1}$, so~$M_{q_1i_1}=M_{r_1j_1}$ by Remark~\ref{re:se}\ref{se1}.  But then $(\ba\bc,\bb\bc)\in\si$ via \ref{C2}, \ref{C4}, \ref{C5} or \ref{C8}, as appropriate.  

\subcase $q_1\neq r_1$.  Without loss, we assume that $q_1=1$ and $r_1=0$.  Since $q_1\leq q$, it follows that in fact $q=r=1$.  It also follows from \eqref{eq:Fag} that
\[
(i_1+r_1,j_1+q_1) = (i_1,j_1+1) = (i+k+f,j+k+f) \in \th_q\sub\th_0.
\]
If we can show that $M_{q_1i_1}=M_{r_1j_1}$, then it will again follow that $(\ba\bc,\bb\bc)\in\si$ via \ref{C2}, \ref{C4}, \ref{C5} or \ref{C8}; alongside we will also verify the condition on $m_0$ and $m_1$ required in the \ref{C8} case.  

If $i=j$, then in fact $i_1=j_1+1$, so Remark \ref{re:se}\ref{se2} gives $M_{q_1i_1} = M_{1,j_1+1} = M_{0j_1} = M_{r_1j_1}$.  If $i=j\geq m_1$, then $i_1\geq i\geq m_1$ and $j_1\geq j\geq m_1\geq m_0$.  If $i=j<m_1$, then the presence of $M_{1i}=\mu$ to the left of $m_1$ implies we are in one of types \ref{RT2}, \ref{RT5}, \ref{RT6} or \ref{RT7}; in each of these cases, and combined with $i_1=j_1+1$, it is easy to check that either $i_1<m_1$ and $j_1<m_0$, or else $i_1\geq m_1$ and $j_1\geq m_1$.

If $i\neq j$, then $i,j\geq m_1$ (as $(i,j)\in\th_q=\th_1$), and so $i_1\geq m_1$ and $j_1\geq m_0$ as above; Remark~\ref{re:se}\ref{se3} then gives $M_{q_1i_1} = M_{1i_1} = \mu = M_{0j_1} = M_{r_1j_1}$, as required.

\case $q\neq r$ and $(i+r,j+q)\in\th_0$.  Here we are in the second or third option under \ref{C8}, but we do not need to distinguish these until later in the proof.  Without loss we assume that $q=1$ and $r=0$, so $(i,j+1)\in\th_0$.  
Note also that from $r_1\leq r$ and $q_1\leq q$, we have $r_1=0$ and $q_1\in\{0,1\}$.  Using~\eqref{eq:Fag}, we also have 
\begin{equation}\label{eq:i1j1}
i_1 = i + (k+f-1+q_1) \AND j_1 = (j+1) + (k+f-1).
\end{equation}

\setcounter{subcaseco}{0}

\subcase $q_1=0$.  From \eqref{eq:i1j1} and $i_1\geq i$ we obtain $k+f-1\geq0$.  Using \eqref{eq:i1j1} again, and $(i,j+1)\in\th_0$, it follows that $(i_1,j_1) \in \th_0$.  Since $q_1=r_1=0$, we also have $\al\ga=\wh{\al\ga}=\wh{\be\ga}=\be\ga$, so $(\ba\bc,\bb\bc)\in\si$ via \ref{C1'}.

\subcase $q_1=1$.  This time \eqref{eq:i1j1} gives
\begin{equation}\label{eq:i1j2_2}
(i_1+r_1,j_1+q_1)=(i_1,j_1+1)=(i+(k+f),j+1+(k+f))\in\th_0.
\end{equation}
Since $q_1\neq r_1$ and $\wh{\al\ga}=\wh{\be\ga}$, as shown above, it remains as usual to show that $M_{q_1i_1}=M_{r_1j_1}$, but we must also check the conditions on $m_0$ and $m_1$ required when applying~\ref{C8}.

If $i\geq m_1$ and $j\geq m_0$, then $i_1\geq i\geq m_1$ and $j_1\geq j\geq m_0$, and it then also follows from Remark \ref{re:se}\ref{se3} that $M_{q_1i_1}=M_{r_1j_1}$.  We assume now that $i<m_1$ and $j<m_0$.  The presence of $M_{0j}=\mu$ to the left of $m_0$ implies that we are in one of types \ref{RT2}, \ref{RT5} or \ref{RT7}.  
In \ref{RT2}, $m_0=m_1=\infty$, and $M_{1i_1}=M_{0j_1}=\mu$, completing the proof in this case.

Next consider \ref{RT5}.  First we claim that $i=j+1$.  Indeed, if $j+1<m_0$, then this follows from $(i,j+1)\in\th_0$.  If $j+1=m_0$ (the only other option, as $j<m_0$), then from $(i,j+1)\in\th_0$ it follows that $i\geq m_0=m_1-1$ (as we are in \ref{RT5}); together with $i<m_1$ it follows that $i=m_1-1=m_0=j+1$, as required.  Now that the claim is proved, it follows from \eqref{eq:i1j2_2} that $i_1=j_1+1$.  Checking the matrix in \ref{RT5}, it follows from this that $M_{1i_1}=M_{0j_1}$, and that the conditions on $m_0$ and $m_1$ also hold.

Finally, consider \ref{RT7}.  By the form of the matrix, we must have $i=m_1-1$ and $j=m_0-1$.  From \eqref{eq:i1j1} we have $(i_1,j_1)=(i+(k+f),j+(k+f))$, and the required conditions again quickly follow.

\medskip

Now that we have proved the first assertion of the proposition, we move on to the second.  For this, suppose $\Pair$ is exceptional, and write $\tau:=\cgx(\Pair)$.  We keep the notation of Definition~\ref{def:exc}, including the exceptional row $q=\hgt(M)\geq2$ and the congruence $\thx_q$.  Again $\tau$ is clearly symmetric and reflexive.  For transitivity, suppose $(\ba,\bb),(\bb,\bc)\in\tau$.  It suffices to assume that \ref{C9} is responsible for $(\ba,\bb)\in\tau$.  Since the entries of $M$ at the positions determined by the $\D$-classes containing $\ba$ and $\bb$ are $\A_q$ (keeping Remark \ref{rem:nuN} in mind for $q=2$), and since $(\bb,\bc)\in\tau$, it quickly follows that one of \ref{C3} or \ref{C9} is responsible for $(\bb,\bc)\in\tau$.  But then $(\ba,\bc)\in\tau$ via \ref{C9} or \ref{C3}, respectively.

For compatibility, let $(\ba,\bb)\in\tau$ and $\bc\in\Ptw n$.  It suffices to assume that \ref{C9} is responsible for $(\ba,\bb)\in\tau$, and by symmetry we just need to show that $(\ba\bc,\bb\bc)\in\tau$.  Writing $\ba=(i,\al)$, $\bb=(j,\be)$ and $\bc=(k,\ga)$, we have $(i,j)\in\thx_q\sm\th_q$ and $(\al,\be)\in\nu_{\S_q}\sm\nu_{\A_q}$.  Also write $\ba\bc=(i_1,\al\ga)$ and $\bb\bc=(j_1,\be\ga)$.  By Lemma \ref{la:A4} we have $\Float(\al,\ga)=\Float(\be,\ga)$, and it quickly follows that $(i_1,j_1)\in\thx_q\sm\th_q$.  If $\al\ga\in D_q$, then it follows from Lemma \ref{la:AA2} that $(\al\ga,\be\ga)\in\nu_{\S_q}\sm\nu_{\A_q}$, which shows that $(\ba\bc,\bb\bc)\in\tau$ via \ref{C9}.  So suppose instead that $\al\ga\in D_r$ with $r<q$; as usual, $(\al,\be)\in\H$ implies $\be\ga\in D_r$ as well.  
If $q>2$ then using Table \ref{tab:se} we see that $M_{ri_1}=M_{rj_1}=R$.
If $q=2$ then $\thx_2\subseteq\theta_1\subseteq \theta_r$, Definition~\ref{def:exc} and Table \ref{tab:se} together give
$M_{ri_1}=M_{rj_1}\in\{R,\rho,\lambda,\mu\}$.
The proof that $(\ba\bc,\bb\bc)\in\tau$ now concludes as in the second case of the \ref{C3} above.
\end{proof}

\section{Every congruence is a C-pair congruence}\label{sec:II}

We now turn to the second stage of the proof of our main theorem:
we fix an arbitrary congruence~$\sigma$ on $\Ptw{n}$ and work towards proving that it arises from a C-pair $(\Theta,M)$.
  We begin in Subsection \ref{subsec:bgp} with some basic general properties of $\si$, and then in Subsection \ref{subsec:Cchain} construct the C-chain $\Th$.  Subsections \ref{subsec:I00}--\ref{subsec:Iq0} establish further auxiliary results concerning~$\si$, focussing on its restrictions to individual $\D$-classes.  
  This is then used in Subsection~\ref{subsec:CPAC} to construct the C-matrix $M$.  Subsection \ref{subsec:further} contains yet further technical lemmas concerning restrictions to pairs of $\D$-classes.  Finally, in Subsection \ref{subsec:CongCpair} we complete the proof of the theorem by showing that $\si$ is either the congruence or exceptional congruence associated to the C-pair $(\Th,M)$.

\subsection{Basic general properties of congruences}
\label{subsec:bgp}

In this subsection we prove three basic lemmas that establish certain `translational properties' of the (fixed) congruence $\si$ on $\Ptw n$.  These lemmas, and many results of the subsections to come, will be concerned with the restrictions of $\si$ to the $\D$-classes of $\Ptw n$: $\si\restr_{D_{qi}} = \si\cap (D_{qi}\times D_{qi})$.
Such a restriction can be naturally interpreted as an equivalence on the associated $\D$-class $D_q$ of $\P_n$, by `forgetting' the entries from $\N$.  Formally, we define
\[
\si_{qi} := \bigset{(\al,\be)\in D_q\times D_q}{((i,\al),(i,\be))\in\si} \qquad\text{for $q\in \bnz$ and $i\in\N$.}
\]

\begin{lemma}
\label{la:Ap1}
If $((i,\alpha),(j,\beta))\in\si$ then $((i+k,\alpha),(j+k,\beta))\in\si$ for all $k\in\N$.
\end{lemma}

\begin{proof}
We have $((i+k,\alpha),(j+k,\beta))=((i,\alpha),(j,\beta))\cdot(k,\id)$.
\end{proof}

\begin{lemma}
\label{la:Ap3}
For any $q\in \bnz$ and $i\in\N$, we have ${\si}_{qi}\subseteq{\si}_{q,i+1}$.
\end{lemma}

\begin{proof}
This is a direct consequence of Lemma \ref{la:Ap1}.
\end{proof}

\begin{lemma}
\label{la:Ap2}
Suppose $((i,\alpha),(j,\alpha))\in\si$ for some $i,j\in\N$ and $\alpha\in D_q$.
Then:
\begin{thmenumerate}
\item
\label{it:Ap2.1}
$\bigset{ ((i,\gamma),(j,\gamma))}{ \gamma\in D_q}\subseteq\si$;
\item
\label{it:Ap2.2}
${\si}_{qi}={\si}_{qj}$.
\end{thmenumerate}
\end{lemma}

\begin{proof}
\ref{it:Ap2.1}  Let $\ga\in D_q$.  Since $\ga\rJ\al$, we have $\ga=\eta_1\al\eta_2$ for some $\eta_1,\eta_2\in\P_n$; by Lemma \ref{la:0float} we may assume that $\Float(\eta_1,\al,\eta_2)=0$.  Then $((i,\gamma),(j,\gamma)) = (0,\eta_1) \cdot ((i,\alpha),(j,\alpha)) \cdot (0,\eta_2) \in \si$.

\ref{it:Ap2.2}  This follows from \ref{it:Ap2.1} and transitivity.  
Indeed, if $\ga,\de\in D_q$, then since $\si$ contains both $((i,\ga),(j,\ga))$ and $((i,\de),(j,\de))$, we have $((i,\ga),(i,\de))\in\si \iff ((j,\ga),(j,\de))\in\si$.
\end{proof}

\subsection{The C-chain associated to a congruence}\label{subsec:Cchain}

Recall that we wish to associate a C-pair $(\Th,M)$ to the congruence $\si$ on $\Ptw n$.  We can already define the C-chain $\Th$.

\begin{defn}[\bf The C-chain associated to a congruence]\label{defn:Th}
Given a congruence $\si$ on $\Ptw n$, we define the tuple $\Th= (\th_0,\ldots,\th_n)$, where for each~$q\in\bnz$:
\begin{align*}
\th_q &:= \bigset{ (i,j)\in\N\times\N}{ ((i,\alpha),(j,\alpha))\in\si \text{ for some } \alpha\in D_q}\\
 &\phantom{:}=\bigset{ (i,j)\in\N\times\N}{ ((i,\alpha),(j,\alpha))\in\si \text{ for all } \alpha\in D_q}.
\end{align*}
\end{defn}

The equality of the two relations in the above definition is an immediate consequence of Lemma \ref{la:Ap2}\ref{it:Ap2.1}.

\begin{lemma}
\label{la:siqcong}
The tuple $\Th=(\th_0,\ldots,\th_n)$ in Definition \ref{defn:Th} is a C-chain.
\end{lemma}

\begin{proof}
Clearly each $\th_q$ is an equivalence on $\N$; compatibility follows from Lemma~\ref{la:Ap1}.
Now, suppose $q>0$, and let $(i,j)\in\th_q$. Let $\alpha\in D_q$ and $\eta\in\P_n$ be such that $\al\eta\in D_{q-1}$; by Lemma~\ref{la:0float} we may assume that $\Float(\al,\eta)=0$.  Then $((i,\al\eta),(j,\al\eta))=((i,\al),(j,\al))\cdot(0,\eta)\in\si$, and so $(i,j)\in \th_{q-1}$, proving that $\th_q\subseteq \th_{q-1}$.
\end{proof}

\subsection{The restrictions in row 0}
\label{subsec:I00}

This and the next two
subsections explore consequences of $\si$ containing certain types of pairs.  The guiding principle is that we are aiming to understand the possible restrictions ${\si}_{qi}$.

We begin with $q=0$, proving some results concerning the behaviour of $\si$ on the ideal~$I_{00}$.  
We will make frequent use of the partition $\omega:=\begin{partn}{1} {\ \! \bn\ \!}\\ {\ \! \bn\ \!} \end{partn}$, which has the single block $\bn\cup\bn'$.  Note that for any $\al\in D_1$ and $\be\in D_0$, we have
\[
\al\om\al = \al \COMMA
\om\al\om = \om \COMMA
\be\om\be = \be \COMMA
\om\be\om = \wh\om = \begin{partn}{1} {\ \! \bn\ \!}\\ \hhline{-} {\ \! \bn\ \!}\end{partn}.
\]
Further, for any $\ga\in\P_n$ we have $\Float(\om,\ga)=\Float(\ga,\om)=\Float(\om,\ga,\om)=0$.  We will typically use these facts without explicit comment.

\begin{lemma}
\label{la:C1}
If $\si\cap (D_{0i}\times D_{0j})\neq\emptyset$ then $(i,j)\in\th_0$.
\end{lemma}

\begin{proof}
For any $(\ba,\bb)\in \si\cap (D_{0i}\times D_{0j})$ we have $((i,\widehat{\omega}),(j,\widehat{\omega})) = (0,\om) \cdot (\ba,\bb) \cdot (0,\om) \in\si$.
\end{proof}

\begin{lemma}
\label{la:C2}
If ${\si}_{0i}\neq \Delta_{D_0}$ then $(i,i+1)\in\th_0$.
\end{lemma}

\begin{proof}
Suppose $(\al,\be)\in\si_{0i}$ with $\al\neq \be$.
Without loss we may assume that $\be$ has an upper block $A$ that does not contain (and is not equal to) any upper blocks of $\al$.
Let $\eta:=\begin{partn}{2} \bn & \\ \hhline{~|-} \bn\setminus A & A\end{partn}$.
Then $\Float(\eta,\alpha)=0$ and $\Float(\eta,\beta)=1$. Hence $((i,\eta\alpha),(i+1,\eta\beta)) = (0,\eta) \cdot ((i,\alpha),(i,\beta)) \in\si$, and the result follows by Lemma \ref{la:C1}.
\end{proof}

We now bring the projection $\Proj\si$ of $\si$ to $\P_n$ into play; see Definition \ref{defn:Proj}.  We also recall the congruences on $\P_n$, as listed in Theorem \ref{thm:CongPn} and depicted in Figure \ref{fig:CongPn}.  Note that $\Proj\si\cap R_0$ is one of $\De_{\P_n}$, $\lam_0$, $\rho_0$ or $R_0$.

\begin{lemma}
\label{la:C3}
If $\Proj\si\cap R_0\neq \De_{\P_n}$, then $\th_0=(m,m+1)^\sharp$ for some $m\in\N$, and we have
\[
\bigset{ ((i,\al),(j,\be))}{ i,j\geq m,\ (\al,\be)\in\Proj\si\restr_{D_0}}\subseteq \si.
\]
\end{lemma}

\begin{proof}
Let $(\ga,\de)\in\Proj\si\restr_{D_0}$ with $\ga\neq \de$.  By definition of~$\Proj\si$, we have ${((k,\ga),(l,\de))\in\si}$ for some $k,l\in\N$, and Lemma \ref{la:C1} then gives $(k,l)\in\th_0$.  Thus, $(k,\de)\rsi(l,\de)\rsi(k,\ga)$.  Since $\ga\neq \de$, it follows from Lemma \ref{la:C2} that $(k,k+1)\in\th_0$.  But this means that $\per\th_0=1$, and so $\th_0=(m,m+1)^\sharp$ for some $m\in\N$.

Now let $i,j\geq m$ and $(\al,\be)\in\Proj\si\restr_{D_0}$ be arbitrary, so that $((g,\al),(h,\be))\in\si$ for some $g,h\in\N$.  If $\al=\be$, then from $(i,j)\in\th_0$ we have $(i,\al)\rsi(j,\al)=(j,\be)$.  Now suppose $\al\neq \be$.  As above, we have $(g,g+1)\in\th_0$ so that $g\geq m$, and similarly $h\geq m$.  But then $(i,g),(h,j)\in\th_0$, and so $(i,\al)\rsi(g,\al)\rsi(h,\be)\rsi(j,\be)$.
\end{proof}

\begin{lemma}
\label{la:C4a'}
\ben
\item \label{C4a'1} If $\Proj\si\sub\mu_{\S_2}$, then $\si_{0i}=\De_{D_0}$ for all $i\in\N$.
\item \label{C4a'2} If $\Proj\si\not\sub\mu_{\S_2}$, then $\th_0=(m,m+1)^\sharp$ for some $m\in\N$, and
\[
\si_{0i} = \begin{cases}
\De_{D_0} &\text{if $i<m$}\\
\Proj\si\restr_{D_0} &\text{if $i\geq m$.}
\end{cases}
\]
\een
\end{lemma}

\pf
\ref{C4a'1}  This follows immediately from the fact that $\mu_{\S_2}\restr_{D_0} = \De_{D_0}$.

\ref{C4a'2}  Using Theorem \ref{thm:CongPn} and Figure \ref{fig:CongPn}, the condition $\Proj\si\not\sub\mu_{\S_2}$ tells us that $\Proj\si\cap R_0\neq \De_{\P_n}$.  Lemma~\ref{la:C3} gives $\th_0=(m,m+1)^\sharp$ for some $m\in\N$.  It remains to check the assertion regarding the $\si_{0i}$.  The $i<m$ case follows immediately from Lemma~\ref{la:C2}.  For $i\geq m$, Lemma \ref{la:C3} gives $\Proj\si\restr_{D_0}\sub\si_{0i}$, and the reverse inclusion follows quickly from the definitions.
\epf

We conclude this subsection by listing the possible restrictions of $\si$ to the $\D$-classes in the bottom row of $\Ptw{n}$.

\begin{lemma}
\label{la:C5}
For a congruence $\si$ on $\Ptw{n}$, and for any $i\in\N$, the relation
${\si}_{0i}$ is one of $\Delta_{D_0}$, $\lambda_0\restr_{D_0}$, $\rho_0\restr_{D_0}$ or $\nabla_{D_0}$.
\end{lemma}

\begin{proof}
This follows directly from Lemma \ref{la:C4a'} and Theorem \ref{thm:CongPn}.
\end{proof}

\subsection{The restrictions in row 1}
\label{subsec:I10}

As we glimpsed in Example \ref{ex:weirdproj}, and saw in more detail in \ref{C8}, the behaviour of a congruence on the ideal $I_{10}$ can be rather complex.
It will be one of the recurring motifs in this paper that rows~$0$ and~$1$ of $\Ptw{n}$ and their related pairs need to receive special treatment.
This subsection establishes technical tools for doing this.  We continue to use the notation $\omega=\begin{partn}{1} {\ \! \bn\ \!}\\ {\ \! \bn\ \!} \end{partn}$.  

We begin with a simple general fact that will be used often.

\begin{lemma}\label{la:(i+1,j)}
If $\si\cap (D_{0i}\times D_{rj})\neq\emptyset$ for some $r\geq1$, then $(i+1,j)\in\th_0$.
\end{lemma}

\pf
If $((i,\alpha),(j,\beta))\in \si\cap (D_{0i}\times D_{rj})$, then $((i+1,\al\wh\om),(j,\be\wh\om)) = ((i,\al),(j,\be))\cdot(0,\wh\om)\in\si$, with $\al\wh\om,\be\wh\om\in D_0$, so that $(i+1,j)\in\th_0$ by Lemma \ref{la:C1}.
\epf

Lemma \ref{la:C1} showed that any $\si$-relationship between $\D$-classes $D_{0i}$ and $D_{0j}$ implies relationships between all elements of these $\D$-classes with equal underlying partitions.  The next lemma does the same for relationships within row $1$, and the following one gives the analogous result for relationships between rows $0$ and $1$.

\begin{lemma}
\label{la:C1'}
If $\si\cap (D_{1i}\times D_{1j})\neq\emptyset$ then $(i,j)\in\th_1$.
\end{lemma}

\begin{proof}
If $(\ba,\bb)\in \si\cap (D_{1i}\times D_{1j})$ then $((i,\omega),(j,\omega)) = (0,\om) \cdot (\ba,\bb) \cdot (0,\om) \in\si$.
\end{proof}

\begin{lemma}
\label{la:D1}
If $\si\cap (D_{0i}\times D_{1j})\neq\emptyset$ then 
\[
\bigset{ ((i,\widehat{\gamma}),(j,\gamma))}{ \gamma\in D_1}\cup
\bigset{ ((j,\gamma),(j,\delta))}{ \gamma,\delta\in D_1,\ \widehat{\gamma}=\widehat{\delta}}\subseteq\si.
\]
\end{lemma}

\begin{proof}
Fix some $((i,\alpha),(j,\beta))\in \si\cap (D_{0i}\times D_{1j})$, and let $\ga\in D_1$ be arbitrary.  Noting that $\wh\ga = (\ga\om)\al(\om\ga)$ and ${\ga = (\ga\om)\be(\om\ga)}$, with $\Float(\ga\om,\al,\om\ga)=\Float(\ga\om,\be,\om\ga)=0$, we have
\[
((i,\wh\ga),(j,\ga)) = (0,\ga\om)\cdot((i,\alpha),(j,\beta))\cdot(0,\om\ga)\in\si.
\]
This shows the inclusion in $\si$ of the first set in the left-hand side union, and the second follows by transitivity.
\end{proof}

The next lemma refers to the congruences $\lam_1$ and $\rho_1$ on $\P_n$.

\begin{lemma}
\label{la:D3}
\ben
\item If $\si_{1i}\not\sub\rho_1$, then $\lam_1\restr_{D_1}\subseteq{\si}_{1i}$.
\item If $\si_{1i}\not\sub\lam_1$, then $\rho_1\restr_{D_1}\subseteq{\si}_{1i}$.
\een
\end{lemma}

\begin{proof}
Only the first statement needs to be proved, as the second is dual.  To do so, fix some ${(\al,\be)\in\si_{1i}\sm\rho_1}$, so $\al,\be\in D_1$ and $\ker\al\neq \ker\be$.  It follows from Lemma \ref{la:A3} that ${\si\cap(D_{0i}\times D_{1i})\neq \emptyset}$, and then from Lemma~\ref{la:D1} that $\bigset{ ((i,\ga),(i,\widehat{\ga}))}{ \ga\in D_1}\subseteq \si$.
But then $(i,\wh\al)\rsi(i,\al)\rsi(i,\be)\rsi(i,\wh\be)$ with $\ker\wh\al\neq \ker\wh\be$, and hence
$\lambda_0\restr_{D_0}\subseteq{\si}_{0i}$ by Lemma \ref{la:C5}.
Now for any $(\gamma,\delta)\in\lambda_1\restr_{D_1}$ we have $\widehat{\gamma}\rL\widehat{\delta}$, i.e.~$(\widehat{\gamma},\widehat{\delta})\in\lambda_0\restr_{D_0}$, and hence
$(i,\gamma)\rsi (i,\widehat{\gamma})\rsi (i,\widehat{\delta})\rsi (i,\delta)$, completing the proof.
\end{proof}

One way in which $\si\restr_{I_{10}}$ may be unusual is that the relations ${\si}_{1i}$ are not necessarily restrictions of congruences of $\P_n$ to $D_1$. Two additional relations that may occur will play an important role in what follows:

\begin{defn}
\label{def:muud}
The relations $\muup$ and $\mudown$ on $D_1$ are defined by:
\[
\muup:=\bigset{ (\alpha,\beta)\in D_1\times D_1}{ \widehat{\alpha}=\widehat{\beta},\ \alpha\rR\beta}
\ANd
\mudown:=\bigset{ (\alpha,\beta)\in D_1\times D_1}{ \widehat{\alpha}=\widehat{\beta},\ \alpha\rL\beta}.
\]
\end{defn}

As the notation suggests, these two relations are closely tied to their counterpart labels in C-matrices; see Subsection \ref{subsec:Cpairs}.

\begin{lemma}
\label{la:D2}
\ben
\item If $\si_{1i}\not\sub\muup$, then $\mudown\subseteq {\si}_{1i}$.
\item If $\si_{1i}\not\sub\mudown$, then $\muup\subseteq {\si}_{1i}$.
\een
\end{lemma}

\begin{proof}
Again, only the first statement needs proof.  If $\si_{1i}$ is not contained in one of $\rho_1$ or $\lam_1$, then by Lemma \ref{la:D3}, $\si_{1i}$ contains one of $\lam_1\restr_{D_1}$ or $\rho_1\restr_{D_1}$, both of which contain $\mudown$.  Thus, we now assume that $\si_{1i}\sub\rho_1\cap\lam_1=\mu_1$.
Fix some $(\al,\be)\in\si_{1i}\sm\muup$.  Noting then that ${(\al,\be)\in\mu_1\sm\muup}$, we have $\wh\al=\wh\be$ and $(\al,\be)\not\in\R$; consequently, $\ker\alpha=\ker\beta$ and ${\dom\alpha\neq\dom\beta}$.
By post-multiplying $((i,\alpha),(i,\beta))$ by $(0,\om)$,
we may assume that
${\alpha=\begin{partn}{4} A_1 & A_2&\dots & A_k\\ \hhline{~|-|-|-} \bn & \multicolumn{3}{c}{}\end{partn}}$ and
${\beta=\begin{partn}{4} A_k & A_1&\dots & A_{k-1}\\ \hhline{~|-|-|-} \bn & \multicolumn{3}{c}{}\end{partn}}$.
Now let $(\gamma,\delta)\in\mudown$ be arbitrary.  So $\ga,\de\in D_1$, $\gamma\rL\delta$ and $\widehat{\gamma}=\widehat{\delta}$, and we need to show that $((i,\gamma),(i,\delta))\in\si$.
If $\gamma=\delta$ there is nothing to prove, so suppose $\gamma\neq \delta$.
We may then write
$\gamma=\begin{partn}{4} B_1 & B_2&\dots & B_l\\ \hhline{~|-|-|-} C_1&C_2&\dots&C_m \end{partn}$ and
$\delta=\begin{partn}{4} B_l & B_1&\dots & B_{l-1}\\ \hhline{~|-|-|-} C_1&C_2&\dots&C_m \end{partn}$.
Then with
$\eta_1:=\begin{partn}{5} B_1 & B_l & B_2&\dots& B_{l-1}\\ \hhline{~|~|-|-|-} A_1 & A_2\cup\dots \cup A_k&\multicolumn{3}{c}{}\end{partn}$ and
$\eta_2:=\begin{partn}{4} \bn & \multicolumn{3}{c}{} \\ \hhline{~|-|-|-} C_1&C_2&\dots&C_m\end{partn}$, we have
$((i,\gamma),(i,\delta)) = (0,\eta_1) \cdot ((i,\alpha),(i,\beta)) \cdot (0,\eta_2) \in\si$, as required.
\end{proof}

We can now describe all possible restrictions of $\si$ to $\D$-classes in row 1.

\begin{lemma}
\label{la:D4}
For a congruence $\si$ on $\Ptw{n}$, and for any $i\in\N$, the relation
$\si_{1i}$ is one of  $\Delta_{D_1}$, $\muup$, $\mudown$, $\mu_1\restr_{D_1}$, $\rho_1\restr_{D_1}$, $\lambda_1\restr_{D_1}$ or $\nabla_{D_1}$.
\end{lemma}

\pf
To simplify the proof, we write $\tau=\si_{1i}$, $\lam=\lam_1\restr_{D_1}$, $\rho=\rho_1\restr_{D_1}$, $\mu=\mu_1\restr_{D_1}$, $\De=\De_{D_1}$ and $\nab=\nab_{D_1}$.  
The following argument is structured around the inclusion diagram of these relations:
\begin{center}
\begin{tikzpicture}[scale=.9]
\node (Delta) at (0,0) {$\De$};
\node (R) at (0,4) {$\nab$};
\node (mudown) at (-1,1) {$\mudown$};
\node (muup) at (1,1) {$\muup$};
\node (mu) at (0,2) {$\mu$};
\node (lambda) at (-1,3) {$\lambda$};
\node (rho) at (1,3) {$\rho$};
\draw
(Delta)--(mudown)--(mu)--(lambda)--(R)
(Delta)--(muup)--(mu)--(rho)--(R)
;
\end{tikzpicture}
\end{center}
\setcounter{caseco}{0}
\case $\tau\not\sub\lam$ and $\tau\not\sub\rho$.  Using Lemma~\ref{la:D3}, these respectively give $\rho\sub\tau$ and $\lam\sub\tau$.  It then follows that $\nab=\lam\vee\rho\sub\tau$, so $\tau=\nab$.

\case $\tau\not\sub\lam$ and $\tau\sub\rho$.  From the former, Lemma \ref{la:D3} gives $\rho\sub\tau$, so $\tau=\rho$.

\case $\tau\sub\lam$ and $\tau\not\sub\rho$.  By symmetry, this time we have $\tau=\lam$.

\case $\tau\sub\lam$ and $\tau\sub\rho$.  Here we have $\tau\sub\lam\cap\rho=\mu$.  As above, we now consider subcases according to whether $\tau$ is contained in $\muup$ and/or~$\mudown$.  We use Lemma \ref{la:D2} in place of Lemma~\ref{la:D3}, and also $\muup\vee\mudown=\mu$ and $\muup\cap\mudown=\De$, to deduce that $\tau$ is one of $\mu$, $\muup$, $\mudown$ or $\De$.
\epf

We conclude this subsection with the following important corollary:

\begin{lemma}
\label{la:D5}
If ${\si}_{1i}\neq\Delta_{D_1}$ then
$\bigset{((j,\widehat{\gamma}),(j+1,\gamma))}{ \gamma\in D_1}\subseteq \si$ for all $j\geq i$.
\end{lemma}

\begin{proof}
By Lemma \ref{la:D4} we may assume without loss that $\muup\subseteq {\si}_{1i}$.  Hence,
\[
\left(\left(i,\begin{partn}{4} \bn &\multicolumn{3}{c}{} \\ \hhline{~|-|-|-} 1&2&\dots&n\end{partn}\right),
\left(i,\begin{partn}{4} \bn &\multicolumn{3}{c}{} \\ \hhline{~|-|-|-} n&1&\dots&n-1\end{partn}\right)\right)\in\si .
\]
Post-multiplying this pair by
$\left(j-i,\begin{partn}{2} 1& 2,\dots,n\\ \hhline{~|-} \bn & \end{partn}\right)$ we obtain
$((j+1,\om),(j,\wh\om))\in\si$.
The result then follows by Lemma \ref{la:D1}.
\end{proof}

\subsection[The restrictions in rows $q\geq2$]{The restrictions in rows \boldmath{$q\geq2$}}
\label{subsec:Iq0}

Next we examine the behaviour of $\si$ on rows $q\geq 2$.  The following sequence of lemmas can be viewed as working the `separation' Lemma \ref{la:sep} into the context of $\Ptw{n}$.  In the next lemma we do not assume that $q\geq2$.

\begin{lemma}
\label{la:B3}
If $\si\cap (D_{qi}\times D_{rj})\neq\emptyset$, then for every $\gamma \in I_q$ there exist 
$\delta\in I_r$  and $l\geq j$ such that $((i,\gamma),(l,\delta))\in\si$.
\end{lemma}

\begin{proof}
Let $((i,\alpha),(j,\beta))\in \si\cap (D_{qi}\times D_{rj})$.
Using Lemma \ref{la:0float}, let $\eta_1,\eta_2\in\P_n$ be such that $\eta_1\alpha\eta_2=\gamma$ with
$\Float(\eta_1,\alpha,\eta_2)=0$, and then let $\delta:=\eta_1\beta\eta_2\in I_r$ and $l:=j+\Float(\eta_1,\beta,\eta_2)\geq j$.
Then $((i,\ga),(l,\de)) = (0,\eta_1) \cdot ((i,\al),(j,\be)) \cdot (0,\eta_2)\in \si$.
\end{proof}

\begin{lemma}
\label{la:B4}
If $\si\cap (D_{qi}\times D_{rj})\neq\emptyset$, where $q>r$ and $q\geq 2$, then for every $\gamma\in I_q\setminus D_0$ there exist $l\geq j$ and $\delta\in\P_n$ such that
$\rank\delta<\rank\gamma$ and $((i,\gamma),(l,\delta))\in\si$.
\end{lemma}

\begin{proof}
Let $((i,\alpha),(j,\beta))\in\si\cap (D_{qi}\times D_{rj})$, and write $s:=\rank\gamma$. If $s=q$ the assertion follows from Lemma \ref{la:B3}.
Now suppose $s<q$.
Write 
$\alpha=\begin{partn}{6} A_1&\dots&A_q&C_1&\dots&C_u\\ \hhline{~|~|~|-|-|-} B_1&\dots&B_q&E_1&\dots&E_v\end{partn}$,
and pick $a_t\in A_t$ ($t=1,\dots,q$).
Since $q\geq 2$ and $q>r$, reordering the transversals of $\alpha$ if necessary, we may assume that one of the following holds:
$a_1\not\in\dom\beta$, or else ${a_1,a_2\in\dom\beta}$ and $(a_1,a_2)\in\ker\beta$.
In either case let $\eta:=\begin{partn}{4} a_1&\dots&a_{s-1}&a_s\\ a_1&\dots&a_{s-1}&\{a_s\}\cup C_1\cup\dots\cup C_u\end{partn}$, again with unlisted elements being singletons. Now for $\gamma':=\eta\alpha$, $\delta':=\eta\beta$ and $l':=j+\Float(\eta,\be)\geq j$, we have $\rank\gamma'=s>\rank\delta'$ and $\Float(\eta,\alpha)=0$, so that $((i,\gamma'),(l',\delta')) = (0,\eta) \cdot ((i,\al),(j,\be))\in\si$.
Since $\rank\gamma=\rank\gamma'$, another application of Lemma \ref{la:B3} now implies the existence of $\delta$ and $l$ as specified.
\end{proof}

\begin{lemma}
\label{la:Z1}
If $\si\cap (D_{qi}\times D_{rj})\neq\emptyset$, where $q>r$ and $q\geq 2$, then $\si\cap (D_{qi}\times D_{sk})\neq\emptyset$ for some $s\leq r$ and $k\geq i$.
\end{lemma}

\pf
If $j\geq i$ we are already done, so suppose $j<i$, and fix $((i,\al),(j,\be))\in\si\cap(D_{qi}\times D_{rj})$.  Write $\be = \begin{partn}{6} A_1&\dots&A_r&C_1&\dots&C_u\\ \hhline{~|~|~|-|-|-} B_1&\dots&B_r&E_1&\dots&E_v\end{partn}$, and set $\eta := \begin{partn}{7} B_1&\dots&B_{r-1}&B_r\cup E_1\cup\cdots\cup E_v&\multicolumn{3}{c}{}\\ \hhline{~|~|~|~|-|-|-} B_1&\dots&B_{r-1}&B_r& E_1&\cdots& E_v\end{partn}$.  Then $\be=\be\eta$ and $\Float(\be,\eta)=0$.  Then with $k:=i+\Float(\al,\eta)$, we have $((j,\be),(k,\al\eta))=((j,\be),(i,\al))\cdot(0,\eta)\in\si$, and so $((i,\al),(k,\al\eta))\in\si$ by transitivity.  But $k\geq i$, and $s:=\rank(\al\eta)\leq\rank\eta=r$.
\epf

\begin{lemma}
\label{la:Z2}
If $\si\cap (D_{qi}\times D_{rj})\neq\emptyset$, where $q>r$ and $q\geq 2$, then $\si\cap (D_{qi}\times D_{0k})\neq\emptyset$ for some $k\in\N$.
\end{lemma}

\pf
If $r=0$ then there is nothing to show, so suppose instead that $0<r<q$.  By induction, it suffices to show that $\si\cap (D_{qi}\times D_{sj'})\neq\emptyset$ for some $j'\in\N$ and some $s<r$.  
By Lemma \ref{la:Z1} we may assume that $j\geq i$, and we fix some $((i,\al),(j,\be))\in\si\cap(D_{qi}\times D_{rj})$.  
Now, Lemma~\ref{la:B4} (with~$\ga=\be$) gives $((i,\be),(l,\de))\in\si$ for some $l\geq j$ and $\de\in I_{r-1}$.  Since $j\geq i$, it then follows from Lemma \ref{la:Ap1} that $((j,\be),(l+j-i,\de))\in\si$, and then by transitivity that $((i,\al),(l+j-i,\de))\in\si$, as required.
\epf

The next two statements refer to the (possibly empty) ideal $I(\si)$ of $\Ptw n$ from Definition \ref{defn:Isi}.

\begin{lemma}
\label{la:B5'}
If $\si\cap (D_{qi}\times D_{rj})\neq\emptyset$, where $q>r$ and $q\geq 2$, then $I_{qi}\cup I_{rj}\subseteq I(\si)$.
\end{lemma}

\pf
Since the ideal $I(\si)$ is a $\si$-class, it suffices to show that $I_{qi}\sub I(\si)$, since then also $I_{rj}\sub I(\si)$.  
By Lemma~\ref{la:Z2} we have $\si\cap (D_{qi}\times D_{0k})\neq\emptyset$ for some $k\in\N$, and by Lemma \ref{la:Z1} we may assume that $k\geq i$.  
Again, it suffices to show that $I_{0k}\sub I(\si)$.  

By Lemma \ref{la:(i+1,j)}, we have $(i,k+1)\in\th_0$; since $i<k+1$ (as $i\leq k$), it follows that ${i\geq m:=\min \th_0}$.  Since $\si\cap (D_{qi}\times D_{0k})\neq\emptyset$, it follows that $\Proj\si\cap (D_q\times D_0)\neq\emptyset$, so Theorem \ref{thm:CongPn} (see~Figure \ref{fig:CongPn}) gives $\Proj\si\supseteq R_0$, and so $\Proj\si\restr_{D_0}=\nab_{D_0}$.  It then follows from Lemma~\ref{la:C3} that $R_{I_{0m}}\sub\si$, i.e.~$I_{0m}\subseteq I(\si)$, and we are done since $k\geq i\geq m$.
\epf

\begin{lemma}
\label{la:B5}
If $((i,\alpha), (j,\beta))\in \si\cap (D_{qi}\times D_{rj})$, where $q\geq r$, $q\geq 2$ and $(\alpha,\beta)\not\in\H$, then $I_{qi}\cup I_{rj}\subseteq I(\si)$.
\end{lemma}

\pf
In light of Lemma \ref{la:B5'}, it suffices to consider the case in which $q=r$, and again it suffices to show that $I(\si)$ contains either $I_{qi}$ or $I_{qj}$.  By Lemma \ref{la:sep}\ref{it:sep2}, one of $\si\cap (D_{qi}\times D_{sk})$ or $\si\cap (D_{qj}\times D_{sk})$ is non-empty for some $s<q$ and $k\in\N$.  The result then follows from another application of Lemma \ref{la:B5'}.
\epf

The next two statements refer to the relations $\nu_N$ defined just before Theorem \ref{thm:CongPn}; recall in particular that $\nu_{\{\id_q\}}=\De_{D_q}$.

\begin{lemma}
\label{la:B14}
If ${\si}_{qi}\subseteq\H$ where $q\geq2$, then ${\si}_{qi}=\nu_N$ for some $N\unlhd \S_q$.
\end{lemma}

\begin{proof}
Bearing in mind the classification of congruences on $\P_n$ from Theorem \ref{thm:CongPn}, it is sufficient to show that
${\si}_{qi}$ is the restriction to $D_q$ of a congruence on $\P_n$.
To prove this, it is in turn sufficient to prove that for $(\alpha,\beta)\in {\si}_{qi}$ and $\gamma\in\P_n$
either $(\alpha\gamma,\beta\gamma)\in{\si}_{qi}$ or $\alpha\gamma,\beta\gamma\in I_{q-1}$.
Since $(\al,\be)\in\H$, it follows as usual that either $\alpha\gamma,\beta\gamma\in D_q$ or $\alpha\gamma,\beta\gamma\in I_{q-1}$, and in the latter case we are done.
So suppose $\alpha\gamma,\beta\gamma\in D_q$.
By Lemma \ref{la:0float}, there exists $\gamma'\in\P_n$ such that $\alpha\gamma=\alpha\gamma'$ and
$\Float(\alpha,\gamma')=0$.
From $\alpha\rL\beta$ it follows that $\beta\gamma=\beta\gamma'$, and also that $\Float(\beta,\gamma')=0$
by Lemma~\ref{la:A4}.
So $((i,\alpha\gamma),(i,\beta\gamma)) = ((i,\alpha),(i,\beta)) \cdot (0,\gamma') \in\si$, and hence $(\alpha\gamma,\beta\gamma)\in{\si}_{qi}$.
\end{proof}

We can now describe all possible restrictions of $\si$ to $\D$-classes in rows $q\geq2$.

\begin{lemma}
\label{la:B8}
For a congruence $\si$ on $\Ptw{n}$, and for any $q\in\{2,\ldots,n\}$ and $i\in\N$, the relation~${\si}_{qi}$ is either $\nabla_{D_q}$ or else $\nu_N$ for some $N\unlhd \S_q$.  
Furthermore, if $\si_{qi}=\nab_{D_q}$ and $q\neq n$, then $D_{qi}\subseteq I(\si)$.
\end{lemma}

\begin{proof}
This follows by combining Lemmas \ref{la:B5} and \ref{la:B14}, keeping in mind that the only $q$ for which $D_q$ is an $\H$-class is $q=n$.
\end{proof}

\subsection{The C-pair associated to a congruence}
\label{subsec:CPAC}

We are now ready to define the C-matrix $M$ associated with $\sigma$, and then prove that $\Pair=(\Theta,M)$ is a C-pair.

To define $M$ we proceed as follows.  For each $\D$-class $D_{qi}$, we refer back to Lemmas~\ref{la:C5},~\ref{la:D4} and~\ref{la:B8}, which list all possible restrictions $\si_{qi}$; in almost all cases, this is enough to uniquely determine the entry $M_{qi}$ in the obvious way, with two ambiguities that need to be resolved:
\bit
\item If $\si_{0i}=\De_{D_0}=\mu_0\restr_{D_0}$, then $M_{0i}$ is either $\mu$ or $\De$, depending  on whether there are $\si$-relationships between elements of $D_{0i}$ and those of some $D_{1j}$.
\item If $\si_{ni}=\nab_{D_n}=\nu_{\S_n}$, then $M_{ni}$ is either $R$ or the $N$-symbol $\S_n$, depending  on whether $D_{ni}\sub I(\si)$.
\eit
More formally:

\begin{defn}[\bf The C-matrix associated to a congruence]\label{defn:M}
Given a congruence $\si$ on~$\Ptw n$, we define the matrix $M=(M_{qi})_{\bnz\times\N}$ according to the rules given in Table \ref{tab:M}.
\end{defn}

\begin{table}[ht]
\begin{center}
$
\begin{array}{|c|c|c|c|} 
\hline
\text{\boldmath{$q$}} & \text{\boldmath{$M_{qi}$}}& \text{\boldmath{$\si_{qi}$}}& \text{\bf Ambiguity resolution}
\\  \hhline{|=|=|=|=|}
\multirow{3}{*}{$q\geq2$} & \De & \De_{D_q} & 
\\  \hhline{~|-|-|-|}
 & N & \nu_N & D_{qi}\not\subseteq I(\si)
\\ \hhline{~|-|-|-|}
 & R & \nab_{D_q} & D_{qi}\subseteq I(\si) 
\\  \hhline{|=|=|=|=|}
\multirow{7}{*}{$q=1$} & \De & \De_{D_1} & 
\\  \hhline{~|-|-|-|}
 & \muup & \muup & 
\\  \hhline{~|-|-|-|}
 & \mudown & \mudown & 
\\  \hhline{~|-|-|-|}
 & \mu & \mu_1\restr_{D_1} & 
\\  \hhline{~|-|-|-|}
 & \lam &\lam_1\restr_{D_1} & 
\\  \hhline{~|-|-|-|}
 & \rho & \rho_1\restr_{D_1} & 
\\ \hhline{~|-|-|-|}
 & R & \nab_{D_q} & 
\\  \hhline{|=|=|=|=|}
\multirow{5}{*}{$q=0$} & \De & \De_{D_0} & \text{$\si\cap(D_{0i}\times D_{1j})=\emptyset$ ($\forall j\in\N$})
\\  \hhline{~|-|-|-|}
 & \mu & \De_{D_0} & \text{$\si\cap(D_{0i}\times D_{1j})\neq\emptyset$ ($\exists j\in\N$})
\\  \hhline{~|-|-|-|}
 & \lam &\lam_0\restr_{D_0} & 
\\  \hhline{~|-|-|-|}
 & \rho & \rho_0\restr_{D_0} & 
\\ \hhline{~|-|-|-|}
 & R & \nab_{D_q} & 
\\ \hline
\end{array}
$
\caption{The specification of the C-matrix $M=(M_{qi})_{\bnz\times\N}$ associated to the congruence $\si$ on $\Ptw n$.}
\label{tab:M}
\end{center}
\end{table}

We remark that when $M_{qi}=N$, the clause $D_{qi}\nsubseteq I(\si)$ in Table \ref{tab:M} always follows from ${\si}_{qi}=\nu_N$ except when $N=\S_n$ as discussed above; similarly, $D_{qi}\sub I(\si)$ is only needed for $M_{qi}=R$ when~$q=n$.

The rest of this subsection is devoted to showing that $\Pair:=(\Th,M)$ is indeed a C-pair, which will be accomplished in Lemma~\ref{la:rowq}.  
To get there, we proceed with a host of auxiliary results about $M$.  They are mostly concerned with what entries in $M$ can occur below and to the right of an entry, and with the interplay between $M$ and $\Th$.  

\begin{lemma}
\label{la:E1}
For any $q\in \bnz$, all entries $M_{qi}$ with $i\geq\min\th_q$ are equal.
\end{lemma}

\begin{proof}
If $\th_q=\Delta_\N$ the statement is vacuous, so suppose $d:=\per\th_q<\infty$.  We aim to prove that $M_{qi}=M_{q,i+1}$ for $i\geq\min\th_q$.  Note that a matrix entry is entirely determined by the restriction of $\si$ to the corresponding $\D$-class and (in some cases) the presence or absence of $\si$-relationships between that $\D$-class and another one in a different row.  By Lemmas \ref{la:Ap3} and \ref{la:Ap2} we have ${\si}_{qi}\subseteq{\si}_{q,i+1}\subseteq\dots\subseteq{\si}_{q,i+d}={\si}_{qi}$, and so ${\si}_{qi}={\si}_{q,i+1}$.  To complete the proof we must show that the following are equivalent:
\ben
\item \label{qirjsk1} $\si\cap(D_{qi}\times D_{rj})\neq\emptyset$ for some $j\in\N$ and $r\neq q$,
\item \label{qirjsk2} $\si\cap(D_{q,i+1}\times D_{rj})\neq \emptyset$ for some $j\in\N$ and $r\neq q$.
\een
For \ref{qirjsk1}$\implies$\ref{qirjsk2}, we use Lemma \ref{la:Ap1}.  For \ref{qirjsk2}$\implies$\ref{qirjsk1}, fix ${((i+1,\al),(j,\be))\in\si\cap(D_{q,i+1}\times D_{rj})}$, where $r\neq q$.  Then from $(i,i+d)\in\th_q$ and Lemma \ref{la:Ap1} we have $(i,\al) \rsi (i+d,\al) \rsi (j+d-1,\be)$.
\end{proof}

Now we look at the entries equal to $\De$, $R$ and $N\unlhd \S_q$:

\begin{lemma}
\label{la:E3}
If $M_{qi}=R$ then $i\geq\min\th_q$, $\per\th_q=1$, and $M_{rj}=R$ whenever $r\leq q$ and $j\geq i$.
\end{lemma}

\begin{proof}
$M_{qi}=R$ means $D_{qi}\subseteq I(\si)$, so that $D_{rj}\subseteq I(\si)$ whenever $r\leq q$ and $j\geq i$;
in particular $((i,\alpha),(i+1,\alpha))\in\si$ for any $\alpha\in D_q$, and all three statements follow.
\end{proof}

\begin{lemma}
\label{la:E4}
If $M_{qi}=N$ where $\{\id_q\}\neq N\unlhd \S_q$, then 
\ben
\item \label{E4.1} $M_{q-1,i}$ is one of $R$, $\mu$, $\rho$ or $\lambda$,
\item \label{E4.2} $M_{q,i+1}$ is either $R$ or some $N'\normal\S_q$ with $N\leq N'$.
\een
\end{lemma}

\begin{proof}
\ref{E4.1}  If $q>2$ then combining Lemmas \ref{la:sep}\ref{it:sep3} and \ref{la:B5} we have $D_{q-1,i}\subseteq I(\si)$, and so $M_{q-1,i}=R$.

For $q=2$, fix some $((i,\al),(i,\be))\in\si$ where $(\al,\be)\in\nu_{\S_2}$ and $\al\neq \be$.  By Lemmas \ref{la:sep}\ref{it:sep3}, \ref{la:0float} and \ref{la:A4}, there exists $\ga\in\P_n$ such that $\ga\al\in D_1$ and $\ga\be\in I_1\sm H_{\ga\al}$, with $\Float(\ga,\al)=\Float(\ga,\be)=0$.  It follows that $((i,\ga\al),(i,\ga\be))\in\si$.  From $(\al,\be)\in\H\sub\R$, we have $(\ga\al,\ga\be)\in\R\sub\D$, so in fact $(\ga\al,\ga\be)\in\si_{1i}$.  Since $(\ga\al,\ga\be)\not\in\H$ but $(\ga\al,\ga\be)\in\R$, we deduce $(\ga\al,\ga\be)\not\in\L$, and so $\si_{1i}\not\sub\mudown$.  Thus, Lemma \ref{la:D2} gives $\muup\sub\si_{1i}$.  The dual of the above argument gives $\mudown\sub\si_{1i}$, so in fact $\mu=\mudown\vee\muup\sub\si_{1i}$.  This rules out the possibilities $\Delta,\muup,\mudown$ for $M_{q-1,i}$ (cf.~Lemma \ref{la:D4}).

\ref{E4.2}  Consulting Table \ref{tab:M}, $M_{q,i+1}$ is certainly $R$ or some $N'\normal\S_q$.  In the latter case, Lemma \ref{la:Ap3} gives $\nu_N=\si_{qi}\sub\si_{q,i+1}=\nu_{N'}$, and so $N\leq N'$.
\end{proof}

\begin{lemma}\label{la:De}
If $M_{qi}=\De$, then $M_{q,i-1}=\De$ if $i\geq1$, and $M_{q+1,i}=\De$ if $q<n$.
\end{lemma}

\pf
For the first statement, suppose $i\geq1$.  From $M_{qi}=\De$ we have $\si_{qi}=\De_{D_q}$, so Lemma~\ref{la:Ap3} gives $\si_{q,i-1}=\De_{D_q}$; this completes the proof for $q\geq1$.  For $q=0$, we could only have $M_{0,i-1}=\De$ or $\mu$; but in the latter case we would have $\si\cap(D_{0,i-1}\times D_{1j})\neq \emptyset$ for some $j\in\N$, and Lemma~\ref{la:Ap1} would then give $\si\cap(D_{0i}\times D_{1,j+1})\neq \emptyset$, whence $M_{0i}=\mu$, a contradiction.

The second statement follows immediately from Lemma \ref{la:D5} (with $j=i$) for $q=0$, or from Lemmas \ref{la:E3} and \ref{la:E4} for $0<q<n$.
\epf

Now we move on to the entries in rows $0$ and $1$ and their interdependencies:

\begin{lemma}
\label{la:E7}
If $\si\cap (D_{0i}\times D_{1j})\neq\emptyset$ then $M_{0i}=M_{1j}\in\{\mu,\lambda,\rho, R\}$,
and either 
\[
\text{$[i< \min\th_0$ and $j<\min\th_1]$ \qquad or \qquad $[i\geq \min\th_0$ and $j\geq\min\th_1]$.}
\]
\end{lemma}

\begin{proof}
Consulting Table \ref{tab:M}, we have $M_{0i}\in\{\mu,\lambda,\rho, R\}$.  It then follows quickly from Lemma~\ref{la:D1} that $M_{1j}=M_{0i}$.

Suppose now that $j\geq\min\th_1$. Let $d:=\per\th_1$, and let $\alpha\in D_1$ be arbitrary.
Using Lemmas~\ref{la:D1} and~\ref{la:Ap1}, and the definition of $\th_1$, we have
$(i,\widehat{\alpha})\rsi(j,\alpha)\rsi(j+d,\alpha)\rsi(i+d,\widehat{\alpha})$,
and so $i\geq \min\th_0$.
An entirely analogous argument shows that if $i\geq\min\th_0$ then $j\geq\min\th_1$, and completes the proof.
\end{proof}

\begin{lemma}
\label{la:E10}
If $M_{1i}\neq\Delta$ then $M_{0i}=M_{1,i+1}\in\{ \mu,\rho,\lambda,R\}$ and $\per\th_0=\per\th_1$.
\end{lemma}

\begin{proof}
The first assertion follows from Lemmas \ref{la:D5} and \ref{la:E7}.  For the second, $\th_1\sub\th_0$ gives $\per\th_0\leq\per\th_1$.  It remains to show that $\per\th_1\leq\per\th_0$.  This being clear if $\per\th_0=\infty$, suppose instead that $\th_0=(m,m+d)^\sharp$.  Fix some $\al\in D_1$, and put $j:=\max(i,m)$.  Since $j\geq i$, Lemma \ref{la:D5} gives ${((j,\wh\al),(j+1,\al))\in\si}$.  Since $j\geq m$, we have $(j,j+d)\in\th_0$.  Combining the above with Lemma \ref{la:Ap1}, it follows that ${(j+1,\al) \rsi (j,\wh\al) \rsi (j+d,\wh\al) \rsi (j+1+d,\al)}$.  But then $(j+1,j+1+d)\in\th_1$, so that ${\per\th_1\leq d=\per\th_0}$, as required.
\end{proof}

\begin{lemma}
\label{la:E11}
If $M_{qi}\in\{\rho,\lambda,R\}$ then $i\geq\min\th_q$ and $\per\th_q=1$.
\end{lemma}

\begin{proof}
The $q=0$ case follows from Lemma \ref{la:C4a'}\ref{C4a'2}, and the $M_{qi}=R$ case from Lemma \ref{la:E3}. So consider $q=1$.
By Lemma \ref{la:E10} and $\si_{1i}\sub\si_{1,i+1}$, we have $M_{0i}=M_{1,i+1}\in\{\rho,\lambda, R\}$.  The $q=0$ case then gives $i\geq\min\th_0$ and $\per\th_0=1$.
Since $\si_{1i}$ contains $\lam_1\restr_{D_1}$ or $\rho_1\restr_{D_1}$, Lemma \ref{la:A3} implies there exists $((i,\al),(i,\be))\in\si\cap (D_{1i}\times D_{0i})$.  Then using Lemma \ref{la:Ap1} and $(i,i+1)\in\th_0$, we have $(i,\al) \rsi (i,\be) \rsi (i+1,\be) \rsi (i+1,\al)$, so that $(i,i+1)\in\th_1$.  The result follows.
\end{proof}

\begin{lemma}\label{la:mumin}
If $M_{0i}=\mu$ for some $i<\min\th_0$, then there exists a unique $j\in\N$ such that $\si\cap(D_{0i}\times D_{1j})\neq \emptyset$.  Furthermore, we have $M_{1j}=\mu$, and also
\[
i<j<\min\th_1 \AND i+\min\th_1=j+\min\th_0.
\]
\end{lemma}

\pf
Beginning with the first assertion, $M_{0i}=\mu$ implies the existence of at least one such~$j$ (see Table \ref{tab:M}); $M_{1j}=\mu$ follows from Lemma \ref{la:E7}.  To prove uniqueness of $j$, and aiming for a contradiction, suppose ${\si\cap(D_{0i}\times D_{1j})\neq \emptyset}$ and $\si\cap(D_{0i}\times D_{1k})\neq \emptyset$ with $j<k$, and write $d:=k-j>0$.  Fix some $\al\in D_1$.  Lemma~\ref{la:D1} gives $(j,\al) \rsi (i,\wh\al) \rsi (k,\al)=(j+d,\al)$.  Combining this with Lemma \ref{la:Ap1}, we have ${(i,\wh\al) \rsi (j+d,\al) \rsi (i+d,\wh\al)}$, which gives $(i,i+d)\in\th_0$, contradicting $i<\min\th_0$.  Thus,~$j$ is indeed unique.

Since $i<\min\th_0$, Lemma \ref{la:E7} gives $j<\min\th_1$.  If $i\geq j$, then since $M_{1j}=\mu$, Lemma \ref{la:D5} gives $\si\cap(D_{0i}\times D_{1,i+1})\neq \emptyset$ with $i+1>j$, contradicting the uniqueness of $j$.  So $i<j$.

For the final assertion, write $m_0:=\min\th_0$ and $m_1:=\min\th_1$.  Since $M_{1j}=\mu$,  Lemma~\ref{la:E10} gives $\per\th_1=\per\th_0$.  It follows that either $m_0=m_1=\infty$ or else $m_0,m_1<\infty$, and of course $i+m_1=j+m_0$ only needs proof in the second case.
Since $i<m_0$, and since ${\si\cap(D_{0i}\times D_{1j})\neq \emptyset}$, Lemma \ref{la:Ap1} (with $k=m_0-i-1$ and $k=m_0-i$) tells us that $\si\cap(D_{0,m_0-1}\times D_{1,j+m_0-i-1})$ and $\si\cap(D_{0,m_0}\times D_{1,j+m_0-i})$ are both non-empty.  It then follows from Lemma \ref{la:E7} that ${j+m_0-i-1 < m_1}$ and $j+m_0-i \geq m_1$.  
Solving these leads to $i+m_1=j+m_0$.
\epf

We are now ready to prove the main result of this subsection.

\begin{lemma}
\label{la:rowq}
Given a congruence $\si$ on $\Ptw{n}$, the pair $\Pair=(\Th,M)$ given in Definitions \ref{defn:Th} and~\ref{defn:M} is a C-pair.
\end{lemma}

\begin{proof}
We have already seen that $\Th$ is a C-chain in Lemma \ref{la:siqcong}, so we now turn to the matrix~$M$.
By Lemmas~\ref{la:E1}--\ref{la:E4}, each row $q\geq 2$ is of type \ref{RT8}--\ref{RT10}, and the verticality conditions \ref{V1} and \ref{V2} hold.
It remains to be proved that rows $0$ and $1$ are of one of the types \ref{RT1}--\ref{RT7}.
We split our considerations into cases, depending on whether $\th_0$ and/or $\th_1$ is $\Delta_\N$.
Throughout the proof we make extensive use of Table \ref{tab:M} without explicit reference, and also of the fact that any entry above or to the left of a $\De$ is also $\De$ (Lemma \ref{la:De}).  We also keep the meaning of symbols such as $i$, $\xi$ and $\ze$ from the row type specifications in Subsection \ref{subsec:Cpairs}.

\setcounter{caseco}{0}

\case $\th_0=\th_1=\Delta_\N$.
By Lemma  \ref{la:E11}, the only symbols that can appear in row $0$ are $\Delta$ and~$\mu$, and in row $1$ the only possibilities are $\Delta$, $\muup$, $\mudown$ and $\mu$.
If row $0$ consists entirely of $\De$s, then so too does row $1$ and we have type \ref{RT1}.  Otherwise, row $0$ has the form $\De\dots\De\mu\mu\mu\dots$, with the first $\mu$ in position $i$, say.  The entries above the $\De$s are also $\De$s.  For any $j\geq i$, it follows from $M_{0j}=\mu$ and Lemma \ref{la:(i+1,j)} that $\si\cap(D_{0j}\times D_{1,j+1})\not=\emptyset$; Lemma \ref{la:mumin} then gives $M_{1,j+1}=\mu$.  Thus, $M_{1k}=\mu$ for all $k\geq i+1$.  Since $M_{1i}\in\{\De,\muup,\mudown,\mu\}$, we have type \ref{RT2}.

\case
 $\th_0\neq\Delta_\N$ and $\th_1=\Delta_\N$.
Lemma \ref{la:E10} implies that row $1$ consists entirely of $\De$s.  It then follows from Lemmas \ref{la:mumin} and \ref{la:D1} that row $0$ may not contain any $\mu$.  If row $0$ consists entirely of $\De$s, then we have \ref{RT1}.  Otherwise, by Lemmas \ref{la:E1} and \ref{la:E11}, row $0$ has the form $\Delta\dots\Delta\xi\xi\dots$, with $\xi\in \{\lambda,\rho,R\}$ and $\per\th_0=1$, and hence we have type \ref{RT3}.

\case $\th_0,\th_1\neq\Delta_\N$.
If all the entries in row $1$ are $\Delta$, then as in the previous case we have type~\ref{RT1} or \ref{RT3}.  So for the remainder of the proof we will assume that some entries of row $1$ are distinct from $\Delta$.  By Lemma \ref{la:E10} and $\th_0\supseteq\th_1$, we must have
\[
\th_0=(m,m+d)^\sharp \AND \th_1=(l,l+d)^\sharp\qquad \text{for some $0\leq m\leq l$ and $d\geq 1$.}
\]
From Lemmas \ref{la:E1} and \ref{la:E10}, there exists $\xi\in\{ \mu,\lambda,\rho , R\}$ such that $M_{0j}=M_{1k}=\xi$ for all $j\geq m$ and all $k\geq l$.  
Furthermore, we note that if $\xi\in\{\lambda,\rho, R\}$ then $d=1$ by Lemma \ref{la:E11}.
We now split into subcases, depending on the relationship between $m$ and $l$.

\setcounter{subcaseco}{0}

\subcase $m=l$.
We claim that any entries on both rows to the left of $m$ equal $\Delta$, and we note then that we will have type \ref{RT4}.  To prove the claim, it is sufficient to show that $M_{0,m-1}=\Delta$ if $m\geq1$.
But if $M_{0,m-1}\neq \De$, then Lemma \ref{la:E11} gives $M_{0,m-1}=\mu$, and Lemma~\ref{la:mumin} then implies the existence of an integer $j$ satisfying $m-1<j<m$, a contradiction.

\subcase $l=m+1$.
If $m=0$ or if $M_{0,m-1}=\Delta$, then $M_{0j}=M_{1j}=\Delta$ for all $j<m$; the entry $M_{1m}$ can only be one of $\Delta$, $\muup$, $\mudown$ or $\mu$ by Lemma \ref{la:E11}, and we have type~\ref{RT6}.  
The only remaining option (again see Lemma \ref{la:E11}) is that ${M_{0i}=\dots=M_{0,m-1}=\mu}$ for some $i\leq m-1$, and we assume that $i$ is minimal with this property.  Again, we must have $M_{0j}=M_{1j}=\Delta$ for all $j<i$.  Applying Lemma \ref{la:mumin}, and keeping ${\min\th_1=\min\th_0+1}$ in mind, it follows that $M_{1,i+1}=\dots=M_{1m}=\mu$.  Finally, the entry $M_{1i}$ can again only be one of $\Delta$, $\muup$, $\mudown$ or $\mu$, and we have type~\ref{RT5}.

\subcase $l>m+1$.
As usual, the entry $M_{1,l-1}$ must be one of $\Delta$, $\muup$, $\mudown$ or $\mu$.  If $M_{1,l-2}\neq \De$, then Lemma~\ref{la:D5} would give $\si\cap(D_{0,l-2}\times D_{1,l-1})\neq \emptyset$, and this contradicts Lemma \ref{la:E7} since $l-2\geq\min\th_0$ (as $l>m+1$) and $l-1<\min\th_1$.  It follows that $M_{1,l-2}=\De$, and hence $M_{1j}=\De$ for all $j\leq l-2$. 
If $m=0$ or $M_{0,m-1}=\Delta$, then $M_{0j}=\Delta$ for all~$j<m$, and we have \ref{RT6}.

So now suppose $m\geq1$ and $M_{0,m-1}\neq \De$, which means $M_{0,m-1}=\mu$.  Applying Lemma \ref{la:mumin} with $i=m-1$, the $j$ from the conclusion has to be $j=l-1$; in particular, we have $M_{1,l-1}=\mu$.  
Since $\si\cap(D_{0,m-1}\times D_{1,l-1})\neq \emptyset$, Lemma \ref{la:(i+1,j)} gives $(m,l-1)\in\th_0$, i.e.~${l-1\equiv m\pmod d}$.
If $m\geq2$ and $M_{0,m-2}\neq \De$, then $M_{0,m-2}=\mu$, and as above Lemma \ref{la:mumin} (with $i=m-2$) leads to $M_{1,l-2}=\mu$, contradicting $M_{1,l-2}=\De$.  Thus, we have either $m=1$ or $M_{0,m-2}=\De$, so that $M_{0j}=\De$ for all $j\leq m-2$.  Therefore, this time we have type \ref{RT7}.
\end{proof}

\subsection{Restrictions to pairs of \boldmath{$\D$}-classes}
\label{subsec:further}

Now that we have associated the C-pair $\Pair=(\Th,M)$ to the congruence $\si$ on $\Ptw n$ (Definitions~\ref{defn:Th} and~\ref{defn:M}), we wish to show that $\si$ is one of the congruences associated to the pair (Definitions~\ref{de:cg} and~\ref{def:exc}).  We do this in Subsection~\ref{subsec:CongCpair}, but first we require some further technical lemmas describing the possible restrictions of $\si$ to pairs of $\D$-classes.

For any $q\in\bnz$, we clearly have $(i,j)\in\th_q\implies \si\cap(D_{qi}\times D_{qj})\neq \emptyset$.  By Lemmas \ref{la:C1} and~\ref{la:C1'}, the reverse implication holds as well for $q\leq1$.  This need not be the case for $q\geq2$, however, as shown by the exceptional congruences.  The next lemma shows how to deal with this possibility.

For the duration of Subsection \ref{subsec:further}, we will treat $\De$-entries in row $2$ as $N$-symbols, $\De\equiv\A_2$ (see~Remark \ref{rem:nuN}).

\begin{lemma}
\label{la:E5}
If $\si\cap (D_{qi}\times D_{qj})\neq\emptyset$, but $(i,j)\not\in\th_q$, then $\Pair$ is exceptional, $\hgt(M)=q$, and $(i,j)\in\thx_q\sm\th_q$.
\end{lemma}

\begin{proof}
Clearly $i\neq j$, say $i<j$.  By Lemmas \ref{la:C1} and \ref{la:C1'} we have $q\geq2$.  
Referring to Definition~\ref{def:exc}, we must show that all of the following items hold:
\bena
\item \label{E5a} $\th_q=(m,m+2d)^\sharp$ for some $m\geq0$ and $d\geq1$;
\item \label{E5b} $\thx_q:=(m,m+d)^\sharp\subseteq\th_{q-1}$ if $q=2$;
\item \label{E5c} $M_{qm}=\A_q$ (remembering $\A_2\equiv\De$ for $q=2$);
\item \label{E5d} $M_{1m}\in\{\mu,\rho,\lam,R\}$ if $q=2$;
\item \label{E5e} $(i,j)\in\thx_q\sm\th_q$.
\eena
Let us begin with an arbitrary $((i,\alpha),(j,\beta))\in\si\cap (D_{qi}\times D_{qj})$.

We first claim that $M_{qi}\not\in\{\S_q,R\}$.  Indeed, if $M_{qi}=R$, then we also have $((i,\alpha),(i,\beta))\in\si$, and hence $((i,\beta),(j,\beta))\in\si$ by transitivity, so that $(i,j)\in\th_q$, a contradiction.
If $M_{qi}=\S_q$, note that Lemma \ref{la:B5} gives $(\al,\be)\in\H\restr_{D_q}=\nu_{\S_q}=\si_{qi}$, so that $((i,\al),(i,\be))\in\si$ again, leading to the same contradiction.  This completes the proof of the claim.

It follows from Lemma \ref{la:B8} that $\si_{qi}=\nu_N$ for some $N\vartriangleleft \S_q$.
An analogous argument shows that $\si_{qj}=\nu_{N'}$ for some $N'\vartriangleleft \S_q$, and so $N\leq N'\leq \A_q$ by Lemma \ref{la:Ap3}.
By Lemma \ref{la:B5}, for any $((i,\ga),(j,\de))\in\si\cap (D_{qi}\times D_{qj})$ we must have $\ga\rH\de$.

Using Lemmas \ref{la:0float} and \ref{la:A4} we see that every element $(i,\ga)$ in the $\R$-class of $(i,\al)$ is $\si$-related to some element of $D_{qj}$; similarly, every element in the $\L$-class of any such $(i,\ga)$ is $\si$-related to some element of $D_{qj}$.  Since $\D=\R\circ\L$, and by the previous paragraph, it follows that:
\begin{equation}\label{eq:DqiDqj}
\forall\gamma\in D_q:\: \exists \delta\in H_\gamma:\: (i,\gamma)\rsi(j,\delta) ,
\end{equation}
where $H_\ga$ is the $\H$-class of $\ga$ in $\P_n$.

Let us now focus on a particular $\H$-class of $D_q$, the one containing the elements $\alpha$ such that
$\dom\alpha=\codom\alpha=\bn$ and $\bn/\ker\alpha=\bn/\ker\beta=\{\{1\},\dots,\{q-1\},\{q,\dots,n\}\}$.
This is a group $\H$-class isomorphic to $\S_q$, and we denote the natural isomorphism by
\[
\pi\mt\pi^\natural = \begin{partn}{3}1^\natural&\cdots&q^\natural\\(1\pi)^\natural&\cdots&(q\pi)^\natural\end{partn} \qquad\text{for $\pi\in\S_q$\COMMA where } j^\natural = \begin{cases}
\{j\} &\text{if $j<q$}\\
\{q,\dots,n\} &\text{if $j=q$.}
\end{cases}
\]
Observe that $\Float(\pi^\natural,\eta)=\Float(\eta,\pi^\natural)=0$ for all $\pi\in \S_q$ and $\eta\in \P_n$.

By \eqref{eq:DqiDqj} we have $((i,\id_q^\natural),(j,\pi^\natural))\in\si$ for some $\pi\in \S_q$.
Note that $\pi\not\in N$, for otherwise $(i,\pi^\natural)\rsi (i,\id_q^\natural)\rsi (j,\pi^\natural)$, contradicting $(i,j)\not\in\th_q$.
Writing $e:=j-i>0$, we have
\[
((i+e,\pi^\natural),(i+2e,(\pi^2)^\natural)) = ((i,\id_q^\natural),(j,\pi^\natural)) \cdot (e,\pi^\natural) \in \si.
\]
Continuing and using transitivity we conclude $((i,\id_q^\sharp),(i+le,\id_q^\sharp))\in\si$, where $l\geq1$ is the order of~$\pi$ in $\S_q$.
Therefore $(i,i+le)\in\th_q$.  In particular, $\th_q\neq \De_\N$, say $\th_q=(m,m+f)^\sharp$.  Since $i\neq i+le$, we also have $i\geq m$, so also $j\geq i\geq m$, and Lemma \ref{la:E1} then gives $N=N'$.

Now, supposing $N\neq \A_q$, the quotient $\S_q/N$ has a trivial center, and, recalling $\pi\not\in N$, there exists $\eta\in \S_q$ such that
$[\pi,\eta]=\pi\eta\pi^{-1}\eta^{-1}\not\in N$.
Consider again the pair $((i,\id_q^\natural),(j,\pi^\natural))\in\si$, and multiply it by $(0,\eta^\natural)$ on the left, and on the right, to obtain
$(j,(\eta\pi)^\natural)\rsi (i,\eta^\natural)\rsi(j,(\pi\eta)^\natural)$,
from which it follows that $[\pi,\eta]=\pd((\eta\pi)^\natural,(\pi\eta)^\natural)\in N$, a contradiction.  (The $\pd$ operator was defined just before Theorem \ref{thm:CongPn}.)
Therefore, we must have $N=N'=\A_q$.  Since $i\geq m$ and $M_{qi}=\A_q$, it follows from Remark \ref{re:se}\ref{se1} that $M_{qm}=\A_q$,
i.e.~\ref{E5c} holds.

It now follows that $\pi$ is an odd permutation, and $\pi^2$ even.  Since $M_{q,i+2e}=\A_q$ by Lemma~\ref{la:E1}, we have
\[
(i,\id_q^\natural)\rsi(i+e,\pi^\natural)\rsi(i+2e,(\pi^2)^\natural)\rsi (i+2e,\id_q^\natural),
\]
so that $(i,i+2e)\in\th_q$.  Hence $f=\per\th_q\mid 2e$, and, combining with $f\nmid e$ (as $(i,i+e)=(i,j)\not\in\th_q$) it follows that $f$ is even, say $f=2d$, and that $j-i=e\equiv d\pmod{2d}$.  Since $i,j\geq m$, it follows that $(i,j)\in(m,m+d)^\sharp\sm\th_q$.
This all shows that \ref{E5a} and \ref{E5e} both hold.

We are left to deal with \ref{E5b} and \ref{E5d}, so we assume that $q=2$ for the rest of the proof.  Note that the permutation $\pi\in\S_2\sm\A_2$ must in fact be the transposition $(1,2)$.  To simplify notation in what follows, we will write
\[
\ga := \id_2^\natural = \begin{partn}{2}1^\natural&2^\natural\\1^\natural&2^\natural\end{partn}
\AND
\de := \pi^\natural = \begin{partn}{2}1^\natural&2^\natural\\2^\natural&1^\natural\end{partn}.
\]
As above, we have $((i,\ga),(j,\de))\in\si$.  Next we claim that
\begin{equation}\label{eq:ideta}
((m,\ga),(m+d,\de))\in\si .
\end{equation}
To prove this, let $t\in\N$ be such that $m+2td\geq i$, say $m+2td = i+u$.  
Since $j-i\equiv d\pmod {2d}$ and $i,j\geq m$, we have $(j,i+d)\in\th_2$, and so $(i,\ga) \rsi (j,\de) \rsi (i+d,\de)$.  Combining the above with Lemma \ref{la:Ap1}, and keeping in mind that $\per\th_2=2d$, it follows that indeed
\[
(m,\ga) \rsi (m+2td,\ga) = (i+u,\ga) \rsi (i+d+u,\de) = (m+d+2td,\de) \rsi (m+d,\de).
\]

Using \eqref{eq:ideta}, and again writing $\omega:=\begin{partn}{1} {\ \! \bn\ \!}\\ {\ \! \bn\ \!} \end{partn}$, we have
\[
((m,\om),(m+d,\om)) = (0,\om) \cdot ((m,\ga),(m+d,\de)) \in \si.
\]
Since $\om\in D_1$, this shows that $(m,m+d)\in\th_1$, from which \ref{E5b} follows.

Finally, let $\eta_1=\begin{partn}{2}1^\natural&2^\natural\\\hhline{~|-}2^\natural&1^\natural\end{partn}$ and $\eta_2=\begin{partn}{2}2^\natural&1^\natural\\\hhline{~|-}1^\natural&2^\natural\end{partn}$.  Then again using \eqref{eq:ideta} we have
\[
((m,\eta_1),(m+d,\eta_1\de)) = (0,\eta_1)\cdot((m,\ga),(m+d,\de))\in\si,
\]
and similarly $((m,\eta_2),(m+d,\de\eta_2))$.  But $\eta_1\de=\de\eta_2=\begin{partn}{2}1^\natural&2^\natural\\\hhline{~|-}1^\natural&2^\natural\end{partn}$, so it follows that
\[
(m,\eta_1) \rsi (m+d,\eta_1\de) = (m+d,\de\eta_2) \rsi (m,\eta_2),
\]
and so $(\eta_1,\eta_2)\in\si_{1m}$.  Since $\eta_1$ and $\eta_2$ are neither $\muup$- nor $\mudown$-related (as they are neither $\R$- nor $\L$-related), it follows that $\si_{1m}$ cannot be one of $\De_{D_1}$, $\muup$ or $\mudown$.  Examining Table \ref{tab:M}, we see then that $M_{1m}\in\{\mu,\rho,\lam,R\}$.  This completes the proof of \ref{E5d}, and indeed of the lemma.
\end{proof}

The next lemma describes the conditions under which $\si$-relationships can exist between distinct $\D$-classes, and then the two subsequent ones characterise all such relationships.

\begin{lemma}
\label{la:DD}
Suppose $\si\cap (D_{qi}\times D_{rj})\neq\emptyset$, where $q\leq r$ and $(q,i)\neq (r,j)$.  Then at least one of the following holds:
\begin{thmenumerate}
\item
\label{DD1}
$M_{qi}=M_{rj}=R$; 
\item
\label{DD2}
$q=r$, $M_{qi}=M_{rj}$ and $(i,j)\in\th_q$;
\item
\label{DD3}
$q=0$, $r=1$, $M_{0i}=M_{1j}\neq \De$, $i\geq \min\th_0$, $j\geq \min\th_1$ and $(i+1,j)\in \th_0$;
\item
\label{DD4}
$q=0$, $r=1$, $M_{0i}=M_{1j}=\mu$, $i< \min\th_0$, $j< \min\th_1$ and $(i+1,j)\in \th_0$;
\item
\label{DD5}
$q=r\geq2$, $\Pair$ is exceptional, $\hgt(M)=q$, and $(i,j)\in\thx_q\sm\th_q$.
\end{thmenumerate}
\end{lemma}

\begin{proof}
We split our considerations into cases, depending on whether $q=r$ and whether $r>1$.

\setcounter{caseco}{0}

\case $q\neq r$ and $r>1$.
By Lemma \ref{la:B5} we have $M_{qi}=M_{rj}=R$, and so \ref{DD1} holds.

\case $q=r>1$.
If $(i,j)\in\th_q$ then Lemma \ref{la:E1} gives \ref{DD2}.
Otherwise Lemma \ref{la:E5} gives~\ref{DD5}.

\case $q=r\leq 1$.
Here Lemmas \ref{la:C1}, \ref{la:C1'} and \ref{la:E1} imply \ref{DD2}.

\case $q=0$ and $r=1$.
Lemmas \ref{la:(i+1,j)} and \ref{la:E7} give $(i+1,j)\in\th_0$ and $M_{0i}=M_{1j}\in\{\mu,\lam,\rho,R\}$.  Lemma \ref{la:E7} also tells us that either $i\geq \min\th_0$ and $j\geq \min\th_1$, or else $i<\min\th_0$ and $j<\min\th_1$.  In the former case, \ref{DD3} holds.  In the latter case, Lemma \ref{la:E11} gives $M_{0i}=\mu$, so \ref{DD4} holds.
\end{proof}

It turns out that the converse of Lemma \ref{la:DD} is \emph{almost} true.  The only exception is in item~\ref{DD5}, which concerns exceptional congruences.  Accordingly, the next lemma treats cases \ref{DD1}--\ref{DD4}, and the following one deals with \ref{DD5}.

\begin{lemma}\label{la:DD'}
Suppose $q,r\in\bnz$ and $i,j\in\N$ are such that $q\leq r$ and $(q,i)\neq (r,j)$.  If any of conditions \ref{DD1}--\ref{DD4} from Lemma \ref{la:DD} hold, then $\si\cap (D_{qi}\times D_{rj})\neq\emptyset$.  Moreover, in these respective cases, the following hold for all $\al\in D_q$ and $\be\in D_r$:
\ben
\item \label{DD'1} $((i,\alpha),(j,\beta))\in\si$;
\item \label{DD'2} $((i,\alpha),(j,\beta))\in\si\iff (\alpha,\beta)\in{\si}_{qi}$;
\item \label{DD'3} $((i,\alpha),(j,\beta))\in\si\iff (\alpha,\widehat{\beta})\in{\si}_{0i}$;
\item \label{DD'4} $((i,\alpha),(j,\beta))\in\si\iff \alpha=\widehat{\beta}$.
\een
\end{lemma}

\pf
\ref{DD1}  By definition of $M$, when $M_{qi}=M_{rj}=R$ we have $D_{qi},D_{rj}\subseteq I(\si)$, which is a $\si$-class, giving the claims.

\ref{DD2}  Since $(i,j)\in\th_q$, certainly $\si\cap (D_{qi}\times D_{qj})\neq\emptyset$.  Moreover, for any $\al,\be\in D_q$ it follows from $((i,\beta),(j,\beta))\in\si$ that $((i,\al),(j,\be)) \in \si \iff ((i,\al),(i,\be)) \in \si \iff (\al,\be) \in \si_{qi}$.

\ref{DD3}  Write $d:=\per\th_0=\per\th_1$ (see Lemma \ref{la:E10}), and let $\al\in D_0$ and $\be\in D_1$.  From $M_{1j}\neq \De$, Lemma \ref{la:D5} gives $((j,\wh\be),(j+1,\be))\in\si$.  It also follows from $(i+1,j)\in\th_0$ that $((i+1,\wh\be),(j,\wh\be))\in\si$.  Since $i\geq \min\th_0$ and $j\geq \min\th_1$, we have $(i,i+d)\in\th_0$ and $(j,j+d)\in\th_1$, with $1\leq d<\infty$.  Combining all of the above with Lemma \ref{la:Ap1}, we have
\[
(i,\wh\be) \rsi (i+d,\wh\be) = ((i+1)+(d-1),\wh\be) \rsi (j+(d-1),\wh\be) \rsi ((j+1)+(d-1),\be) = (j+d,\be) \rsi (j,\be).
\]
Consequently, $\si\cap (D_{0i}\times D_{1j})\neq\emptyset$.  From $((i,\wh\be),(j,\be))\in\si$, it also follows that
\[
((i,\al),(j,\be))\in\si \iff ((i,\al),(i,\wh\be))\in\si \iff (\al,\wh\be)\in\si_{0i}.
\]

\ref{DD4}  As in the previous part, it is enough to show that $((i,\wh\be),(j,\be))\in\si$ for all $\be\in D_1$, keeping in mind that $\si_{0i}=\De_{D_0}$, as $M_{0i}=\mu$.  By Lemma \ref{la:mumin}, we have $\si\cap(D_{0i}\times D_{1k})\neq \emptyset$ for a unique $k\in\N$, and we also have $M_{1k}=\mu$, $i<k<m_1$, and $i+m_1=k+m_0$.  By Lemmas~\ref{la:D1} and \ref{la:(i+1,j)} we have $((i,\wh\be),(k,\be))\in\si$ and $(i+1,k)\in\th_0$.  Since also $(i+1,j)\in\th_0$ by assumption, we have $(j,k)\in\th_0$.  

To complete the proof, it remains to show that $j=k$, as we have already shown that ${((i,\wh\be),(k,\be))\in\si}$.  Aiming for a contradiction, suppose instead that $j\neq k$.  Since $(j,k)\in\th_0$, we then have $j,k\geq m_0$.  Since also $(i+1,j)\in\th_0$ it follows that $i+1\geq m_0$ as well.  Combined with $i<m_0$ (which is one of the underlying assumptions in this case), we deduce that in fact $m_0=i+1$.  From $i+m_1=k+m_0$, it follows that $m_1=k+1$.  
Since $j<m_1=k+1$ and $j\neq k$, we then deduce $j\leq k-1$.  Since $M_{1j}=\mu$, Lemma \ref{la:D5} then gives $((k-1,\wh\be),(k,\be))\in\si$.  Combined with $((i,\wh\be),(k,\be))\in\si$ and transitivity, it follows that $((i,\wh\be),(k-1,\wh\be))\in\si$ and so $(i,k-1)\in\th_0$.  Since $i<m_0$ we deduce that $i=k-1$.  But then $j\leq k-1=i$.  Adding $e:=i-j\geq0$ to $(j,i+1)\in\th_0$, we obtain $(i,i+1+e)\in\th_0$, with $i+1+e>i$, and this contradicts $i<m_0$.  This completes the proof.
\epf

\begin{lemma}\label{la:DD''}
Suppose $\Pair$ is exceptional, with $\hgt(M)=q$.  If $\si\cap (D_{qi}\times D_{qj})\neq\emptyset$ for some $(i,j)\in\thx_q\sm\th_q$, then for all $\al,\be\in D_q$:
\[
((i,\alpha),(j,\beta))\in\si \IFf \alpha\rH\beta \text{ \ and \ }\pd(\alpha,\beta)\in \S_q\setminus \A_q.
\]
\end{lemma}

\pf
Since $M_{qi}\neq R$, Lemma \ref{la:B5} tells us that any pair $((i,\al),(j,\be))\in\si\cap (D_{qi}\times D_{qj})$ satisfies $\al\rH\be$.  So we need to show that for $(\al,\be)\in\H\restr_{D_q}$ we have
\[
((i,\alpha),(j,\beta))\in\si \iff \pd(\alpha,\beta)\in \S_q\setminus \A_q, \qquad\text{i.e.}\qquad ((i,\alpha),(j,\beta))\in\si \iff (\al,\be)\not\in\nu_{\A_q}.
\]

($\Rightarrow$)  Aiming for a contradiction, suppose $((i,\alpha),(j,\beta))\in\si$ and $(\al,\be)\in\nu_{\A_q}$.  Then from $M_{qi}=\A_q$, we have $((i,\al),(i,\be))\in\si$, so that $((i,\be),(j,\be))\in\si$, contradicting $(i,j)\not\in\th_q$.  

($\Leftarrow$)  Suppose $(\al,\be)\not\in\nu_{\A_q}$.  As in the proof of Lemma \ref{la:E5}, it follows from $\si\cap (D_{qi}\times D_{qj})\neq\emptyset$ that $((i,\al),(j,\ga))\in\si$ for some $\ga\in D_q$.  As in the previous two paragraphs, we have $\al\rH\ga$ and $(\al,\ga)\not\in\nu_{\A_q}$.  Since the $\H$-class containing $\al,\be,\ga$ is split into two $\nu_{\A_q}$-classes, and since $(\al,\be),(\al,\ga)\not\in\nu_{\A_q}$, it follows that $(\be,\ga)\in\nu_{\A_q}$.  Since $\si_{qj}=\nu_{\A_q}$, this gives $((j,\be),(j,\ga))\in\si$.  It then follows by transitivity that $((i,\al),(j,\be))\in\si$.
\epf

\subsection{The completion of the proof}
\label{subsec:CongCpair}

We are now ready to complete the proof of our main result, Theorem \ref{thm:main}.

\begin{prop}
\label{pr:siissipi}
Let $\si$ be a congruence on $\Ptw{n}$, and let $\Pair:=(\Th,M)$ be the C-pair given in Definitions \ref{defn:Th} and~\ref{defn:M}.
Then $\si=\cg(\Pair)$, or else 
$\Pair$ is exceptional and $\si=\cgx(\Pair)$.
\end{prop}

\begin{proof}
Suppose first that $\Pair$ is not exceptional, and let $\tau:=\cg(\Pair)$.  We need to prove that
\begin{equation}
\label{eq:sisipi}
((i,\alpha),(j,\beta))\in \si \iff ((i,\alpha),(j,\beta))\in\tau\qquad\text{for all  $i,j\in\N$ and all $\al,\be\in\P_n$}.
\end{equation}
To do so, fix $i,j\in\N$ and $\al,\be\in\P_n$, and write $q:=\rank\al$ and $r:=\rank\be$.  

Since the matrix entries are defined with direct reference to the restrictions of $\si$ to the corresponding $\D$-classes, it immediately follows that \eqref{eq:sisipi} holds whenever $(q,i)=(r,j)$.
So let us assume that $(q,i)\neq (r,j)$, and without loss of generality that $q\leq r$.
By inspection of Lemma~\ref{la:DD} and \ref{C1}--\ref{C8} we see that if $M_{qi}\neq M_{rj}$ then
$\si\cap (D_{qi}\times D_{rj})=\emptyset= \tau\cap (D_{qi}\times D_{rj})$,
and so~\eqref{eq:sisipi} holds.
So for the rest of the proof, we assume that $M_{qi}=M_{rj}$. 
We now split into cases, depending on the actual value of $M_{qi}$.

\setcounter{caseco}{0}

\case 
\label{ca:De}
$M_{qi}=M_{rj}=\Delta$.  If we do not have both $q=r$ and $(i,j)\in\th_q$, then by Lemma~\ref{la:DD}\ref{DD2} and \ref{C1}, $\si\cap(D_{qi}\times D_{rj})=\emptyset=\tau\cap(D_{qi}\times D_{rj})$.  For this, note that when $q=r=2$, item~\ref{DD5} of Lemma \ref{la:DD} involves $M_{qi}=M_{rj}=\A_2\equiv\De$; however the remaining conditions of this item cannot hold, as $\Pair$ is not exceptional.  If $q=r$ and $(i,j)\in\th_q$ then Lemma \ref{la:DD'}\ref{DD'2} and~\ref{C1} give
 \[
((i,\alpha),(j,\beta))\in\si\iff
  (\alpha,\beta)\in{\si}_{qi}=\Delta_{D_{q}}
 \iff
 \al=\be \iff
 ((i,\alpha),(j,\beta))\in\tau.
\]

\case $M_{qi}=M_{rj}=R$.
This is an immediate consequence of Lemma \ref{la:DD'}\ref{DD'1} and
 \ref{C2}.

\case 
\label{ca:N}
$M_{qi}=M_{rj}=N\unlhd \S_q$. 
Here we must of course have $q=r\geq 2$, and hence $i\neq j$.
Again, note that item \ref{DD5} from Lemma \ref{la:DD} cannot hold, since $\Pair$ is not exceptional. 
Therefore, if $(i,j)\not\in\th_q$ then $\si\cap (D_{qi}\times D_{rj})=\emptyset= \tau\cap (D_{qi}\times D_{rj})$.
If $(i,j)\in\th_q$ then by Lemma \ref{la:DD'}\ref{DD'2} and \ref{C3},
\[
((i,\alpha),(j,\beta))\in\si \Iff
(\alpha,\beta)\in{\si}_{qi}=\nu_N
  \Iff
\alpha\rH\beta \text{ and } \pd(\alpha,\beta)\in  N
\Iff
((i,\alpha),(j,\beta))\in\tau.
\]

In all the remaining cases we have $q,r\in\{0,1\}$.

\case $M_{qi}=M_{rj}=\lambda$.
By Lemma \ref{la:E11} we must have $i\geq\min\th_q$, $j\geq\min\th_r$ and $\per\th_q=\per\th_r=1$. 
Now, if $q=r$, then $i,j\geq\min\th_q$ gives $(i,j)\in\th_q$, and we use Lemma \ref{la:DD'}\ref{DD'2} and~\ref{C4} to obtain
\[
((i,\alpha),(j,\beta))\in \si\iff (\alpha,\beta)\in{\si}_{qi} = \lam_q\restr_{D_q} \iff
\widehat{\alpha}\rL\widehat{\beta}\iff ((i,\alpha),(j,\beta))\in\tau.
\]
If $q\neq r$, i.e.~$q=0$ and $r=1$, we use Lemma \ref{la:DD'}\ref{DD'3}
 and \ref{C4}, noting that $\al=\wh\al$ and $(i+1,j)\in\th_0$ (as $i\geq\min\th_0$, $j\geq\min\th_1\geq\min\th_0$ and $\per\th_0=1$):
 \[
((i,\alpha),(j,\beta))\in \si\iff (\alpha,\widehat{\beta})\in{\si}_{0i}= \lam_0\restr_{D_0} \iff
\alpha\rL\widehat{\beta}\iff ((i,\alpha),(j,\beta))\in\tau.
\]

\case $M_{qi}=M_{rj}=\rho$. This is dual to the previous case.

\case $M_{qi}=M_{rj}\in\{\muup,\mudown\}$.  Since there is at most one such entry, this case does not arise.


\case $M_{qi}=M_{rj}=\mu$.  Suppose first that $q=r$. In the same way as in Case \ref{ca:N} we can deal with the case $(i,j)\not\in\th_q$. If $(i,j)\in\th_q$ then using Lemma \ref{la:DD'}\ref{DD'2} and \ref{C8} we have:
 \[
((i,\alpha),(j,\beta))\in \si\iff (\alpha,\beta)\in{\si}_{qi}= \mu_q\restr_{D_q} \iff
\widehat{\alpha}=\widehat{\beta}\iff ((i,\alpha),(j,\beta))\in\tau.
\]

Now suppose $q\neq r$, i.e.~$q=0$ and $r=1$.  In this case, by \ref{C8}, $\tau\cap(D_{qi}\times D_{rj})$ is non-empty precisely when one of conditions \ref{DD3} or \ref{DD4} of Lemma \ref{la:DD} holds.  By Lemmas \ref{la:DD} and \ref{la:DD'}, these are precisely the conditions for $\si\cap(D_{qi}\times D_{rj})$ to be non-empty.  By \ref{C8} and Lemma \ref{la:DD'}, when one of these conditions holds, we have
\[
((i,\al),(j,\be))\in\si \iff \al=\wh\be \iff ((i,\al),(j,\be))\in\tau,
\]
keeping in mind that $\al=\wh\al$ (as $\al\in D_0$), and that $\si_{0i}=\De_{D_0}$ (as $M_{0i}=\mu$).

This completes the proof in the non-exceptional case.

\bigskip

Suppose now that $\Pair$ is exceptional and that $\sigma\neq\cg(\Pair)$.
This time let $\tau:=\cgx(\Pair)$.   
Since $\tau$ differs from $\cg(\Pair)$ only by virtue of containing certain pairs from $\D$-classes whose corresponding entry is $\A_q\unlhd \S_q$ (including $\A_2\equiv\De$ for $q=2$), as per Definition \ref{def:exc},
it follows that the preceding argument remains valid, with the exception of Cases \ref{ca:De} and \ref{ca:N}, at the point where we ruled out the conditions from Lemma~\ref{la:DD}\ref{DD5}.
So this time we use Lemma \ref{la:DD''} and~\ref{C9} to obtain:
\[
((i,\alpha),(j,\beta))\in\si \IFf \alpha\rH\beta \text{ \ and \ } \pd(\alpha,\beta)\in \S_q\setminus \A_q \IFf ((i,\alpha),(j,\beta))\in\tau,
\]
and the proof is complete.
\end{proof}

\section{Description of the inclusion ordering in terms of C-Pairs}
\label{sec:inclusion}

Having shown how to encode congruences on $\Ptw{n}$ as C-pairs, we now want to express the inclusion ordering on congruences in terms of an appropriate ordering on C-pairs (Theorem \ref{thm:comparisons}).

To build towards this, let $\leqc$ be the ordering on C-chains defined by componentwise inclusion of congruences on $\N$.
Next, on the set
\[
\{ \Delta,\muup,\mudown,\mu,\lambda,\rho,R\}\cup\set{ N}{ \{\id_q\}\neq N\unlhd \S_q,\ 2\leq q\leq n}
\]
of all possible C-matrix entries, we define an ordering via Hasse diagram in Figure \ref{fig:M}.  With a slight abuse of notation we will denote this ordering by $\leqc$ as well.
Next we extend this ordering to an ordering $\leqc$ on the set of all C-matrices in a componentwise manner.
And, finally, we define~$\leqc$ on the set of all C-pairs, also componentwise.

\begin{figure}[ht]
\begin{center}

\begin{tikzpicture}[scale=1]
\node (Delta) at (2,-0.5) {$\Delta$};
\node (S2) at (-0.5,3) {$\S_2$};
\node (A3) at (1,1.5) {$\A_3$};
\node (S3) at (1,3) {$\S_3$};
\node (K4) at (2,0.5) {$\mathcal K_4$};
\node (A4) at (2,1.5) {$\A_4$};
\node (S4) at (2,3) {$\S_4$};
\node (A5) at (3,1.5) {$\A_5$};
\node (S5) at (3,3) {$\S_5$};
\node at (4,2.25) {$\dots$};
\node (An) at (5,1.5) {$\A_n$};
\node (Sn) at (5,3) {$\S_n$};
\node (R) at (2,4.5) {$R$};

\node (mudown) at (-1.5,1) {$\mudown$};
\node (muup) at (-3,1) {$\muup$};
\node (mu) at (-2.25,2) {$\mu$};
\node (lambda) at (-1.5,3) {$\rho$};
\node (rho) at (-3,3) {$\lam$};

\draw
(Delta)--(S2)
(Delta)--(A3)--(S3)
(Delta)--(K4)--(A4)--(S4)
(Delta)--(A5)--(S5)
(Delta)--(An)--(Sn)
(Delta)--(mudown)--(mu)--(lambda)
(Delta)--(muup)--(mu)--(rho)
(rho)--(R)
(lambda)--(R)
(S2)--(R)
(S3)--(R)
(S4)--(R)
(S5)--(R)
(Sn)--(R)
;
\end{tikzpicture}

\caption{The partial ordering $\leqc$ on  the C-matrix entries.}
\label{fig:M}

\end{center}
\end{figure}
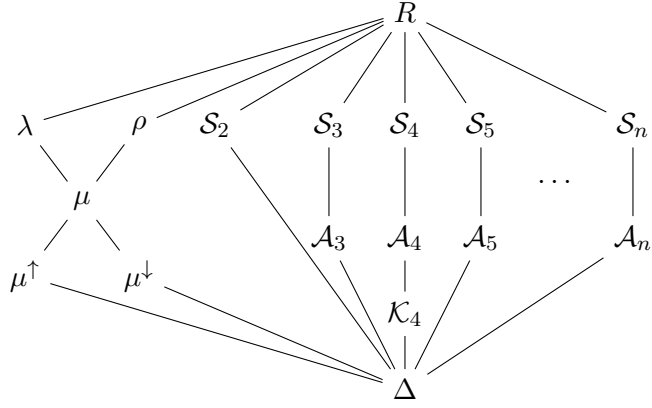

Ideally, one might hope that $\sigma^1\subseteq \sigma^2\iff\Pair^1\leqc\Pair^2$, where $\sigma^t$ are congruences and $\Pair^t$ are their corresponding C-pairs.
Unfortunately, this is not true, due to 
related pairs brought in by matching $\mu$s in rows $0$ and $1$, as well as those brought in by the exceptional congruences.
The most succinct statement we can make, which will be then used in the full description, as well as subsequent applications, is the following:

\begin{lemma}
\label{la:shortcomp}
Let $\si^1$ and $\si^2$ be two congruences on $\Ptw{n}$, with corresponding C-pairs ${\Pair^1=(\Th^1,M^1)}$ and ${\Pair^2=(\Th^2,M^2)}$.
\begin{thmenumerate}
\item
\label{it:shco1}
If $\si^1\subseteq\si^2$ then $\Pair^1\leqc\Pair^2$.
\item
\label{it:shco2}
If $\Pair^1\leqc\Pair^2$, $M^1$ is not of type~\ref{RT2},~\ref{RT5} or~\ref{RT7},
and each $\sigma^t=\cg(\Pair^t)$, then 
$\sigma^1\subseteq\sigma^2$.
\end{thmenumerate}
\end{lemma}

\begin{proof}
\ref{it:shco1}
Suppose $\si^1\sub\si^2$.
To show that $\Theta^1\leqc\Theta^2$, let $q\in\bnz$ and $\alpha\in D_q$,
and use \ref{C1'} to obtain
\[
(i,j)\in\th_q^1 \implies ((i,\al),(j,\al))\in\si^1\sub\si^2 \implies (i,j)\in\th_q^2.
\]

To show that $M^1\leqc M^2$, fix some $q\in\bnz$ and $i\in\N$; we need to show that $M_{qi}^1\leqc M_{qi}^2$.  From $\si^1\sub\si^2$, we immediately obtain $\si_{qi}^1\sub\si_{qi}^2$.  Comparing Table \ref{tab:M} with Figure \ref{fig:M}, we see that $\si_{qi}^1\sub\si_{qi}^2$ implies $M_{qi}^1\leqc M_{qi}^2$ in all but the following two cases:
\[
\text{[$q=n$, $M_{ni}^1=R$ and $M_{ni}^2=\S_n$] \qquad or \qquad [$q=0$, $M_{0i}^1=\mu$ and $M_{0i}^2=\De$].}
\]
So it remains to show that these cases do not arise.  Now, if $M_{ni}^1=R$, then $D_{ni}\sub I(\si^1)\sub I(\si^2)$, which means that $M_{ni}^2=R$ as well.  If $M_{0i}^1=\mu$, then we have $\si^1\cap(D_{0i}\times D_{1j})\neq \emptyset$ for some $j\in\N$, and this implies $\si^2\cap(D_{0i}\times D_{1j})\neq \emptyset$, so that $M_{0i}^2=\mu$ as well.

\ref{it:shco2}
Suppose the stated assumptions hold.  Fix some $(\ba,\bb)\in\si^1$, and write ${\ba=(i,\al)\in D_{qi}}$ and $\bb=(j,\be)\in D_{rj}$.  We must show that $(\ba,\bb)\in\si^2$.
Comparing Table~\ref{tab:M} with Figure \ref{fig:M}, we see that $M_{qi}^1\leqc M_{qi}^2$ implies $\si_{qi}^1\sub\si_{qi}^2$ in all cases, so we certainly have $(\ba,\bb)\in\si^2$ when $(q,i)=(r,j)$.  We now assume $(q,i)\neq (r,j)$, and we split our considerations into cases, depending on which of \ref{C1}--\ref{C8} is responsible for $(\ba,\bb)\in\si^1=\cg(\Pair^1)$.
\smallskip

\ref{C1}
Here $\alpha=\beta$ and $(i,j)\in\th_q^1\subseteq \th_q^2$, so that $(\ba,\bb)\in\si^2$ by \ref{C1'}.
\smallskip

\ref{C2}
From $R=M_{qi}^1\leqc M_{qi}^2$ and Figure \ref{fig:M} we conclude $M_{qi}^2=R$, and analogously $M_{rj}^2=R$.  Thus, $(\ba,\bb)\in\si^2$ by \ref{C2}.
\smallskip

\ref{C3}
From $(q,i)\neq (r,j)$ and $q=r$ we have $i\neq j$, and hence $i,j\geq \min\th_q^1\geq\min\th_q^2$, which implies $M_{qi}^2=M_{rj}^2$. 
If $M_{qi}^2=R$ then $(\ba,\bb)\in\si^2$ by \ref{C2}; otherwise $M_{qi}^2=N'\geq N$, so that $(\ba,\bb)\in\si^2$ by \ref{C3}, keeping in mind that $(i,j)\in\th_q^1\subseteq\th_q^2$.
\smallskip

\ref{C4}
Because of the $\lam$ entries, we have $i\geq\min\th_q^1\geq\min\th_q^2$ and $j\geq\min\th_r^1\geq\min\th_r^2$, from which we deduce that $M_{qi}^2=M_{rj}^2$.
The possible values for these entries are $\lambda$ and $R$, and so
$(\ba,\bb)\in\si^2$ by \ref{C4} or \ref{C2}.
\smallskip

\ref{C5}
This is dual to \ref{C4}.
\smallskip

\ref{C6} and \ref{C7}
do not arise, due to $(q,i)\neq (r,j)$.
\smallskip

\ref{C8}
If $q=r$, then since $(i,j)\in\th_q^1\sub\th_q^2$ we have $M_{qi}^2=M_{rj}^2\in\{\mu,\rho,\lambda,R\}$, and then $(\ba,\bb)\in\si^2$ by \ref{C2}, \ref{C4}, \ref{C5} or \ref{C8}.

Now suppose $q\neq r$, say $q=0$ and $r=1$.  Then $(i+1,j)\in\th_0^1\subseteq \th_0^2$, and because of the constraints on the row types of~$M^1$, we have $i\geq\min\th_0^1\geq\min\th_0^2$ and $j\geq\min\th_1^1\geq\min\th_1^2$.  Therefore $M_{qi}=M_{rj}$, and $(\ba,\bb)\in\si^2$ as above.
\end{proof}

The assumption about the forbidden row types only came into play in the very last paragraph of the above proof.  Nevertheless, it is easy to see that when $M^1$ is of type \ref{RT2}, \ref{RT5} or \ref{RT7}, the condition $\Pair^1\leqc\Pair^2$ is no longer sufficient for $\cg(\Pair^1)\subseteq\cg(\Pair^2)$:

\begin{exa}\label{eg:noncomp}
Consider the C-pairs
\[
\Pair^1 := 
\Cmatsetup
\begin{array}{|c|c|c|c|c|cc}
\hhline{|-|-|-|-|-|~~}
\cellcolor{delcol}\Delta & \cellcolor{delcol}\Delta & \cellcolor{delcol}\Delta & \cellcolor{delcol}\Delta & \cellcolor{delcol}\cdots&\hspace{2mm}& \Delta_\N \\ \hhline{|-|-|-|-|-|~~}
\cellcolor{delcol}\vvdots & \cellcolor{delcol}\vvdots & \cellcolor{delcol}\vvdots & \cellcolor{delcol}\vvdots & \cellcolor{delcol}\vvdots&&\vvdots  \\ \hhline{|-|-|-|-|-|~~}
\cellcolor{delcol}\Delta & \cellcolor{delcol}\Delta & \cellcolor{delcol}\Delta & \cellcolor{delcol}\Delta & \cellcolor{delcol}\cdots&&\Delta_\N  \\ \hhline{|-|-|-|-|-|~~}
\cellcolor{excepcol}\De&\cellcolor{mucol}\mu&\cellcolor{mucol}\mu&\cellcolor{mucol}\mu& \cellcolor{mucol}\cdots&& \Delta_\N\\ \hhline{|-|-|-|-|-|~~}
\cellcolor{mucol}\mu&\cellcolor{mucol}\mu&\cellcolor{mucol}\mu&\cellcolor{mucol}\mu& \cellcolor{mucol}\cdots&& \Delta_\N\\ \hhline{|-|-|-|-|-|~~}
\end{array}
\normalsize\AND
\Pair^2 := 
\Cmatsetup
\begin{array}{|c|c|c|c|c|cc}
\hhline{|-|-|-|-|-|~~}
\cellcolor{delcol}\Delta & \cellcolor{delcol}\Delta & \cellcolor{delcol}\Delta & \cellcolor{delcol}\Delta & \cellcolor{delcol}\cdots&\hspace{2mm}& \Delta_\N  \\ \hhline{|-|-|-|-|-|~~}
\cellcolor{delcol}\vvdots & \cellcolor{delcol}\vvdots & \cellcolor{delcol}\vvdots & \cellcolor{delcol}\vvdots & \cellcolor{delcol}\vvdots&&\vvdots \\ \hhline{|-|-|-|-|-|~~}
\cellcolor{delcol}\Delta & \cellcolor{delcol}\Delta & \cellcolor{delcol}\Delta & \cellcolor{delcol}\Delta & \cellcolor{delcol}\cdots&&\Delta_\N \\ \hhline{|-|-|-|-|-|~~}
\cellcolor{delcol}\De&\cellcolor{excepcol}\mu&\cellcolor{Rcol}R&\cellcolor{Rcol}R& \cellcolor{Rcol}\cdots&&(2,3)^\sharp\\ \hhline{|-|-|-|-|-|~~}
\cellcolor{Rcol}R & \cellcolor{Rcol}R & \cellcolor{Rcol}R & \cellcolor{Rcol}R & \cellcolor{Rcol}\cdots&&\nabla_\N \\ \hhline{|-|-|-|-|-|~~}
\end{array}
\ .
\]
Then clearly $\Pair^1\leqc\Pair^2$. However, $\cg(\Pair^1)\nsubseteq\cg(\Pair^2)$,
because $((0,\widehat{\alpha}),(1,\alpha))\in \cg(\Pair^1)\setminus \cg(\Pair^2)$ for all $\alpha\in D_1$.
Intuitively, the `problem' is that the relationships between $D_{00}$ and $D_{11}$ indicated by the first matching $\mu$s in $M^1$ have been `broken' by $M^2$.
\end{exa}

Our full description of inclusions will have to deal with the `problem' raised in the example just considered, and also with the exceptional congruences.
To do this, we introduce some notation.
Suppose $M$ is a C-matrix of type \ref{RT2}, \ref{RT5} or~\ref{RT7}.
These are precisely the types that have `initial $\mu$s' in row $0$, by which we mean entries $M_{0j}=\mu$ with $j<\min\th_0$.  These initial $\mu$s are coloured green in the description of row types in Subsection \ref{subsec:Cpairs}.  We define $\mumin_{0}(M)$ to be the position of the first initial $\mu$ in row $0$.  We then define $\mumin_{1}(M)$ to be the position of its `matching $\mu$' in row $1$.  Thus, in the notation of Subsection \ref{subsec:Cpairs}:
\[
\mumin_{0}(M) = \begin{cases}
i &\text{for \ref{RT2} and \ref{RT5}}\\
m-1 &\text{for \ref{RT7},}
\end{cases}
\AND
\mumin_{1}(M) = \begin{cases}
i+1 &\text{for \ref{RT2} and \ref{RT5}}\\
l-1 &\text{for \ref{RT7}.}
\end{cases}
\]
Note that $\mumin_{1}(M)$ need not be the position of the first $\mu$ in row $1$, as we could have $\zeta=\mu$ in types \ref{RT2} and \ref{RT5}.  Also note that in any of types \ref{RT2}, \ref{RT5} or \ref{RT7}, we have
\begin{align}
\label{eq:j-i1} j-i = \mumin_{1}(M)-\mumin_{0}(M) &\implies M_{0i}=M_{1j} &&\text{for all $i,j\geq\mumin_{0}(M)$, and}\\[5mm]
\label{eq:j-i2} j-i = \mumin_{1}(M)-\mumin_{0}(M) &\iff (i+1,j)\in\th_0 &&\text{for all $\mumin_{0}(M)\leq i<\min\th_0$}\\
\nonumber &&&\hspace{3.7mm}\text{and $\mumin_{1}(M)\leq j<\min\th_1$.}
\end{align}
Indeed, these are both easily checked by examining the three types.

Also, to deal with exceptional congruences, for an exceptional C-pair $\Pair=(\Th,M)$, recall that~$\hgt(M)$ is the index of the exceptional row (see Definition~\ref{def:exc}).

\begin{thm}
\label{thm:comparisons}
Let $n\geq 1$, and let $\Pair^1=(\Th^1,M^1)$ and $\Pair^2=(\Th^2,M^2)$ be two C-pairs for $\Ptw{n}$.
\begin{thmenumerate}
\item
\label{it:comp1}
We have $\cg(\Pair^1)\subseteq\cg(\Pair^2)$ if and only if both of the following hold:
\begin{thmsubenumerate}
\item
\label{it:comp1a}
$\Pair^1\leqc\Pair^2$, and
\item
\label{it:comp1b}
if $M^1$ has type \ref{RT2}, \ref{RT5} or \ref{RT7}, then at least one of the following holds:
\begin{enumerate}[label=\textup{(b\arabic*)},leftmargin=10mm]
\item \label{1b1} $\min\th_0^2\leq\mumin_{0}(M^1)$ and $\min\th_1^2\leq\mumin_{1}(M^1)$, or
\item \label{1b2} $M^2$ also has type \ref{RT2}, \ref{RT5} or \ref{RT7} (not necessarily the same as $M^1$), and $\mumin_{1}(M^2)-\mumin_{0}(M^2)=\mumin_{1}(M^1)-\mumin_{0}(M^1)$.
\end{enumerate}
\end{thmsubenumerate}
\item
\label{it:comp2}
When $\Pair^2$ is exceptional, we have $\cg(\Pair^1)\subseteq\cgx(\Pair^2)$ if and only if $\cg(\Pair^1)\subseteq \cg(\Pair^2)$.
\item
\label{it:comp3}
When $\Pair^1$ is exceptional, we have $\cgx(\Pair^1)\subseteq\cg(\Pair^2)$ if and only if all of the following hold, where $q:=\hgt(M^1)$:
\begin{thmsubenumerate}
\item
\label{it:comp3a}
$\cg(\Pair^1)\subseteq\cg(\Pair^2)$,
\item
\label{it:comp3b}
$2\per\th_q^2\mid\per\th_q^1$, and
\item
\label{it:comp3c}
$M_{qi}^2\in\{\S_q,R\}$ for all $i\geq\min\th_q^2$.
\end{thmsubenumerate}
\item
\label{it:comp4}
When both $\Pair^1$ and $\Pair^2$ are exceptional, we have $\cgx(\Pair^1)\subseteq\cgx(\Pair^2)$ if and only if  both of the following hold:
\begin{thmsubenumerate}
\item
\label{it:comp4a} $\cg(\Pair^1)\subseteq\cg(\Pair^2)$, and
\item
\label{it:comp4b} if $\hgt(M^1)=\hgt(M^2)=:q$, then the ratio $\per\th_q^1/\per\th_q^2$ is an odd integer. 
\end{thmsubenumerate}
\end{thmenumerate}
\end{thm}

\begin{proof}
\ref{it:comp1}
If $M^1$ does not have type \ref{RT2}, \ref{RT5} or \ref{RT7}, this follows from Lemma \ref{la:shortcomp}.  Suppose now that $M^1$ has one of types \ref{RT2}, \ref{RT5} or \ref{RT7}, and write $\si^1=\cg(\Pair^1)$ and $\si^2=\cg(\Pair^2)$.  

($\Rightarrow$)  Suppose $\si^1\sub\si^2$.  Lemma \ref{la:shortcomp}\ref{it:shco1} gives $\Pair^1\leqc\Pair^2$.  We have to show additionally that one of \ref{1b1} or \ref{1b2} holds.  Put $i=\mumin_{0}(M^1)$ and $j=\mumin_{1}(M^1)$.  In any of the three row types, we have $(i+1,j)\in\th_0^1$, and also $i<\min\th_0^1$ and $j<\min\th_1^1$.  
%
Thus, $((i,\wh\al),(j,\al))\in\si^1$ via \ref{C8} for all $\al\in D_1$.  Since $\si^1\sub\si^2$, it follows that $\si^2\cap(D_{0i}\times D_{1j})\neq \emptyset$.  But then, by Definition \ref{de:cg} and the specification of row types from Subsection \ref{subsec:Cpairs}, we have either
\[
\text{$[i< \min\th_0^2$ and $j<\min\th_1^2]$ \qquad or \qquad $[i\geq \min\th_0^2$ and $j\geq\min\th_1^2]$.}
\]
The second of these is precisely \ref{1b1}, so we assume the first holds.  Since ${\si^2\cap(D_{0i}\times D_{1j})\neq \emptyset}$, we see by examining the types that $M^2$ has type \ref{RT2}, \ref{RT5} or \ref{RT7}, and that $M_{0i}^2=M_{1j}^2=\mu$; in particular, $i\geq\mumin_{0}(M^2)$ and $j\geq\mumin_{1}(M^2)$.  Since $(i+1,j)\in\th_0^1\sub\th_0^2$, \eqref{eq:j-i2} yields
\[
j-i=\mumin_{1}(M^2)-\mumin_{0}(M^2),
\]
as required.

($\Leftarrow$)  Suppose now that \ref{it:comp1a} and \ref{it:comp1b} both hold.  The proof of Lemma \ref{la:shortcomp}\ref{it:shco2} remains valid until the point in the \ref{C8} case when we appealed to the assumption that $M^1$ was not of type \ref{RT2}, \ref{RT5} or \ref{RT7}.
So we reconnect with the proof at that point, and recall that
\[
((i,\alpha),(j,\beta))\in \si^1\cap (D_{0i}\times D_{1j}) \COMMA
M_{0i}^1=M_{1j}^1=\mu \COMMA
\al=\wh\be \COMMA
(i+1,j)\in\th_0^1\subseteq\th_0^2 ,
\]
and we wish to show that $((i,\alpha),(j,\beta))\in \si^2$.  Furthermore, if $i\geq\min\th_0^1$ and $j\geq\min\th_1^1$ the rest of the proof of Lemma \ref{la:shortcomp}\ref{it:shco2} applies.  So we are left to consider the case in which $i<\min\th_0^1$ and $j<\min\th_1^1$.  Since $M_{0i}^1=M_{1j}^1=\mu$, we have $i\geq\mumin_{0}(M^1)$ and $j\geq\mumin_{1}(M^1)$.  From $M^1\leqc M^2$ we have $M_{0i}^2,M_{1j}^2\in \{\mu,\rho,\lambda,R\}$; see~Figure \ref{fig:M}.

Suppose first that \ref{1b1} holds.  In particular, $\min\th_0^2,\min\th_1^2<\infty$, and also ${i\geq\mumin_{0}(M^1)\geq\min\th_0^2}$ and $j\geq\mumin_{1}(M^1)\geq\min\th_1^2$.  Examining the row types in Subsection \ref{subsec:Cpairs}, it follows that ${M_{0i}^2=M_{1j}^2\in\{\mu,\rho,\lam,R\}}$, and so $((i,\alpha),(j,\beta))\in\si^2$ via \ref{C2}, \ref{C4}, \ref{C5} or \ref{C8}.

Now suppose \ref{1b2} holds.  Combined with \eqref{eq:j-i2} applied to $\si^1$, it follows that
\begin{equation}\label{eq:j-i3}
j-i = \mumin_{1}(M^1)-\mumin_{0}(M^1) = \mumin_{1}(M^2)-\mumin_{0}(M^2).
\end{equation}
Since $M_{0i}^2,M_{1j}^2\neq \De$, we have $i,j\geq\mumin_{0}(M^2)$.  It follows from \eqref{eq:j-i3} and \eqref{eq:j-i1} that $M_{0i}^2=M_{1j}^2$, and from \eqref{eq:j-i3}, and inspection of the types \ref{RT2}, \ref{RT5} and \ref{RT7}, that either 
\[
[i<\min\th_0^2 \text{ and } j<\min\th_1^2] \qquad \text{or} \qquad [i\geq\min\th_0^2 \text{ and } j\geq\min\th_1^2].
\]
We then have $((i,\alpha),(j,\beta))\in\si^2$ via \ref{C2}, \ref{C4}, \ref{C5} or \ref{C8}.  
\medskip

\ref{it:comp2}
($\Leftarrow$)
This follows immediately from $\cg(\Pair^2)\subseteq\cgx(\Pair^2)$; see Definition \ref{def:exc}.
\smallskip

($\Rightarrow$)
We need to show that no `exceptional pair' $((i,\alpha),(j,\beta))\in \cgx(\Pair^2)\setminus\cg(\Pair^2)$ belongs to~$\cg(\Pair^1)$.  But this follows quickly from $M_{qi}^1,M_{qj}^1\leqc \A_q$ and the definition of $\cg(\Pair^1)$.

\medskip

\ref{it:comp3}
For this part we write $\th_q^1=(m,m+2d)^\sharp$.
\smallskip

($\Rightarrow$)
Clearly \ref{it:comp3a} holds, and it follows from part \ref{it:comp1} that $\Pair^1\leqc \Pair^2$.  Recall that all $M_{qi}^2$ (${i\geq \min\th_q^2}$) are equal.
Fix some $(\al,\be)\in\nu_{\S_q}\sm\nu_{\A_q}$, so $((m,\al),(m+d,\be))\in\cgx(\Pair^1)\sub\cg(\Pair^2)$.  Since $M^1\leqc M^2$ and $M_{qm}^1=M_{q,m+d}^1=\A_q$, it follows that $((m,\al),(m+d,\be))\in\cg(\Pair^2)$ via \ref{C2} or \ref{C3}, with $M_{qm}^2=M_{q,m+d}^2\in\{\S_q,R\}$; this shows that~\ref{it:comp3c} holds.  
It also follows from $M_{qm}^2\in\{\S_q,R\}$ that $((m,\al),(m,\be))\in\cg(\Pair^2)$, and so $((m,\be),(m+d,\be))\in\cg(\Pair^2)$ by transitivity.  Thus, $(m,m+d)\in\th_q^2$, and so $\per\th_q^2\mid d$, which gives~\ref{it:comp3b}.
\smallskip

($\Leftarrow$)
By \ref{it:comp3b} we have $\per\th_q^2\mid d$.  By \ref{it:comp3a} and part \ref{it:comp1} we have $\th_q^1\sub\th_q^2$, so $\min\th_q^2\leq m$.  It follows that $(m,m+d)^\sharp\sub\th_q^2$.
Now let $((i,\alpha),(j,\beta))\in\cgx(\Pair^1)\setminus\cg(\Pair^1)$.  Then \ref{C9} is responsible for this pair; consequently, we have $i,j\geq m\geq\min\th_q^2$, $(i,j)\in(m,m+d)^\sharp\sub\th_q^2$, and $(\al,\be)\in\nu_{\S_q}\sm\nu_{\A_q}$.  
By \ref{it:comp3c} we have $M_{qi}^2=M_{qj}^2\in\{\S_q,R\}$, and so $((i,\alpha),(j,\beta))\in\cg(\Pair^2)$ via \ref{C2} or \ref{C3}.
\medskip

\ref{it:comp4}
We again let $q:=\hgt(M^1)$, and write $\th_q^1=(m,m+2d)^\sharp$.
\smallskip

($\Rightarrow$)
That $\cg(\Pair^1)\subseteq\cg(\Pair^2)$ follows from part \ref{it:comp2}.  
If $\hgt(M^2)\neq q$ we are finished. So suppose $\hgt(M^2)=q$.
Part \ref{it:comp1} then gives $\th_q^1\sub\th_q^2$, so that $\min\th_q^2\leq m$ and $\per\th_q^2\mid2d$.
Let $(\al,\be)\in\nu_{\S_q}\sm\nu_{\A_q}$, so $((m,\al),(m+d,\be))\in\cgx(\Pair^1)\sub\cgx(\Pair^2)$.  
Since $M^2$ is exceptional, and since $m\geq\min\th_q^2$ we must have $M_{qm}^2=M_{q,m+d}^2=\A_q$.  Since $(\alpha,\beta)\not\in\nu_{\A_q}$, we must have ${d=(m+d)-m\equiv e\pmod{2e}}$, where $2e:=\per\th_q^2$.  It quickly follows that $\per\th_q^1/\per\th_q^2=d/e$ is an odd integer.
\smallskip

($\Leftarrow$)
From $\cg(\Pair^1)\subseteq\cg(\Pair^2)$ it follows that $\cg(\Pair^1)\subseteq \cgx(\Pair^2)$, and also that $m\geq\min\th_q^2$ using part \ref{it:comp1}.
Now let $((i,\alpha),(j,\beta))\in \cgx(\Pair^1)\setminus\cg(\Pair^1)$.  This must be via \ref{C9}, so we have $i,j\geq m\geq\min\th_q^2$, $j-i\equiv d\pmod{2d}$ and $(\al,\be)\in\nu_{\S_q}\sm\nu_{\A_q}$.  

Suppose first that $\hgt(M^2)\neq q$, i.e.~row $q$ is not exceptional in $M^2$.   Since $i,j\geq\min\th_q^2$, we see that $M_{qi}^2=M_{qj}^2$ is the `terminal symbol' of row $q$.  This is not an $N$-symbol as row~$q$ is not exceptional,
and it is not $\De$ as $M_{qi}^2\geqc M_{qi}^1=\A_q$.  Thus, $M_{qi}^2=M_{qj}^2=R$, so ${((i,\alpha),(j,\beta))\in \cg(\Pair^2)}$ via~\ref{C2}.

Now suppose $\hgt(M^2)=q$, and let $2e:=\per\th_q^2$.  Since $d/e=\per\th_q^1/\per\th_q^2$ is an odd integer, it quickly follows that $j-i\equiv e\pmod{2e}$.  But then $((i,\alpha),(j,\beta))\in \cgx(\Pair^2)$ via \ref{C9}, completing the proof of this case, and of the theorem.
\end{proof}

\begin{rem}
In Example \ref{eg:noncomp} we exhibited C-pairs $\Pair^1\leqc\Pair^2$ with $\cg(\Pair^1)\not\sub\cg(\Pair^2)$.  Examining Theorem \ref{thm:comparisons}\ref{it:comp1}, we see that $M^1$ has type \ref{RT2}, but items \ref{1b1} and~\ref{1b2} both fail: \ref{1b1} because $\min\th_1^2>\mumin_{1}(M^1)$, and \ref{1b2} because $M^2$ has type \ref{RT6}.
\end{rem}

\section{\boldmath Congruences of $d$-twisted partition monoids}
\label{sec:dtwist}

Recall from Subsection \ref{subsec:finite} that for $n,d\geq0$, the $d$-twisted partition monoid is defined as the Rees quotient
\[
\Ptw{n,d}:=\Ptw{n}/R_I \qquad\text{where $I:=I_{n,d+1}$.}
\]
We now apply the main results of the preceding sections to classify the congruences on~$\Ptw{n,d}$, and characterise the inclusion order in the lattice $\Cong(\Ptw{n,d})$.

As explained in Example \ref{ex:ReesC}, the Rees congruence $R_I$ has C-pair representation $R_I = \cg(\Pair)$, where
\[
\Pair:=
\Cmatsetup
\begin{array}{|c|c|c|c|c|c|c|cc}
\hhline{|-|-|-|-|-|-|-|~~}
\cellcolor{delcol}\Delta & \cellcolor{delcol}\Delta & \cellcolor{delcol}\cdots & \cellcolor{delcol}\Delta & \cellcolor{Rcol}R & \cellcolor{Rcol}R& \cellcolor{Rcol}\cdots&\hspace{2mm}& (d+1,d+2)^\sharp
\\ \hhline{|-|-|-|-|-|-|-|~~}
\cellcolor{delcol}\vvdots & \cellcolor{delcol}\vvdots & \cellcolor{delcol}\vvdots & \cellcolor{delcol}\vvdots & \cellcolor{Rcol}\vvdots & \cellcolor{Rcol}\vvdots& \cellcolor{Rcol}\vvdots&&\vvdots
\\ \hhline{|-|-|-|-|-|-|-|~~}
\cellcolor{delcol}\Delta & \cellcolor{delcol}\Delta & \cellcolor{delcol}\cdots & \cellcolor{delcol}\Delta & \cellcolor{Rcol}R & \cellcolor{Rcol}R& \cellcolor{Rcol}\cdots&& (d+1,d+2)^\sharp
\\ \hhline{|-|-|-|-|-|-|-|~~}
\cellcolor{delcol}\Delta & \cellcolor{delcol}\Delta & \cellcolor{delcol}\cdots & \cellcolor{delcol}\Delta & \cellcolor{Rcol}R & \cellcolor{Rcol}R& \cellcolor{Rcol}\cdots&& (d+1,d+2)^\sharp
\\ \hhline{|-|-|-|-|-|-|-|~~}
\multicolumn{1}{c}{}&\multicolumn{1}{c}{}&\multicolumn{1}{c}{}&\multicolumn{1}{c}{\text{\scriptsize $d$}}&\multicolumn{1}{c}{}&\multicolumn{1}{c}{}
\end{array}
\ .
\]
By the Correspondence Theorem (see for example \cite[Theorem 6.20]{BS1981}), the congruence lattice $\Cong(\Ptw{n,d})$ is isomorphic to the interval
$[R_I,\nabla_{\Ptw{n}}]$ in $\Cong(\Ptw{n})$.

With the help of our description of inclusion in Theorem \ref{thm:comparisons}, let us look at this interval more closely. So
consider some congruence
$\si\in [R_I,\nabla_{\Ptw{n}}]$, and let $(\Th,M)$ be the C-pair associated to~$\si$.
By Theorem \ref{thm:comparisons} we must have:
\begin{itemize}
\item
$M_{qi}=R$ for all $q\in \bnz$ and $i\geq d+1$;
\item
$\min\th_q\leq d+1$ and $\per\th_q=1$ for all $q\in \bnz$;
\item
$M_{qi}\in \{ \Delta,R\}\cup \set{ N}{ \{\id_q\}\neq N\unlhd \S_q}$ for all $2\leq q\leq n$ and $i\leq d$;
\item
$M_{1i}\in\{\Delta,\muup,\mudown,\mu, R\}$ for all $i\leq d$;
\item
$M_{0i}\in \{\Delta,\mu,R\}$ for all $i\leq d$.
\end{itemize}
It is significant to observe that  that the C-pair $(\Th,M)$ cannot be exceptional, as every row ends with an infinite sequence of $R$s; hence ${\si=\cg(\Th,M)}$. Furthermore, $M$ has no $\lam$ or $\rho$ entries.
Theorem \ref{thm:comparisons} also gives the converse:
if the above conditions are satisfied then $\cg(\Th,M)$ does belong to the interval $[R_I,\nabla_{\P_n}]$.
Furthermore, the value $\min\th_q$ can be deduced from the matrix: it is the first point where $R$ makes an appearance in row $q$.
It therefore follows that $\si$ can be encoded by the $\bnz\times \bdz$ submatrix $M'$ consisting of columns $0,1,\dots,d$ of $M$.
In this context we will write $\min_q(M')$ for the value $\min\th_q$.

Another consequence of the above conditions is that not all row types \ref{RT1}--\ref{RT10} are possible for the matrix $M$, and those that are possible have additional restrictions.
Specifically:
\begin{itemize}
\item
Row types \ref{RT1}, \ref{RT2}, \ref{RT3}, \ref{RT8} and \ref{RT9} do not occur.
\item
In row types \ref{RT4}--\ref{RT7} we have $\xi=R$.
\end{itemize}

Restricting to the $\bnz\times \bdz$ submatrix $M'$, we arrive at the following \emph{finitary row types:}
\bigskip

\newcounter{fTR}
\renewcommand{\thefTR}{\textup{fRT\arabic{fTR}}}
\refstepcounter{fTR}\label{fR1} 
\refstepcounter{fTR}\label{fR2} 
\refstepcounter{fTR}\label{fR3} 
\refstepcounter{fTR}\label{fR4} 
\refstepcounter{fTR}\label{fR5} 
\setcounter{fTR}{0}

\begin{align*}
&
\textbf{\stepcounter{fTR}\thefTR}
&&
\raisebox{-8pt}{
$\begin{array}{r|c|c|c|c|c|c|c|}\hhline{~|-|-|-|-|-|-|}
\renewcommand{\arraystretch}{2}
\text{\scriptsize 1} &\cellcolor{delcol}\Delta &\cellcolor{delcol}\dots & \cellcolor{delcol}\Delta & \cellcolor{Rcol}R&\cellcolor{Rcol}\dots& \cellcolor{Rcol} R\\ \hhline{~|-|-|-|-|-|-|}
\text{\scriptsize 0}&\cellcolor{delcol}\Delta &\cellcolor{delcol}\dots &\cellcolor{delcol}\Delta & \cellcolor{Rcol}R&\cellcolor{Rcol}\dots& \cellcolor{Rcol} R\\ \hhline{~|-|-|-|-|-|-|}
\multicolumn{1}{c}{}&\multicolumn{1}{c}{}&\multicolumn{1}{c}{}&\multicolumn{1}{c}{}&\multicolumn{1}{c}{\text{\scriptsize $k$}}&\multicolumn{1}{c}{}&\multicolumn{1}{c}{}
\end{array}$
}
&&
\begin{array}{l} 0\leq k\leq d+1, \end{array}
 \\
 &
\textbf{\stepcounter{fTR}\thefTR}
 &&
 \raisebox{-8pt}{
$\begin{array}{r|c|c|c|c|c|c|c|c|c|c|c|c|}\hhline{~|-|-|-|-|-|-|-|-|-|-|-|}
\renewcommand{\arraystretch}{2}
\text{\scriptsize 1} &\cellcolor{delcol}\Delta &\cellcolor{delcol}\dots & \cellcolor{delcol}\Delta & 
\cellcolor{excepcol}\zeta & \cellcolor{mucol} \mu & \cellcolor{mucol} \dots &\cellcolor{mucol} \mu &
\cellcolor{mucol} \mu &
\cellcolor{Rcol} R&\cellcolor{Rcol}\dots& \cellcolor{Rcol} R\\ \hhline{~|-|-|-|-|-|-|-|-|-|-|-|}
\text{\scriptsize 0}&\cellcolor{delcol}\Delta &\cellcolor{delcol}\dots &\cellcolor{delcol}\Delta & 
\cellcolor{mucol}\mu & \cellcolor{mucol} \mu & \cellcolor{mucol} \dots &\cellcolor{mucol} \mu &
\cellcolor{Rcol} R &\cellcolor{Rcol} R&\cellcolor{Rcol}\dots& \cellcolor{Rcol} R\\ \hhline{~|-|-|-|-|-|-|-|-|-|-|-|}
\multicolumn{1}{c}{}&\multicolumn{1}{c}{}&\multicolumn{1}{c}{}&\multicolumn{1}{c}{}&\multicolumn{1}{c}{\text{\scriptsize $i$}}&\multicolumn{1}{c}{}&\multicolumn{1}{c}{}&\multicolumn{1}{c}{}&\multicolumn{1}{c}{\text{\scriptsize $k$}}&\multicolumn{1}{c}{}&\multicolumn{1}{c}{}
\end{array}$
}
&&
\begin{array}{l}
0\leq i< k\leq d,\\
\zeta\in \{\mu,\muup,\mudown,\Delta\},
\end{array}
\\
&
\textbf{\stepcounter{fTR}\thefTR}
&&
\raisebox{-8pt}{
$\begin{array}{r|c|c|c|c|c|c|c|c|c|c|}\hhline{~|-|-|-|-|-|-|-|-|-|-|}
\renewcommand{\arraystretch}{2}
\text{\scriptsize 1} &\cellcolor{delcol}\Delta &\cellcolor{delcol}\dots & \cellcolor{delcol}\Delta & 
\cellcolor{delcol}\Delta & \cellcolor{delcol}\dots & \cellcolor{delcol}\Delta &
\cellcolor{excepcol} \zeta & \cellcolor{Rcol} R& \cellcolor{Rcol}\dots&\cellcolor{Rcol} R
\\ \hhline{~|-|-|-|-|-|-|-|-|-|-|}
\text{\scriptsize 0}&\cellcolor{delcol}\Delta &\cellcolor{delcol}\dots &\cellcolor{delcol}\Delta & 
\cellcolor{Rcol} R &\cellcolor{Rcol}\dots&\cellcolor{Rcol} R&\cellcolor{Rcol} R&\cellcolor{Rcol} R& \cellcolor{Rcol}\dots&\cellcolor{Rcol} R
\\ \hhline{~|-|-|-|-|-|-|-|-|-|-|}
\multicolumn{1}{c}{}&\multicolumn{1}{c}{}&\multicolumn{1}{c}{}&\multicolumn{1}{c}{}&\multicolumn{1}{c}{\text{\scriptsize $k$}}&\multicolumn{1}{c}{}&\multicolumn{1}{c}{}&\multicolumn{1}{c}{}&\multicolumn{1}{c}{\text{\scriptsize $l$}}&\multicolumn{1}{c}{}&\multicolumn{1}{c}{}
\end{array}$
}
&&
\begin{array}{l}
0\leq k<l\leq d+1,\\ 
\zeta\in \{\mu,\muup,\mudown,\Delta\},
\end{array}
\displaybreak[2]
\\
&
\textbf{\stepcounter{fTR}\thefTR}
&&
\raisebox{-8pt}{
$\begin{array}{r|c|c|c|c|c|c|c|c|c|c|c|}\hhline{~|-|-|-|-|-|-|-|-|-|-|-|}
\renewcommand{\arraystretch}{2}
\text{\scriptsize 1} &\cellcolor{delcol}\Delta &\cellcolor{delcol}\dots & \cellcolor{delcol}\Delta  & \cellcolor{delcol}\Delta &
\cellcolor{delcol}\Delta & \cellcolor{delcol}\dots & \cellcolor{delcol}\Delta &
\cellcolor{mucol} \mu & \cellcolor{Rcol} R& \cellcolor{Rcol}\dots & \cellcolor{Rcol} R 
\\ \hhline{~|-|-|-|-|-|-|-|-|-|-|-|}
\text{\scriptsize 0}&\cellcolor{delcol}\Delta &\cellcolor{delcol}\dots &\cellcolor{delcol}\Delta  
&\cellcolor{mucol}\mu&
\cellcolor{Rcol} R &\cellcolor{Rcol}\dots&\cellcolor{Rcol} R&\cellcolor{Rcol} R&\cellcolor{Rcol} R& \cellcolor{Rcol}\dots & \cellcolor{Rcol} R
\\ \hhline{~|-|-|-|-|-|-|-|-|-|-|-|}
\multicolumn{1}{c}{}&\multicolumn{1}{c}{}&\multicolumn{1}{c}{}&\multicolumn{1}{c}{}&\multicolumn{1}{c}{}&\multicolumn{1}{c}{\text{\scriptsize $k$}}&\multicolumn{1}{c}{}&\multicolumn{1}{c}{}&\multicolumn{1}{c}{}&\multicolumn{1}{c}{\text{\scriptsize $l$}}
\end{array}$
}
&&
\begin{array}{l}0<k<l-1\leq d,\end{array}
\\
&
\textbf{\stepcounter{fTR}\thefTR}
&&
\raisebox{-5pt}{
$\begin{array}{r|c|c|c|c|c|c|c|c|c|c|}\hhline{~|-|-|-|-|-|-|-|-|-|-|}
\renewcommand{\arraystretch}{2}
\text{\scriptsize $q$} &\cellcolor{delcol}\Delta &\cellcolor{delcol}\dots & \cellcolor{delcol}\Delta  & 
\cellcolor{Ncol} N_i &\cellcolor{Ncol} N_{i+1} & \cellcolor{Ncol}\dots & \cellcolor{Ncol} N_{k-1} &
\cellcolor{Rcol} R & \cellcolor{Rcol}\dots & \cellcolor{Rcol} R
\\ \hhline{~|-|-|-|-|-|-|-|-|-|-|}
\multicolumn{1}{c}{}&\multicolumn{1}{c}{}&\multicolumn{1}{c}{}&\multicolumn{1}{c}{}&\multicolumn{1}{c}{\text{\scriptsize $i$}}&\multicolumn{1}{c}{}&\multicolumn{1}{c}{}&\multicolumn{1}{c}{}&\multicolumn{1}{c}{\text{\scriptsize $k$}}&\multicolumn{1}{c}{}&\multicolumn{1}{c}{}
\end{array}$
}
&&
\begin{array}{l}
q\geq2,\\
0\leq i \leq k\leq d+1,\\
\{\id_q\}\not=N_i\leq \dots \leq N_{k-1},\\
N_i,\dots,N_{k-1}\unlhd \S_q.
\end{array}
\end{align*}

\begin{defn}[\bf Finitary C-matrix]
\label{de:fC}
A \emph{finitary C-matrix}, or \emph{fC-matrix} for short, is a matrix ${M=(M_{qi})_{\bnz\times \bdz}}$
with entries from $\{ \Delta,\muup,\mudown,\mu,R\}\cup\set{N}{\{\id_q\}\neq N\normal\S_q,\ 2\leq q\leq n}$ such that rows $0$ and $1$ are of one of  types \ref{fR1}--\ref{fR4}, each row $q\geq 2$ is of type \ref{fR5},
and $M$ satisfies the verticality conditions \ref{V1} and \ref{V2}.
\end{defn}

Switching to the representation of $\Ptw{n,d}$ as $(\bdz\times\P_n)\cup\{\zero\}$, with product given in \eqref{eq:multwk}, Definition \ref{de:cg} translates into the following description of the congruence defined by an fC-matrix.

\begin{defn}[\bf Congruence corresponding to a finitary C-matrix]
\label{de:fCcong}
The \emph{congruence associated} with a finitary C-matrix $M$ is the relation $\cg(M)$ on $\Ptw{n,d}$ consisting of all pairs ${((i,\alpha),(j,\beta))\in \Ptw{n,d}\times \Ptw{n,d}}$ such that one of the following holds, writing $q=\rank\alpha$ and ${r=\rank\beta}$:
\begin{enumerate}[label=\textsf{(fC\arabic*)}, widest=(C8), leftmargin=10mm]
\item
$M_{qi}=M_{rj}=\De$, $i=j$ and $\alpha=\beta$;
\item
$M_{qi}=M_{rj}=R$;
\item
$M_{qi}=M_{rj}=N$, $i=j$, $\alpha\rH\beta$ and $\pd(\alpha,\beta)\in N$;
\item
$M_{qi}=M_{rj}=\mudown$, $\widehat{\alpha}=\widehat{\beta}$ and $\alpha\rL\beta$;
\item
$M_{qi}=M_{rj}=\muup$, $\widehat{\alpha}=\widehat{\beta}$ and $\alpha\rR\beta$;
\item
$M_{qi}=M_{rj}=\mu$, $\widehat{\alpha}=\widehat{\beta}$, and either $(q,i)=(r,j)$ or $i-j = \min_q(M)-\min_r(M)$;
\end{enumerate}
as well as the pairs:
\begin{enumerate}[label=\textsf{(fC\arabic*)}, widest=(C8), leftmargin=10mm]
\addtocounter{enumi}{6}
\item
$((i,\alpha),\zero),(\zero,(i,\alpha))$ with $M_{qi}=R$;
\item
$(\zero,\zero)$.
\end{enumerate}
\end{defn}

Putting all these observations together, and combining with Theorem \ref{thm:main} we obtain the following classification of the  congruences on $\Ptw{n,d}$:

\begin{thm}
\label{thm:finmain}
For $n\geq1$ and $d\geq0$, the congruences on the $d$-twisted partition monoid $\Ptw{n,d}$ are precisely $\fcg(M)$, where $M$ is any fC-matrix. \epfres
\end{thm}

The description of inclusion given in Theorem \ref{thm:comparisons} also becomes much simpler, in that only part~\ref{it:comp1} applies.
However, the complication caused by the matching $\mu$s in rows $0$ and $1$ persists.  The following statement uses the $\mumin_{0}(M)$ and $\mumin_{1}(M)$ notation introduced before Theorem~\ref{thm:comparisons}, which applies to fC-matrices of types \ref{fR2} and \ref{fR4}; in these types we have $\mumin_{1}(M)-\mumin_{0}(M)=\min_1(M)-\min_0(M)$.

\begin{thm}
\label{thm:fincomp}
Let $n\geq 1$ and $d\geq 0$, and let 
$M^1$ and $M^2$ be any two fC-matrices for $\Ptw{n,d}$. Then $\fcg(M^1)\sub\fcg(M^2)$ if and only if both of the following hold:
\begin{thmsubenumerate}
\item
\label{it:fc1}
$M^1\leqc M^2$;
\item
\label{it:fc2}
If $M^1$ has type \ref{fR2} or \ref{fR4}, then at least one of the following holds:
\begin{enumerate}[label=\textup{(b\arabic*)},leftmargin=10mm]
\item \label{fb1} $\min_0(M^2)\leq\mumin_{0}(M^1)$ and $\min_1(M^2)\leq\mumin_{1}(M^1)$, or
\item \label{fb2} $M^2$ also has type \ref{fR2} or \ref{fR4}, and $\min_1(M^2)-\min_0(M^2)=\min_1(M^1)-\min_0(M^1)$.  \epfres
\end{enumerate}
\end{thmsubenumerate}
\end{thm}

The very special case of the $0$-twisted partition monoid $\Ptw{n,0}$ deserves a separate mention,
not least because it provided an
early source of motivation for the work presented here, in the form of a question
V. Mazorchuk asked the first author at the 2018 Rhodesfest conference in Bar Ilan.
Mazorchuk observed that the ideals of $\Ptw{n,0}$ form a chain, and asked whether the methods of 
\cite{EMRT2018,ER2020} can be applied to describe its congruences. This indeed is the case, but such a description can also be derived as a (very) special case of Theorems \ref{thm:finmain} and \ref{thm:fincomp}.

Indeed, when $d=0$ the fC-matrices are just columns.  There are two basic patterns (with ${\{\id_q\}\neq N\normal\S_q}$ in row $q\geq2$ in the second), as well as four `sporadic' forms:
\begin{equation}\label{eq:fC0}
{
\Cmatsetup
\begin{array}{|c|}\hline
\renewcommand{\arraystretch}{2}
\cellcolor{delcol}\Delta \\ \hline
\cellcolor{delcol}\vvdots \\ \hline
\cellcolor{delcol}\Delta \\ \hline
\cellcolor{delcol}\Delta \\ \hline
\cellcolor{Rcol}R \\ \hline
\cellcolor{Rcol}\vvdots \\ \hline
\cellcolor{Rcol}R \\ \hline
\end{array}
\ \COMMA 
\begin{array}{c|c|}\hhline{~|-}
 \renewcommand{\arraystretch}{2}
& \cellcolor{delcol}\Delta \\ \hhline{~|-}
& \cellcolor{delcol}\vvdots \\ \hhline{~|-}
& \cellcolor{delcol}\Delta \\ \hhline{~|-}
\text{\scriptsize $q$} & \cellcolor{Ncol}N\\ \hhline{~|-}
& \cellcolor{Rcol}R \\ \hhline{~|-}
& \cellcolor{Rcol}\vvdots \\ \hhline{~|-}
& \cellcolor{Rcol}R \\ \hhline{~|-}
\end{array}
\qquad\AND\qquad
\begin{array}{|c|}\hline
\renewcommand{\arraystretch}{2}
\cellcolor{delcol}\Delta \\ \hline
\cellcolor{delcol}\vvdots \\ \hline
\cellcolor{delcol}\Delta \\ \hline
\cellcolor{delcol}\Delta \\ \hline
\cellcolor{excepcol}\muup \\ \hline
\cellcolor{Rcol}R \\ \hline
\end{array}
\ \COMMA
\begin{array}{|c|}\hline
\renewcommand{\arraystretch}{2}
\cellcolor{delcol}\Delta \\ \hline
\cellcolor{delcol}\vvdots \\ \hline
\cellcolor{delcol}\Delta \\ \hline
\cellcolor{delcol}\Delta \\ \hline
\cellcolor{excepcol}\mudown \\ \hline
\cellcolor{Rcol}R \\ \hline
\end{array}
\ \COMMA
\begin{array}{|c|}\hline
\renewcommand{\arraystretch}{2}
\cellcolor{delcol}\Delta \\ \hline
\cellcolor{delcol}\vvdots \\ \hline
\cellcolor{delcol}\Delta \\ \hline
\cellcolor{delcol}\Delta \\ \hline
\cellcolor{excepcol}\mu \\ \hline
\cellcolor{Rcol}R \\ \hline
\end{array}
\ \COMMA
\begin{array}{|c|}\hline
\renewcommand{\arraystretch}{2}
\cellcolor{delcol}\Delta \\ \hline
\cellcolor{delcol}\vvdots \\ \hline
\cellcolor{delcol}\Delta \\ \hline
\cellcolor{Ncol}\S_2 \\ \hline
\cellcolor{excepcol}\mu \\ \hline
\cellcolor{Rcol}R \\ \hline
\end{array}
\ .
}
\end{equation}
Identifying $\Ptw{n,0}$ with the set $\P_n\cup\{\zero\}$, with product given in \eqref{eq:multw0}, the simple forms of the fC-matrices in \eqref{eq:fC0} lead to a neat description of the congruences of $\Ptw{n,0}$, which dispenses with matrices altogether, and which we now give.  For the statement, we define $\rank(\zero)=-\infty$.  We also slightly abuse notation, by momentarily re-using symbols to give convenient names to the congruences.

\begin{thm}
\label{thm:CongPn0}
For $n\geq2$, the congruences on the 0-twisted partition monoid~$\Ptw{n,0}$ are precisely:
\begin{itemize}
\item
the Rees congruences $R_q:= \bigset{ (\alpha,\beta)\in\Ptw{n,0}\times\Ptw{n,0}}{ \alpha=\beta\text{ or } \rank\alpha,\rank\beta\leq q}$ for $q\in\{-\infty,0,\dots,n\}$, including $\nab_{\Ptw{n,0}} = R_n$ and $\De_{\Ptw{n,0}}=R_{-\infty}$,
\item
the relations $R_N:=R_{q-1}\cup\nu_N$ for $q\in\{2,\dots,n\}$ and $\{\id_q\}\neq N\unlhd \S_q$,
\item
the relations 
\begin{align*}
\muup &:= R_0 \cup \bigset{ (\alpha,\beta)\in D_1\times D_1}{ \widehat{\alpha}=\widehat{\beta},\ \al\rR\be}, \\
\mudown &:= R_0 \cup \bigset{ (\alpha,\beta)\in D_1\times D_1}{ \widehat{\alpha}=\widehat{\beta},\ \al\rL\be}, \\
\mu &:= R_0 \cup \bigset{ (\alpha,\beta)\in D_1\times D_1}{ \widehat{\alpha}=\widehat{\beta}}, \\
\mu_{\S_2} &:= \mu\cup\nu_{\S_2}.
\end{align*}
\end{itemize}
The congruence lattice $\Cong(\Ptw{n,0})$ is shown in Figure \ref{fig:CongPn0}.  \epfres
\end{thm}

\begin{figure}[ht]
\begin{center}
\begin{tikzpicture}[scale=0.9]
\begin{scope}[minimum size=7mm,inner sep=0.5pt, outer sep=1pt]
\node (D) at (0,0) [draw=blue,fill=white,circle,line width=1.2pt] {\footnotesize $\De$};
\node (R0) at (0,1.4) [draw=blue,fill=white,circle,line width=1.2pt] {\footnotesize $R_0$};
\node (mud) at (-1,2.6) [draw=red,fill=white,circle,line width=1.2pt] {\footnotesize $\mudown$};
\node (muu) at (1,2.6) [draw=red,fill=white,circle,line width=1.2pt] {\footnotesize $\muup$};
\node (mu) at (0,3.8) [draw=red,fill=white,circle,line width=1.2pt] {\footnotesize $\mu$};
\node (R1) at (-1,5) [draw=blue,fill=white,circle,line width=1.2pt] {\footnotesize $R_1$};
\node (muS2) at (1,5) [draw=red,fill=white,circle,line width=1.2pt] {\footnotesize $\mu_{S_2}$};
\node (S2) at (0,6.2) [draw,circle]  {\footnotesize $R_{\S_2}$};
\node (R2) at (0,7.6) [draw=blue,fill=white,circle,line width=1.2pt]{\footnotesize$R_2$};
\node (A3) at (0,9) [draw,circle] {\footnotesize $R_{\A_3}$};
\node (Sn) at (0,11.8)[draw,circle]  {\footnotesize $R_{\S_n}$};
\node (Rn) at (0,13.2)  [draw=blue,fill=white,circle,line width=1.2pt]  {\footnotesize $R_n$};
\node (dots) at (0,10.65) {\footnotesize $\vdots$};
\node () at (.95,13.25) {\footnotesize $=\nab$};
\end{scope}
\draw (D)--(R0) (R0)--(mud) (R0)--(muu) (mud)--(mu) (muu)--(mu) (mu)--(R1) (mu)--(muS2) (R1)--(S2) (muS2)--(S2)
(S2)--(R2) (R2)--(A3) (A3)--(0,10) (0,11)--(Sn)--(Rn);
\end{tikzpicture}
\caption{The Hasse diagram of $\Cong(\Ptw{n,0})$; Rees congruences are indicated in blue outline, `sporadic' congruences in red, and we abbreviate $\De=\De_{\Ptw{n,0}}$ and $\nab=\nab_{\Ptw{n,0}}$.  
}
\label{fig:CongPn0}
\end{center}
\end{figure}
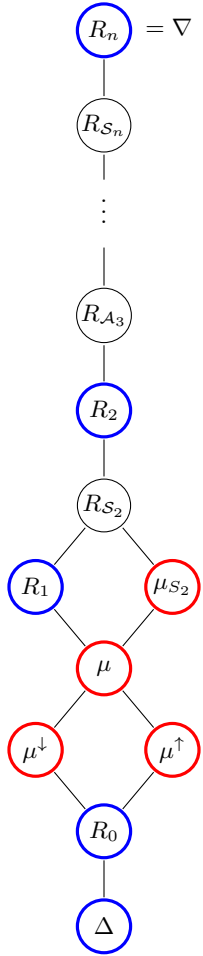

It is interesting to compare the structures of the lattices $\Cong(\P_n)$ and $\Cong(\Ptw{n,0})$ shown in Figures \ref{fig:CongPn} and \ref{fig:CongPn0}.  There are certainly some obvious similarities.  Both have a chain at the top of the lattice, consisting of the interval $[R_{\S_2},\nab]$, and below this both lattices feature four-element diamond sublattices; they differ, however, in the number of these diamonds, as well as the way they connect to each other.

For $d>0$ the lattices $\Cong(\Ptw{n,d})$ are much more complicated, even for small values of $n$ and~$d$.
As an illustration, Figure \ref{fig:CongP32} shows the Hasse diagram of $\Cong(\Ptw{3,2})$.
By the Correspondence Theorem, this lattice contains a principal filter isomorphic to $\Cong(\Ptw{3,1})$,
which in turn contains a copy of $\Cong(\Ptw{3,0})$; these two lattices are highlighted in the figure, as are the Rees congruences.
The figure was produced using the \textsf{Digraphs} package \cite{Digraphs} in \textsf{GAP} \cite{GAP4}, as well as Graphviz \cite{Graphviz} and dot2tex \cite{dot2tex}.  

The \textsf{Semigroups} package \cite{Semigroups} for \textsf{GAP} can directly compute the congruences of $\Ptw{n,d}$ for relatively small $n$ and $d$, by performing a simple but time-consuming search, and this formed an important part of our initial investigations on the topic.  However, Figure \ref{fig:CongP32} was created using our combinatorial description of the lattice via fC-matrices encapsulated by Theorems \ref{thm:finmain} and~\ref{thm:fincomp}, which allows one to deal with larger $n$ and $d$.

For fixed $n$ we have a sequence of lattices $\Cong(\Ptw{n,d})$ for $d=0,1,2,3,4,\ldots$, and as each is contained in the next, $\Cong(\Ptw{n})$ contains the direct limit of this chain:
\[
\bigcup_{d\in\N} [R_{I_{n,d+1}},\nab_{\Ptw n}].
\]
However, this limit is not $\Cong(\Ptw n)$ itself.  Indeed, the congruences belonging to this sublattice can be characterised in many equivalent ways: for example,
\bit
\item those containing a Rees congruence of the form $R_{I_{ni}}$ for some $i\in\N$, or
\item those whose associated C-matrix has the top row of type \ref{RT10}, or
\item those whose projection to $\P_n$ is the universal congruence.
\eit
A more detailed and systematic analysis of the properties of the lattices $\Cong(\Ptw{n})$ and $\Cong(\Ptw{n,d})$ will be the subject of a future article \cite{ERtwisted2}.

\begin{figure}[htp]
\begin{center}
\includegraphics[width=\textwidth]{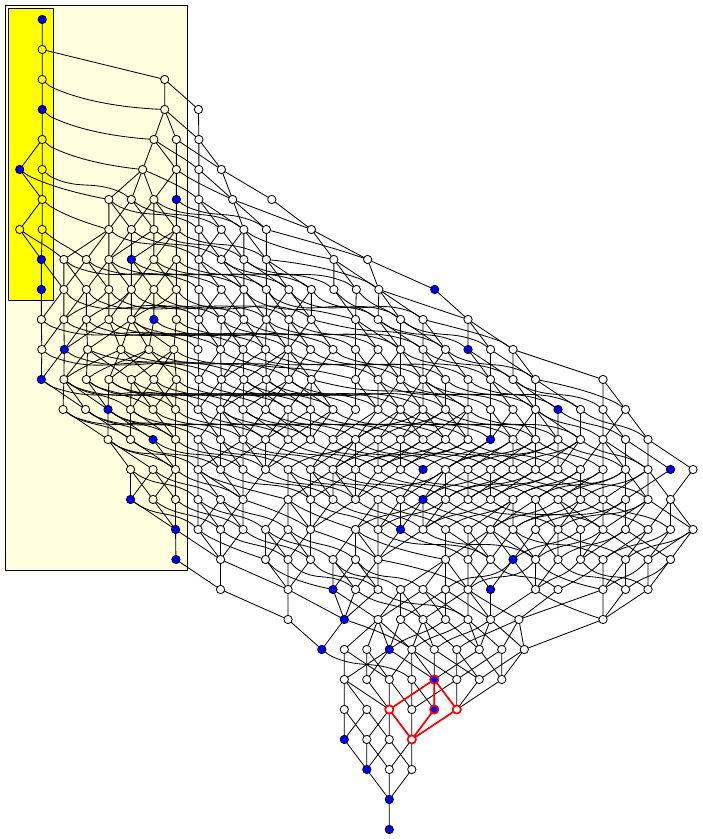}
\caption{Hasse diagram of $\Cong(\Ptw{3,2})$, with sublattices corresponding to $\Cong(\Ptw{3,1})$ and 
$\Cong(\Ptw{3,0})$ highlighted.  Vertices corresponding to Rees congruences are coloured blue,
and a copy of the 5-element diamond is higlighted in red.}
\label{fig:CongP32}
\end{center}
\end{figure}

\section{The (non-)partition monoids \boldmath{$\Ptw{1}$} and \boldmath{$\Ptw{1,d}$}}
\label{sec:n1}

In this final section we consider an interesting kind of degeneracy that arises by considering small values of $n$.

When $n=0$, the partition monoid consists of the empty partition only, and hence is trivial.
It then follows that $\Ptw{0}\cong\N$, and its congruence lattice is completely described by \eqref{eq:th1th2}.

When $n=1$ there are precisely two partitions, namely
$\begin{partn}{1} {\ \! 1\ \!}\\ {\ \! 1\ \!}\end{partn}$ and
$\begin{partn}{1} {\ \! 1\ \!}\\ \hhline{-} {\ \! 1\ \!}\end{partn}$, and $\P_{1}$ is isomorphic to $(\{0,1\},\times)$, the two-element semilattice.
 It follows that $\Ptw{1}$ is isomorphic to $\N\times\{0,1\}$ under the multiplication
 \[
 (i,q)\multw(j,r)=(i+j+\de_{0q}\de_{0r},qr),
 \]
where $\de$ is the Kronecker delta.  Theorem \ref{thm:main} remains valid for $n=1$, even though many congruences become redundant.  For one thing, there are no rows $q\geq2$, and so no $N$-symbols, and no exceptional congruences.  Additionally, since $\wh\al=\wh\be$ for all $\al,\be\in\P_1$, it follows that certain symbols play the same role:  $\muup\equiv\mudown\equiv\De$ and $\lam\equiv\rho\equiv R$, and there are no unmatched~$\mu$s.  Thus, C-matrices have labels from $\{\De,\mu,R\}$, and only items \ref{C1}, \ref{C2} and \ref{C8} from Definition \ref{de:cg} are needed to specify the congruence $\cg(\Th,M)$.

Turning to the finite monoids $\Ptw{1,d}$, the case $d=0$ is trivial, with the congruence lattice a three-element chain. So, let us assume that $d>0$. 
Here there are only three families of fC-matrix:
\begin{align*}
&\Cmatsetup\begin{array}{r|c|c|c|c|c|c|c|c|c|c|}\hhline{~|-|-|-|-|-|-|-|-|-|-|}
\renewcommand{\arraystretch}{2}
\text{\scriptsize 1} &\cellcolor{delcol}\Delta &\cellcolor{delcol}\dots & \cellcolor{delcol}\Delta & 
\cellcolor{delcol}\Delta & \cellcolor{delcol}\dots & \cellcolor{delcol}\Delta &
\cellcolor{Rcol} R& \cellcolor{Rcol}\dots&\cellcolor{Rcol} R
\\ \hhline{~|-|-|-|-|-|-|-|-|-|-|}
\text{\scriptsize 0}&\cellcolor{delcol}\Delta &\cellcolor{delcol}\dots &\cellcolor{delcol}\Delta & 
\cellcolor{Rcol} R &\cellcolor{Rcol}\dots&\cellcolor{Rcol} R&\cellcolor{Rcol} R& \cellcolor{Rcol}\dots&\cellcolor{Rcol} R
\\ \hhline{~|-|-|-|-|-|-|-|-|-|-|}
\multicolumn{1}{c}{}&\multicolumn{1}{c}{}&\multicolumn{1}{c}{}&\multicolumn{1}{c}{}&\multicolumn{1}{c}{\text{\scriptsize $i$}}&\multicolumn{1}{c}{}&\multicolumn{1}{c}{}&\multicolumn{1}{c}{\text{\scriptsize $j$}}&\multicolumn{1}{c}{}&\multicolumn{1}{c}{}
\end{array}
&&\hspace{-1.8cm}\text{for $0\leq i\leq j\leq d+1$,}\\[2mm]
&\Cmatsetup\begin{array}{r|c|c|c|c|c|c|c|c|c|c|c|c|}\hhline{~|-|-|-|-|-|-|-|-|-|-|-|}
\renewcommand{\arraystretch}{2}
\text{\scriptsize 1} &\cellcolor{delcol}\Delta &\cellcolor{delcol}\dots & \cellcolor{delcol}\Delta &\cellcolor{delcol}\Delta & \cellcolor{mucol} \mu & \cellcolor{mucol} \dots &\cellcolor{mucol} \mu &
\cellcolor{mucol} \mu &
\cellcolor{Rcol} R&\cellcolor{Rcol}\dots& \cellcolor{Rcol} R\\ \hhline{~|-|-|-|-|-|-|-|-|-|-|-|}
\text{\scriptsize 0}&\cellcolor{delcol}\Delta &\cellcolor{delcol}\dots &\cellcolor{delcol}\Delta & 
\cellcolor{mucol}\mu & \cellcolor{mucol} \mu & \cellcolor{mucol} \dots &\cellcolor{mucol} \mu &
\cellcolor{Rcol} R &\cellcolor{Rcol} R&\cellcolor{Rcol}\dots& \cellcolor{Rcol} R\\ \hhline{~|-|-|-|-|-|-|-|-|-|-|-|}
\multicolumn{1}{c}{}&\multicolumn{1}{c}{}&\multicolumn{1}{c}{}&\multicolumn{1}{c}{}&\multicolumn{1}{c}{\text{\scriptsize $i$}}&\multicolumn{1}{c}{}&\multicolumn{1}{c}{}&\multicolumn{1}{c}{}&\multicolumn{1}{c}{\text{\scriptsize $j$}}&\multicolumn{1}{c}{}&\multicolumn{1}{c}{}
\end{array}
&&\hspace{-1.8cm}\text{for $0\leq i< j\leq d$,}\\[2mm]
&\Cmatsetup\begin{array}{r|c|c|c|c|c|c|c|c|c|c|c|}\hhline{~|-|-|-|-|-|-|-|-|-|-|-|}
\renewcommand{\arraystretch}{2}
\text{\scriptsize 1} &\cellcolor{delcol}\Delta &\cellcolor{delcol}\dots & \cellcolor{delcol}\Delta  & \cellcolor{delcol}\Delta &
\cellcolor{delcol}\Delta & \cellcolor{delcol}\dots & \cellcolor{delcol}\Delta &
\cellcolor{mucol} \mu & \cellcolor{Rcol} R& \cellcolor{Rcol}\dots & \cellcolor{Rcol} R 
\\ \hhline{~|-|-|-|-|-|-|-|-|-|-|-|}
\text{\scriptsize 0}&\cellcolor{delcol}\Delta &\cellcolor{delcol}\dots &\cellcolor{delcol}\Delta  
&\cellcolor{mucol}\mu&
\cellcolor{Rcol} R &\cellcolor{Rcol}\dots&\cellcolor{Rcol} R&\cellcolor{Rcol} R&\cellcolor{Rcol} R& \cellcolor{Rcol}\dots & \cellcolor{Rcol} R
\\ \hhline{~|-|-|-|-|-|-|-|-|-|-|-|}
\multicolumn{1}{c}{}&\multicolumn{1}{c}{}&\multicolumn{1}{c}{}&\multicolumn{1}{c}{}&\multicolumn{1}{c}{}&\multicolumn{1}{c}{\text{\scriptsize $i$}}&\multicolumn{1}{c}{}&\multicolumn{1}{c}{}&\multicolumn{1}{c}{}&\multicolumn{1}{c}{\text{\scriptsize $j$}}
\end{array}
&&\hspace{-1.8cm}\text{for $1\leq i< j-1\leq d$.}
\end{align*}
Each such fC-matrix leads to a (unique) congruence, and we denote the three families of congruences by $R_{ij}$ ($0\leq i\leq j\leq d+1$), $\si_{ij}$ ($0\leq i< j\leq d$) and $\tau_{ij}$ ($1\leq i< j-1\leq d$), respectively.  The inclusion relation among these congruences takes on a particularly simple form, and the lattice $\Cong(\Ptw{1,d})$ has a neat structure; see Figure \ref{fig:CongP14} for $d=4$.  

We remark that it is apparent from Figure \ref{fig:CongP14} that $\Cong(\Ptw{1,d})$ contains many five-element diamond sublattices, which means that this lattice is not distributive.  Although it is less obvious, the lattices $\Cong(\Ptw{n,d})$ also contain diamonds for $n\geq2$ and $d\geq1$, though not for $d=0$; for example, Figure~\ref{fig:CongP32} indicates a diamond sublattice of $\Cong(\Ptw{3,2})$ in red.  
Distributivity, modularity and other properties of the lattices $\Cong(\Ptw{n})$ and $\Cong(\Ptw{n,d})$
will be one of the main topics of the forthcoming article \cite{ERtwisted2}.

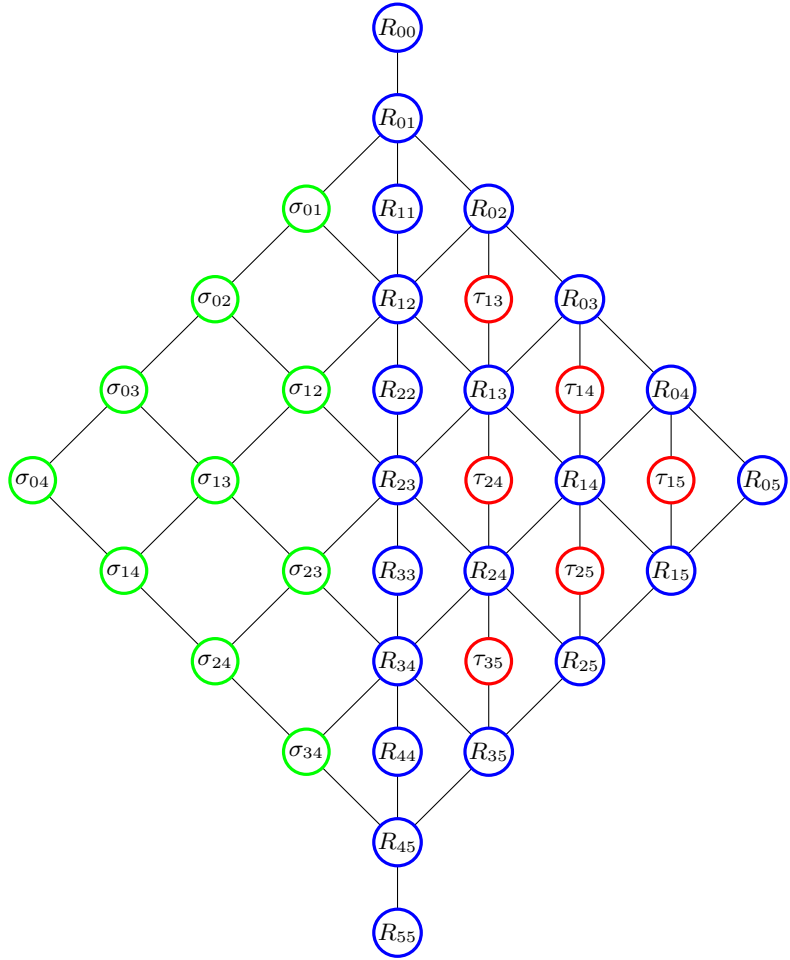
\begin{figure}[ht]
\begin{center}
\begin{tikzpicture}[scale=1.2]
\foreach \x/\y in {0/0,1/2,2/3,3/4} {\draw(\x,\y)--(\x,10-\y);}
\foreach \x in {0,1,2,3,4} {\draw(0-\x,1+\x)--(4-\x,5+\x);}
\foreach \x in {0,1,2,3,4} {\draw(0+\x,1+\x)--(-4+\x,5+\x);}
\begin{scope}[minimum size=6mm,inner sep=0.5pt, outer sep=1pt]
\foreach \L/\x/\y in {55/0/0, 45/0/1, 44/0/2, 34/0/3, 33/0/4, 23/0/5, 22/0/6, 12/0/7, 11/0/8, 01/0/9, 00/0/10, 35/1/2, 24/1/4, 13/1/6, 02/1/8, 25/2/3, 14/2/5, 03/2/7, 15/3/4, 04/3/6, 05/4/5} {\node (R\x\y) at (\x,\y) [draw=blue,fill=white,circle,line width=1.2pt] {\footnotesize $R_{\L}$};}
\foreach \L/\x/\y in {35/1/3, 24/1/5, 13/1/7, 25/2/4, 14/2/6, 15/3/5} {\node (R\x\y) at (\x,\y) [draw=red,fill=white,circle,line width=1.2pt] {\footnotesize $\tau_{\L}$};}
\foreach \L/\x/\y in {34/1/2, 23/1/4, 12/1/6, 01/1/8, 24/2/3, 13/2/5, 02/2/7, 14/3/4, 03/3/6, 04/4/5} {\node (R\x\y) at (-\x,\y) [draw=green,fill=white,circle,line width=1.2pt] {\footnotesize $\si_{\L}$};}
\end{scope}
\end{tikzpicture}
\caption{The Hasse diagram of $\Cong(\Ptw{1,4})$.}
\label{fig:CongP14}
\end{center}
\end{figure}

\footnotesize
\def\bibspacing{-1.1pt}
\bibliography{biblio}
\bibliographystyle{abbrv}

\end{document}